\documentclass{article}
\usepackage[utf8]{inputenc}
\usepackage{babel}
\usepackage{xcolor}  
\usepackage{amsmath}  
\usepackage{graphicx}   
\usepackage{amssymb}  
\usepackage{amsthm}
\usepackage{hyperref}
\usepackage{tkz-tab}
\usepackage{cases}
\usepackage[utf8]{inputenc}
\usepackage{geometry, mathtools, amsfonts, amsthm, thmtools, tabularx, appendix, bbm, titling, bm, cancel}
\usepackage{pgfplots}
\pgfplotsset{compat=1.16}
\makeatletter
\DeclareRobustCommand{\ccong}{\mathrel{\mathpalette\@verequiv\sim}}
\newcommand{\@verequiv}[2]{%
  \lower.5\p@\vbox{
    \lineskiplimit\maxdimen
    \lineskip-.5\p@
    \ialign{%
      $\m@th#1\hfil##\hfil$\crcr
      #2\crcr
      \equiv\crcr
    }%
  }%
}
\makeatother

\newtheorem{theorem}{Theorem}[section]

\newtheorem{proposition}[theorem]{Proposition}
\newtheorem{lemma}[theorem]{Lemma}
\newtheorem{corollary}[theorem]{Corollary}
\newtheorem{remark}[theorem]{Remark}

\title{Stability of partially congested travelling wave solutions for the extended Aw-Rascle system}

\date{\today}
\author{Émile Deléage\thanks{Aix Marseille Univ., CNRS, Centrale Marseille,  I2M UMR CNRS 7373, 13000 Marseille, France and Univ. Grenoble Alpes, CNRS, INRAE, IRD, Grenoble INP, IGE, 38000 Grenoble, France; emile.deleage@univ-amu.fr}, Muhammed Ali Mehmood\thanks{Imperial College London, London, United Kingdom; muhammed.mehmood21@imperial.ac.uk}}
\begin{document}

\maketitle

\begin{abstract}
We prove the non-linear stability of a class of travelling-wave solutions to the extended Aw-Rascle system with a singular offset function, which is formally equivalent to the compressible pressureless Navier-Stokes system with a singular viscosity. These solutions encode the effect of congestion by connecting a congested left state to an uncongested right state, and may also be viewed as approximations of solutions to the ‘hard-congestion model'. By using carefully weighted energy estimates we are able to prove the non-linear stability of viscous shock waves to the Aw-Rascle system under a small zero integral perturbation, which in particular extends previous results that do not handle the case where the viscosity is singular. 
\end{abstract}

	

\section{Introduction}

\subsection{Presentation of the model}

We study the following generalised Aw-Rascle system on the real line:
\begin{equation} \label{ARoriginal}
    \left\{
    \begin{aligned}
    &\partial_t \rho +\partial_y(\rho u)=0,\\
    &\partial_t(\rho w)+\partial_y (\rho u w)=0.
    \end{aligned}
    \right.
\end{equation}
Here, the quantities $\rho$ represents the density while $u$ and $w$ respectively refer to the actual and desired velocities of agents. 
 This system was originally coined in 2000 by Aw and Rascle \cite{aw2000resurrection} and has popularly been used to model the evolution of a system of interacting agents, such as the flow of vehicular traffic \cite{sheng_concentration_2022, HCL, HCLMehmood} or crowd dynamics \cite{aceves2023pedestrian}. The standard Aw-Rascle system is complemented by the relation $w = u + P(\rho)$, where $P=P(\rho)$ is known as the `offset' function.
In this paper, we consider the case where $w = u+\partial_y p_\epsilon(\rho)$, with $p_\epsilon$ being a singular function of the density depending on a parameter $\epsilon$. More precisely, for $\epsilon>0$ fixed, we take
\begin{equation} \label{offsetINTRO}
    w= u+ \partial_{y}p_\epsilon(\rho) =u+ \partial_{y} \left(\epsilon \left(\frac{\rho}{1-\rho}\right)^\gamma \right),\qquad \gamma \geq 1.
\end{equation}
Note that the singularity as $\rho$ approaches $1$ in \eqref{offsetINTRO} is physically significant, since it implies that the density of agents within the system may not exceed a maximal packing constraint, i.e. $\rho < 1$. In one-dimension it is interesting to note that the system may be formally rewritten as 
\begin{equation}
\label{pressurelessns}
    \left\{
    \begin{aligned}
        &\partial_t \rho +\partial_y(\rho u)=0,\\
        &\partial_t(\rho u)+\partial_y(\rho u^2)-\partial_y(\rho^2p_\epsilon'(\rho)\partial_y u)=0,
    \end{aligned}
    \right.
\end{equation}
which resembles the one-dimensional compressible pressureless Navier-Stokes equations with a singular degenerate diffusion coefficient $\lambda_\epsilon := \rho^2p_\epsilon'(\rho)$.

The system \eqref{pressurelessns} with \eqref{offsetINTRO} was rigorously derived by Lefebvre-Lepot and Maury in \cite{lefebvre_maury} from a microscopic lubrication model, and describes the evolution of particles suspended in a viscous fluid that interact with each other via a lubricating force. The diffusion coefficient obtained by Lefebvre-Lepot and Maury is $\epsilon(1-\rho)^{-1}$, where $\epsilon$ is the viscosity of the interstitial fluid. The case $\gamma = 1$ has also proven to be physically relevant in applications. Indeed, the diffusion coefficient of dense granular suspensions was experimentally measured to behave like $(\phi_c-\phi)^{-2}$ as the solid volume fraction $\phi$ (which corresponds to the non-dimensionalised density) approaches the maximal volume fraction $\phi_c$ (see for instance \cite{guazzelli2018rheology} for a review). 

Let us now go back to the momentum equation of \eqref{pressurelessns}. In the region where $\rho$ is far from 1, the viscosity $\lambda_\epsilon$ vanishes uniformly as $\epsilon$ goes to 0 and \eqref{pressurelessns} formally degenerates towards the pressureless gas system \cite{berthelin2002existence,boudin2000solution}. On the other hand, in the congested region, $\rho$ is very close to 1, and the singularity $(1-\rho)^{-(\gamma+1)}$ compensates the degeneracy in $\epsilon$. The limit $\epsilon \to 0$ in model \eqref{pressurelessns}-\eqref{offsetINTRO} (see \cite{HCL}) is then viewed as a transition towards a granular suspension model, where the interacting force is governed by the contact between the solid grains. The Aw-Rascle system with this choice of offset function was investigated by the authors of \cite{HCL} on a one-dimensional periodic domain, where the global well-posedness for fixed $\epsilon$ was studied, as well as the limit $\epsilon \to 0$, which is known as the ‘hard-congestion limit'. The authors demonstrated that up to a subsequence, solutions of \eqref{ARoriginal} with \eqref{offsetINTRO} converge towards weak solutions of the ‘hard-congestion model' (see also \cite{HCLMehmood,bresch2014singular}):
\begin{subequations} \label{HCL}
\newcommand{\mystrut}{\vphantom{\pder{}{}}}
\begin{numcases}{}
    \partial_{t} \rho + \partial_{y} (\rho u) = 0, \label{HCL-L1}   \\[1ex]
    \partial_{t} (\rho  u + \partial_{y}\pi) + \partial_{y}((\rho u + \partial_{y}\pi)u)  = 0,  \label{HCL-L2} \\[1ex]
    0 \le \rho \le 1,~ (1-\rho)\pi = 0,~ \pi \ge 0. \label{HCL-L3}
\end{numcases}
\end{subequations}   In this sense, the original system \eqref{ARoriginal} may be referred to as an approximation of \eqref{HCL}. System \eqref{HCL} is an example of a free-congested system, where the free phase refers to the region where $\rho < 1$ and the congested region is where $\rho = 1$. The potential $\pi$ is obtained in the limit and can be viewed as a Lagrange multiplier associated with the incompressibility constraint $\partial_{x}u=0$ which holds in the congested region. The existence of strong solutions to the system \eqref{HCL} is still not known and there are also no results on the stability of non-trivial solutions (i.e. where $\pi$ is not identically $0$). Note however that the existence of weak and measure-valued (duality) solutions to this model was recently obtained on the real line in \cite{duality2024HCL}. We also refer to \cite{chaudhuri_analysis_2022,locallycosntrainedAR,resonanceAr} and the references therein for further examples of the theoretical and numerical analysis pertaining to the Aw-Rascle system.

In the present work, we are concerned with the stability of a specific class of solutions to the system \eqref{ARoriginal}, namely the solutions $(\rho, u)$ which are travelling waves that connect a congested left state ($\rho = 1$) to a non-congested right state ($\rho < 1$). Since systems \eqref{ARoriginal} and \eqref{pressurelessns} are equivalent (for sufficiently regular solutions), our task is closely related to the stability of travelling wave solutions to the compressible Navier-Stokes system, which has been studied by Dalibard and Perrin in \cite{dalibard2019existence}. There, the authors studied system \eqref{pressurelessns} with the addition of a pressure and a constant viscosity coefficient  $\mu>0$. A similar study was also carried out in \cite{vasseur2016nonlinear} by Vasseur and Yao. Both of these works make use of a new formulation of the system, which is obtained by introducing the ‘effective velocity' $w$ and rewriting the system in terms of $(w,v)$. The parabolic equation satisfied by $v$ and the transport structure for $w$ then allows for desirable energy estimates, which are carried out with the help of integrated variables (see Equation \eqref{integratedvariables} below). Rewriting the system using the effective velocity has also proven to be advantageous when investigating the existence and uniqueness of weak/strong solutions to the compressible Navier-Stokes system with a density-dependent viscosity of the form $\mu(\rho) = \rho^{\alpha}$ for $\alpha > 0$ (see \cite{Burtea_2020,Constantin_2020} for example).

 To the best of our knowledge, there are no known results concerning the stability of shock waves where the viscosity coefficient is singular. The interest in such a result is twofold. Firstly, the stability of strong solutions to \eqref{ARoriginal} is significant due to the equivalence with the compressible Navier-Stokes system \eqref{pressurelessns}, for which such a result does not exist in the literature as far as we know. On the other hand, a stability result for partially congested solutions would imply that the system \eqref{HCL} is also expected to be stable, which provides additional validity for the model \eqref{HCL} and further verifies the need for a rigid well-posedness theory. Note that in our case the presence of a singular, degenerate viscosity coefficient and the lack of pressure prevents us from using the arguments of \cite{dalibard2019existence,vasseur2016nonlinear} to obtain the estimates required to prove the global existence and stability of travelling wave solutions to \eqref{ARoriginal}. Nonetheless, we demonstrate in this paper that through a careful choice of weighted energy estimates and the identification of good unknowns taking congestion into account, we can obtain results analogous to those of \cite{dalibard2019existence} for the Aw-Rascle system \eqref{ARoriginal} with the singular offset function \eqref{offsetINTRO}.

 Let us now give an outline of the paper. In the next subsection we detail our main results, which are the existence of travelling wave solutions to System \eqref{ARoriginal}-\eqref{offsetINTRO},  and the nonlinear stability of these solutions. Then in Section \ref{sec:prelim} we introduce basic properties of travelling wave solutions, an integrated formulation of the system and some useful preliminary estimates. The bulk of our work goes into Section \ref{sectionGWP}, where we obtain the well-posedness of the integrated system. Lastly, we prove the equivalence between the integrated system and the original system in Section \ref{sec:ASYMstab}. 

\subsection{Main results} \label{sec:mainresults}We now mention our main results. Let the Lagrangian mass coordinate $x$ be such that $\mathrm{d}x=\rho \mathrm{d}y-\rho u\mathrm{d}t$, and $v:=1/\rho$ the specific volume. Then  \eqref{pressurelessns}
 becomes
\begin{equation}
\label{lagrangian}
    \left\{
    \begin{aligned}
        &\partial_t v -\partial_x u=0,\\
        &\partial_t u -\partial_x(\phi_\epsilon(v)\partial_x u)=0,
    \end{aligned}
    \right.
\end{equation}
where \begin{equation}
\label{diffusioncoefficient}
    \phi_\epsilon(v) := \frac{p_\epsilon'(1/v)}{v^3} =\frac{\epsilon\gamma}{v(v-1)^{\gamma+1}}, \qquad\mathrm{such}\,\,\,\mathrm{that}\qquad \partial_y p_\epsilon(\rho) = -\phi_\epsilon(v)\partial_x v.
\end{equation}
In these coordinates, it follows that 
\begin{equation*}
    w = u +\partial_y p_\epsilon(\rho) = u - \phi_\epsilon(v) \partial_x v
\end{equation*}
is constant, i.e. solves $\partial_t w =0$.
\begin{remark}
All of our results remain true if we replace \eqref{offsetINTRO} by  
\begin{equation*}
    w = u+ \partial_y \tilde{p}_\epsilon(\rho) = u +\partial_y\left(\frac{\epsilon f(\rho)}{(1-\rho)^{\gamma}}\right),\qquad \gamma \geq 1,
\end{equation*}
where $f$ is smooth on $\mathbb{R}_+^*$ and such that $\gamma f(\rho)+(1-\rho)f'(\rho)>0$, i.e. such that $\tilde{p}_\epsilon'(\rho)>0$. In this case, one obtains that 
\begin{equation}
\label{generalisedviscosity}
\phi_\epsilon(v) = \epsilon v^{\gamma-3}\frac{(v-1)f'(1/v) + \gamma f(1/v)}{(v-1)^{\gamma+1}}.    
\end{equation}
In order to improve readability and without loss of generality, we will stick to the case $f(\rho) = \rho^{\gamma}$, for which the coefficient $\phi_\epsilon$ can be written in the more compact form \eqref{diffusioncoefficient}. This is also the form considered in \cite{HCL}. All computations are similar (but heavier) in the general case \eqref{generalisedviscosity}.
\end{remark}

Our first lemma establishes the existence of travelling wave solutions, and gives a quantitative description of the profile.
\begin{proposition}
    \label{travelingsolution}
    Let $1=v_- < v_{+}$ and $ u_{-} > u_{+}$ be real numbers. Then there exists a unique (up to a translation) travelling wave solution $(\mathbf{u},\mathbf{v})(t,x) = (u_{\epsilon}, v_{\epsilon})(\xi)$ of \eqref{lagrangian}, complemented with the far field condition $(\mathbf{u},\mathbf{v})\rightarrow(u_\pm,v_\pm)$ as $\xi \to\pm\infty$, where $\xi:= x-st$ and $s$ is the shock speed which satisfies
    \begin{equation}
        s = \frac{u_{-}-u_{+}}{v_{+}-1}.
    \end{equation}
    The solution is a smooth monotone increasing function connecting 1 (at $-\infty$) to $v_+$ (at $+\infty$). 
    
    By setting $v_\epsilon(0) =(1+v_+)/2$, one then has the following estimates:
    \begin{itemize}
        \item In the congested region $(\xi<0)$,
    \begin{equation}
    \label{doublebound}
        1 + \left(B-\frac{A_0}{\epsilon}\xi\right)^{-1/\gamma} \leq v_\epsilon(\xi) \leq 1+\left(B-\frac{A_1}{\epsilon}\xi\right)^{-1/\gamma},
    \end{equation}
    where
    \begin{equation*}
        A_0:= \frac{s(v_+-1)(v_+ +1)}{2}, \qquad  A_1:= \frac{s(v_+-1)}{2}, \qquad \mathrm{and} \qquad B:= \left(\frac{2}{v_+-1}\right)^\gamma.
    \end{equation*}
    \item In the free region $\xi>0$,
    \begin{equation}
    \label{doublebound2}
        v_+-\frac{v_+-1}{2}\exp\left(-\frac{A_2}{\epsilon}\xi\right) \leq v_\epsilon(\xi) \leq v_+-\frac{v_+-1}{2}\exp\left(-\frac{A_3}{\epsilon}\xi\right),
    \end{equation}
    where 
    \begin{equation*}
        A_2:= \frac{s(v_++1)(v_+-1)^{\gamma+1}}{2^{\gamma+2}\gamma}\qquad \mathrm{and} \qquad A_3 :=  \frac{s v_+(v_+-1)^{\gamma+1}}{\gamma}.
    \end{equation*}
    \end{itemize}
    It follows that $v_\epsilon$ converges to the shock wave $v(\xi):= \mathbf{1}_{\xi<0}+v_+\mathbf{1}_{\xi>0}$ as $\epsilon\rightarrow 0$, and it holds
    \begin{equation}
        \label{asymptoticbound}
        v_\epsilon(\xi)\leq \left[1+\left(B-\frac{A_1}{\epsilon}\xi\right)^{-1/\gamma}\right]\mathbf{1}_{\xi<0}+v_+\mathbf{1}_{\xi\geq 0}.
    \end{equation}
\end{proposition}
\color{black}
With this in hand, we dedicate the rest of our effort towards studying the stability of the profiles $(u_{\epsilon}, v_{\epsilon})$ where $\epsilon < < 1$. We first express System \eqref{lagrangian} in term of the unknowns $v$ and $w=u-\phi_\epsilon(v)\partial_x v$:
\begin{equation}  \label{ARvw}
    \left\{
    \begin{aligned}
        &\partial_t w = 0,\\
        &\partial_t v -\partial_x w - \partial_x (\phi_\epsilon(v)\partial_x v)=0.
    \end{aligned}
    \right.
\end{equation}
In order to obtain effective energy estimates, we take inspiration from \cite{dalibard2019existence, vasseur2016nonlinear} and choose to re-write this system in terms of the integrated variables. Suppose we have initial data $(w_{0}, v_{0})$ such that $(w_{0} - (w_{\epsilon})(0), v_{0} - (v_{\epsilon})(0)) \in (L^{1}_{0}(\mathbb{R}) \cap L^{\infty}(\mathbb{R}))^{2}$ where $L^{1}_{0}(\mathbb{R})$ is the subset of $L^{1}(\mathbb{R})$ consisting of zero mean functions. Then we define the integrated initial data as
\begin{equation*}
    W_{0}(x) := \int_{-\infty}^{x} (w_{0}(z) - w_{\epsilon}(z))~dz, \quad V_{0}(x) := \int_{-\infty}^{x} (v_{0}(z) - v_{\epsilon}(z))~dz.
\end{equation*}
Supposing that $(w-w_{\epsilon}, v- v_{\epsilon})(t) \in L^{1}_{0}(\mathbb{R})$ holds true for positive times, we may also define the integrated variables
\begin{equation}
    \label{integratedvariables}
     W(t,x) := \int_{-\infty}^{x} (w(t,z) - w_{\epsilon}(t,z))~dz, \quad V(t,x) := \int_{-\infty}^{x} (v(t,z) - v_{\epsilon}(t,z))~dz.
\end{equation}  Integrating \eqref{ARvw} between $-\infty$ and $x$ formally, we see that $(W,V)$ solves 
\begin{equation} \label{INTEGRATEDsystem}
        \left\{
    \begin{aligned}
        &\partial_t W = 0,\\
        &\partial_t V -\partial_x W - \phi_\epsilon(v)\partial_x v+\phi_\epsilon(v_\epsilon)\partial_xv_\epsilon=0, \\ 
        &(W,V)(0, \cdot) = (W_{0}, V_{0}).
    \end{aligned}
    \right.
\end{equation}

Since $W$ is constant in time, we will denote it by its initial value $W_0$ from now on. The following result pertains to this system.
\begin{theorem}[Existence of a strong solution to the integrated system] \label{GWPintegrated} Let $T>0$ and $\gamma \ge 1$. Assume that $(W_{0}, V_{0}) \in (H^{2}(\mathbb{R}))^{2}$ and define 
    \begin{equation}
        \eta_0 := \frac{\partial_x V_0}{v_\epsilon|_{t=0}-1}.
    \end{equation}
    Suppose that $\eta_0 \in H^1(\mathbb{R})$, $\sqrt{x}\partial_x^k W_0 \in L^2(\mathbb{R}_+)$ for $k=0,1,2$, and that the following estimate holds:
    \begin{equation}
    \begin{aligned}
        \label{THMIntExIDassumption}
         \sum_{k=0}^2 \left(c_k \epsilon^{2k} E_{k}(0; V_{2})+\epsilon^{2k-1}\|\sqrt{x}\partial_x^k W_0\|_{L^2(\mathbb{R}_+)}^2\right) + \|W_0\|_{L^2_{(\mathbb{R})}}^2 + \left(\frac{T}{\epsilon}\right)^{1/\gamma}\left(\epsilon^2\|\partial_x W_0\|^2_{L^2_{(\mathbb{R})}}+\epsilon^{4} \|\partial_{x}^{2}W_{0}\|_{L^{2}_{(\mathbb{R})}}^{2}\right)  \leq \delta_0 \epsilon^3
    \end{aligned}
    \end{equation}
    \color{black}
    for some constants $\delta_{0}, c_0, c_1, c_2$ depending only on $s,\gamma, v_+$. Then there exists a unique solution $(W_0,V)$ to \eqref{INTEGRATEDsystem} on $(0,T)$ such that
    \begin{align*}
        &V \in C([0,T]; H^{2}(\mathbb{R})) \cap L^{2}(0,T; H^{3}(\mathbb{R})). 
    \end{align*} Moreover, there exists a constant $C = C(s,v_{+}, \gamma, \delta_{0})$ such that
    \begin{equation} \label{GWPintegratedfinalEST}
    \begin{aligned}
        \sup_{t \in [0,T]}  & \left(  \int_{\mathbb{R}} V^{2}~dx + \int_{0}^{t} \int_{\mathbb{R}} \phi_{\epsilon}(v_{\epsilon})|\partial_{x}V|^{2}~dxds \right) \\[1ex] &+ \sup_{t \in [0,T]} \left( \sum_{k=0}^{1}\epsilon^{2k+2} \left[  \int_{\mathbb{R}}  \left|\partial_{x}^{k}\left(\frac{\partial_{x}V}{v_{\epsilon}-1} \right)\right|^{2} ~dx + \int_{0}^{t} \int_{\mathbb{R}} \phi_{\epsilon}(v_{\epsilon})\left|\partial_{x}^{k+1}\left(\frac{\partial_{x}V}{v_{\epsilon}-1} \right)\right|^{2} ~dxds  \right]  \right) \le C \epsilon^{3}.
       \end{aligned}
    \end{equation}
\end{theorem}

\begin{remark}
    It follows from Proposition \ref{travelingsolution} that the coefficient of diffusion $\phi_\epsilon(v_\epsilon)$ (see \eqref{diffusioncoefficient} above) is singular and tends to $+\infty$ as $x$ tends to $-\infty$. As a consequence, we cannot close the estimates on $V$ in the classical Sobolev spaces $H^s$. This is why we work with the weighted quantity $\eta:=\partial_x V/(v_\epsilon-1)$ instead (see \eqref{defeta}), which yields better estimates. The drawback of this method is that the time-independent quantity $W_0$ cannot be bounded uniformly in time in such weighted spaces. Hence we only obtain local in time well-posedness when $W_0\neq 0$, with a time of existence $T$ proportional to $1/\|\partial_x W_0\|_{H^1}^{2\gamma}$. 
\end{remark}
\begin{remark} 
   The assumption $\sqrt{x} \partial_{x}^{k}W_0 \in L^{2}(\mathbb{R}_{+})$ for $k=0,1,2$ is classical for this kind of system and was already used in \cite{dalibard2021local}. However, it is possible to remove this assumption and only assume that $W_0\in H^2(\mathbb{R})$. One then obtains a shorter time of existence, proportional to  $1/\|W_0\|_{H^2}^2$ (see Remark \ref{REMARK_improvedWP} below).
\end{remark}

\color{black}
\begin{remark} \label{rmkgwpC}
     Using $v_{\epsilon} < v_{+}$ and the regularity of $v_\epsilon$, the bound \eqref{GWPintegratedfinalEST} implies that $V \in L^{\infty}(0,T; H^{2}(\mathbb{R})) \cap L^{2}(0,T; H^{3}(\mathbb{R}))$. Furthermore, by using the smallness assumption \eqref{THMIntExIDassumption},  we will show the lower bound $v>1$ for every $t,x$. In other words, the perturbation does not reach the congested state (see Remark \ref{stayuncongested}).
\end{remark}
Under the same assumptions, we also prove a stability result.
\begin{theorem}[Nonlinear stability of travelling wave solutions] \label{THMstability} Let $T>0$, $\gamma \ge 1$. Assume that the initial data $(u_{0}, v_{0})$ is such that \begin{equation} \label{THMstabilityIDassumptions}
    u_{0} - (u_{\epsilon})_{t=0} \in W^{1,1}_{0}(\mathbb{R}) \cap H^{1}(\mathbb{R}), \quad {\frac{\partial_x[u_0-(u_\epsilon)_{t=0}]}{v_\epsilon-1}\in L^2(\mathbb{R})}, \quad v_{0} - (v_{\epsilon})_{t=0} \in W^{2,1}_{0}(\mathbb{R}) \cap H^{2}(\mathbb{R}),
\end{equation} and the associated integrated initial data $(W_{0}, V_{0})$ satisfies \eqref{THMIntExIDassumption}. Then there exists a unique global solution $(u,v)$ to \eqref{ARoriginal} on $[0,T]$ which satisfies
\begin{align*}
    &u-u_{\epsilon} \in C([0,T]; H^{1}(\mathbb{R}) \cap L^{1}_{0}(\mathbb{R})), \\[1ex]
    &v-v_{\epsilon} \in C([0,T]; H^{1}(\mathbb{R}) \cap L^{1}_{0}(\mathbb{R})) \cap L^{2}(0,T; H^{2}(\mathbb{R})).
\end{align*}
     In particular, there exists a constant $C_{1}>0$ dependent on $\gamma, v_{+}, T$  such that \begin{equation*}
\|u-u_{\epsilon}\|_{L^{\infty}(0,T; H^{1}(\mathbb{R}))} + \|v-v_{\epsilon}\|_{L^{\infty}(0,T; H^{1}(\mathbb{R}))} + \|v-v_{\epsilon}\|_{L^{2}(0,T; H^{2}(\mathbb{R}))} \le C_{1}.
     \end{equation*} There also exists a constant $C_{2} > 0$ dependent on $\gamma, v_{+}, T, \epsilon$  such that
     \begin{equation*}
         \|u-u_{\epsilon}\|_{L^{\infty}(0,T; L^{1}(\mathbb{R}))} + \|v-v_{\epsilon}\|_{L^{\infty}(0,T; L^{1}(\mathbb{R}))} \le C_{2}.
     \end{equation*}
\end{theorem}
As a consequence of the estimates derived during the course of proving the above theorem, we are also able to assert global existence with $T = + \infty$ and long-time stability if we remove the perturbation on the desired velocity.
\begin{corollary} \label{CORstability}
If we additionally assume that \begin{equation} \label{CORassumption}
    (u - u_{\epsilon})(0) = (\partial_{x}\phi_{\epsilon}(v) - \partial_{x}\phi_{\epsilon}(v_{\epsilon}))(0),
\end{equation} i.e. that $(w - w_{\epsilon})(0) = 0$, then the solution $(u,v)$ is defined on $\mathbb{R}_{+} \times \mathbb{R}$, satisfies \eqref{GWPintegratedfinalEST} with $T=+\infty$ and additionally
 \begin{equation*}
         \sup_{x \in \mathbb{R}} \left( |v-v_{\epsilon}|(t,x) + |u-u_{\epsilon}|(t,x)      \right) \to 0 \text{ as } t \to \infty.
     \end{equation*}
\end{corollary}

\section{Properties of the travelling wave solutions} \label{sec:prelim}

\subsection{Existence and asymptotic behavior of the travelling waves}

We give here the proof of Proposition \ref{travelingsolution}.
\begin{proof}
We first prove the existence of travelling wave solutions to \eqref{lagrangian}:
\begin{equation*}
    \left\{
    \begin{aligned}
        &\partial_t v -\partial_x u=0,\\
        &\partial_t u -\partial_x(\phi_\epsilon(v)\partial_x u)=0,
    \end{aligned}
    \right.
\end{equation*}
 We look for a travelling wave solution of this system, i.e. a pair $(u_\epsilon,v_\epsilon)$ where $u_\epsilon,v_\epsilon$ are functions of the variable $\xi=x-st$, $s$ being the speed of propagation of the solution. We also suppose that $(u_\epsilon,v_\epsilon)\rightarrow (u_\pm,v_\pm)$ and that $(u_\epsilon',v_\epsilon')\rightarrow (0,0)$ as $\xi$ goes to $\pm\infty$. We first obtain
\begin{equation*}
    \left\{
    \begin{aligned}
        &-sv_\epsilon'-u_\epsilon'=0,\\
        &-s u_\epsilon'-(\phi_\epsilon(v_\epsilon)u_\epsilon')'=0.
    \end{aligned}
    \right.
\end{equation*}
Integrating this equation between $\xi$ and $+\infty$, we get
\begin{equation*}
    \left\{
    \begin{aligned}
        &sv_\epsilon+u_\epsilon=sv_++u_+,\\
        &s u_\epsilon+\phi_\epsilon(v_\epsilon)u_\epsilon'=su_+.
    \end{aligned}
    \right.
\end{equation*}
The first equation yields $u_\epsilon = sv_++u_+-sv_\epsilon$. Substituting this into the second line, we deduce that $v_\epsilon$ solves the following ODE:
\begin{equation}
    \label{ODE}
    v_\epsilon' =\frac{s(v_+-v_\epsilon)}{\phi_\epsilon(v_\epsilon)}=\frac{s(v_+-v_\epsilon)v_\epsilon(v_\epsilon-1)^{\gamma+1}}{\epsilon\gamma}.
\end{equation}
Assuming $v_+>1$ and $s>0$, we may deduce that $v_-=1$, and $u_-=s(v_+-1)+u_+$, or equivalently $s=(u_--u_+)/(v_+-1)$. The function $v_\epsilon$ is therefore an increasing function taking values in the interval $(1, v_{+})$ (see Figure \ref{fig:A}).  The Cauchy-Lipschitz theorem thus yields that $v_\epsilon$ is the unique (up to a translation) global solution of \eqref{ODE}, as stated in Proposition \ref{travelingsolution}. Note that, as $v_\epsilon$ approaches $1$, the diffusion coefficient $\phi_\epsilon(v_\epsilon)$ tends to $+\infty$.

\begin{figure}[ht]
    \centering
    \includegraphics[scale=0.6]{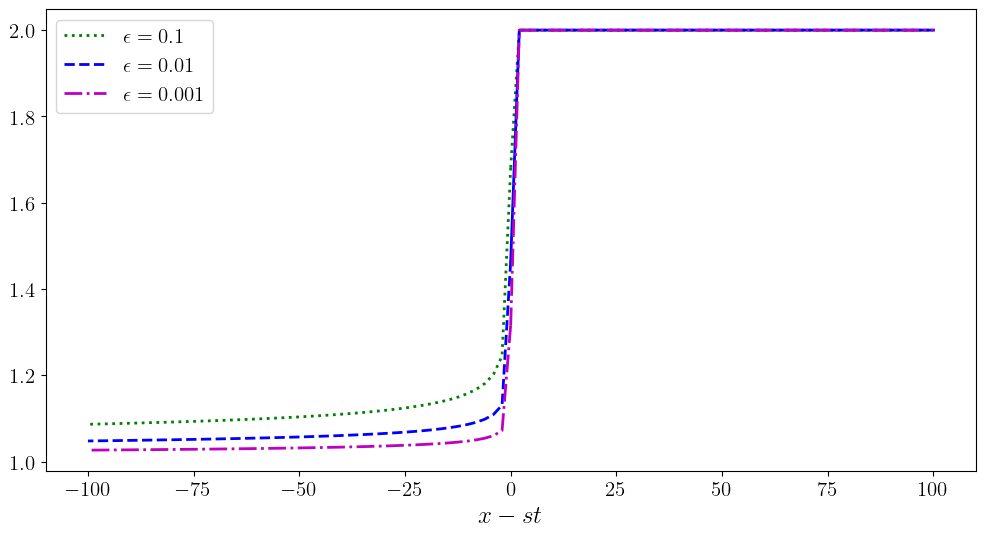}
    \caption{The profile of $v_{\epsilon}$, where we fix $v_{\epsilon}(0) = (1 + v_{+})/2$ and $v_{+} = 2, \gamma = 5$.}
    \label{fig:A}
\end{figure}

Let us now make this statement more quantitative. From now on, suppose that $v_\epsilon(0)= (1+v_+)/2$, i.e. $v_\epsilon$ is halfway between 1 and $v_+$. Let $\xi <0$. From the ODE \eqref{ODE} and the monotonicity of $v_\epsilon$, we obtain the following bound:
\begin{equation*}
   \frac{s(v_+-1)(v_\epsilon(\xi)-1)^{\gamma+1}}{2\epsilon\gamma}\leq  v_\epsilon'(\xi) \leq \frac{s(v_+-1)(v_+ +1)(v_\epsilon(\xi)-1)^{\gamma+1}}{2\epsilon\gamma}.
\end{equation*}
Dividing by $(v_\epsilon(\xi)-1)^{\gamma+1}$ and integrating between $\xi$ and 0 yields
\begin{equation*}
   - \frac{s(v_+-1)}{2\epsilon\gamma}\xi \leq -\frac{2^\gamma}{\gamma(v_+-1)^\gamma} + \frac{1}{\gamma(v_\epsilon(\xi)-1)^\gamma} \leq -\frac{s(v_+-1)(v_+ +1)}{2\epsilon\gamma}\xi,
\end{equation*}
i.e. 
\begin{equation*}
    1 + \left(B-\frac{A_0}{\epsilon}\xi\right)^{-1/\gamma} \leq v_\epsilon(\xi) \leq 1+\left(B-\frac{A_1}{\epsilon}\xi\right)^{-1/\gamma},
\end{equation*}
where
\begin{equation*}
    A_0:= \frac{s(v_+-1)(v_+ +1)}{2}, \qquad  A_1:= \frac{s(v_+-1)}{2}, \,\,\, \mathrm{and} \,\,\, B:= \left(\frac{2}{v_+-1}\right)^\gamma.
\end{equation*}
Estimate \eqref{asymptoticbound} follows. When $\xi >0$, one has similarly
\begin{equation*}
     \frac{s(v_++1)(v_+-1)^{\gamma+1}}{2^{\gamma+2}\epsilon\gamma} \leq \frac{v_\epsilon'(\xi)}{v_+-v_\epsilon(\xi)} \leq  \frac{sv_+(v_+-1)^{\gamma+1}}{\epsilon\gamma},
\end{equation*}
i.e. 
\begin{equation*}
    \frac{A_2}{\epsilon} \leq \frac{v_\epsilon'(\xi)}{v_+-v_\epsilon(\xi)} \leq \frac{A_3}{\epsilon},
    \end{equation*}
    with 
    \begin{equation*}
    A_2:= \frac{s(v_++1)(v_+-1)^{\gamma+1}}{2^{\gamma+2}\gamma} \quad \mathrm{and} \quad A_3:= \frac{sv_+(v_+-1)^{\gamma+1}}{\gamma} .
\end{equation*}
\color{black}
Integrating between 0 and $\xi$ yields 
\begin{equation*}
    \frac{A_2}{\epsilon}\xi \leq -\ln[v_+-v_\epsilon(\xi)] + \ln\left[\frac{v_+-1}{2}\right]\leq \frac{A_3}{\epsilon}\xi,
\end{equation*}
i.e.
\begin{equation*}
    \frac{v_+-1}{2}\exp\left(-\frac{A_2}{\epsilon}\xi\right)\geq v_+-v_\epsilon(\xi) \geq \frac{v_+-1}{2}\exp\left(-\frac{A_3}{\epsilon}\xi\right),
\end{equation*}
which is \color{black} the desired estimate. 

\end{proof}

\subsection{Passage to the integrated system and reformulation of the equations}

We now want to study the stability of the travelling wave solution obtained in the previous section. In order to do so, we first rewrite \eqref{lagrangian} in term of the unknowns $v$ and $w  =u-\phi_\epsilon(v)\partial_x v$:
\begin{equation} 
    \left\{
    \begin{aligned}
        &\partial_t w = 0,\\
        &\partial_t v -\partial_x w - \partial_x (\phi_\epsilon(v)\partial_x v)=0.
    \end{aligned}
    \right.
\end{equation}
Let $(w_\epsilon,v_\epsilon)$ denote the travelling wave solution. It solves the same system:
\begin{equation}
    \left\{
    \begin{aligned}
        &\partial_t w_\epsilon= 0,\\
        &\partial_t v_\epsilon -\partial_x w_\epsilon - \partial_x (\phi_\epsilon(v_\epsilon)\partial_x v_\epsilon)=0.
    \end{aligned}
    \right.
\end{equation}
Taking the difference between these two systems yields
\begin{equation}
    \label{differencesystem}
    \left\{
    \begin{aligned}
        &\partial_t(w- w_\epsilon)= 0,\\
        &\partial_t (v-v_\epsilon) -\partial_x (w-w_\epsilon) - \partial_x(\phi_\epsilon(v)\partial_x v)+\partial_x(\phi_\epsilon(v_\epsilon)\partial_x v_\epsilon)=0.
    \end{aligned}
    \right.
\end{equation}
Following \cite{dalibard2019existence,vasseur2016nonlinear}, we introduce $W$, $V$ such that
\begin{equation}
    W(t,x):=\int_{-\infty}^x(w(t,x')-w_\epsilon(t,x'))\mathrm{d}x'\,\,\,\mathrm{and}\,\,\, V(t,x):=\int_{-\infty}^x(v(t,x')-v_\epsilon (t,x'))\mathrm{d}x'.
\end{equation}
Integrating \eqref{differencesystem} between $-\infty$ and $x$ yields for $(W,V)$ the following system:
\begin{equation*}
        \left\{
    \begin{aligned}
        &\partial_tW = 0,\\
        &\partial_t V -\partial_x W - \phi_\epsilon(v)\partial_x v+\phi_\epsilon(v_\epsilon)\partial_xv_\epsilon=0.
    \end{aligned}
    \right.
\end{equation*}
Introducing $\psi_\epsilon$ such that $\psi_\epsilon'=\phi_\epsilon$, we obtain that $W=W_0$ is constant and
\begin{equation} \label{ARintegrated}    
   \partial_t V -\partial_x W_0 -\partial_x(\psi_\epsilon(v)-\psi_\epsilon(v_\epsilon))=0.
\end{equation}
We first notice that we can replace $v$ by $v_\epsilon+\partial_xV$. In order to write the system as a linearized part around $v_\epsilon$ plus a perturbation part, we also move the terms involving $\psi_\epsilon$ to the right-hand side and subtract $\partial_x(\phi_\epsilon(v_\epsilon)\partial_xV)$ from both sides:
\begin{equation*}
        \partial_t V -\partial_x W_0 -\partial_x(\phi_\epsilon(v_\epsilon)\partial_xV)=\partial_x\left[\psi_\epsilon(v_\epsilon+\partial_xV)-\psi_\epsilon(v_\epsilon)-\psi_\epsilon'(v_\epsilon)\partial_x V\right].
\end{equation*}
This previous equation can be written in the following compact way:
\begin{equation}
        \partial_t V -\partial_x W_0 -\partial_x(\phi_\epsilon(v_\epsilon)\partial_xV)=\partial_x H(\partial_xV),
\end{equation}
with
\begin{equation}
    H(f):=\psi_\epsilon(v_\epsilon+f)-\psi_\epsilon(v_\epsilon)-\psi_\epsilon'(v_\epsilon)f.
\end{equation}

\subsection{Preliminary estimates}

We give in this section some estimates on the functions $v_\epsilon$, $\psi_\epsilon$, $H$ and their derivatives.

\begin{lemma}
\label{vepsestimate}
(Estimates on $v_\epsilon$) For any $k\geq 1$, there exists a constant $C= C(k,\gamma,v_+,s)$ such that 
\begin{equation} 
    \forall \epsilon>0, \qquad\left|\partial_x^kv_\epsilon\right|\leq \frac{C(v_\epsilon-1)^{k\gamma +1}}{\epsilon^k}.
\end{equation}
\end{lemma}

\begin{proof}
By using the ODE \eqref{ODE} satisfied by $v_\epsilon$, one can show by induction that for every $k\geq 1$, there exists a function $f_k$ independent of $\epsilon$ such that
\begin{itemize}
    \item $f_k$ is smooth on $\mathbb{R}_+^*$,
    \item $f_k(1)\neq 0$,
    \item $\partial_x^kv_\epsilon=f_k(v_\epsilon)(v_\epsilon-1)^{k\gamma+1}/\epsilon^k$.
\end{itemize}
Since $f_k$ is continuous on $[1,v_+]$ and $1<v_\epsilon<v_+$, the factor $f_k(v_\epsilon)$ is uniformly bounded. 
\end{proof}

\begin{lemma}
\label{psiestimate}
    (Estimates on $\psi_\epsilon$) Let $\Bar{v}>1$ be arbitrary. For any $k\geq 1$, there exists $C= C(k,\gamma,\Bar{v})$ such that, for every $v\in (1,\Bar{v})$, for every $\epsilon>0$,
    \begin{equation}
        \left|\psi_\epsilon^{(k)}(v)\right|\leq \frac{C\epsilon}{(v-1)^{\gamma+k}}.
    \end{equation}
\end{lemma}

\begin{proof}
    As in the previous lemma, one can show that there exists a smooth function $g_k$ defined on $\mathbb{R}_+^*$, independent of $\epsilon$, such that $\psi_\epsilon^{(k)}(v)=\epsilon g_k(v)(v-1)^{-\gamma-k}$ and $g_k(1)\neq 0$. 
    Since $g_k$ is smooth, we can bound it on $[1, \Bar{v}]$.
\end{proof}

As a consequence of Lemmas \ref{vepsestimate} and \ref{psiestimate}, we obtain the following estimates:
\begin{lemma}
    \label{dxkveps}
    The functions $\phi_\epsilon^{1/2}\partial_x v_\epsilon$ and $\phi_\epsilon^{1/2}\dfrac{\partial_x^kv_\epsilon}{v_\epsilon-1}$, $k\geq 2$, belong to $L^2(\mathbb{R})$, with a time-independent $L^2(\mathbb{R})$ norm.
\end{lemma}
\begin{proof}
    We first see with Proposition \ref{travelingsolution} and the ODE \eqref{ODE} satisfied by $v_\epsilon$ that, for any $k\geq 1$, $\partial_x^kv_\epsilon\in L^1(\mathbb{R})\cap L^\infty(\mathbb{R})$. Since $\phi_\epsilon$ and $v_\epsilon-1$ go to a positive and finite limit as $\xi$ goes to $+\infty$, the integrability of the functions $\phi_\epsilon^{1/2}\partial_x v_\epsilon$ and $\phi_\epsilon^{1/2}\dfrac{\partial_x^kv_\epsilon}{v_\epsilon-1}$ at $+\infty$ is a consequence of the one of $\partial_x^k v_\epsilon$.

    Concerning the integrability when $\xi\rightarrow -\infty$, we first see with Lemma \ref{vepsestimate} that there exists $C=C(\epsilon)>0$ such that 
    \begin{equation*}
        |\phi_\epsilon^{1/2}\partial_x v_\epsilon| \leq C(v_\epsilon-1)^{(\gamma+1)/2},\quad \mathrm{and} \quad \left|\phi_\epsilon^{1/2}\frac{\partial_x^kv_\epsilon}{v_\epsilon-1}\right| \leq C(v_\epsilon-1)^{(2k-1)\gamma/2-1/2}. 
    \end{equation*}
    By Proposition \ref{travelingsolution}, we deduce that 
    \begin{equation*}
         |\phi_\epsilon^{1/2}\partial_x v_\epsilon| \leq C\left(B-\frac{A_1}{\epsilon}\xi\right)^{-1/2-1/(2\gamma)}\quad\mathrm{and}\quad\left|\phi_\epsilon^{1/2}\frac{\partial_x^kv_\epsilon}{v_\epsilon-1}\right| \leq C\left(B-\frac{A_1}{\epsilon}\xi\right)^{1/2+1/(2\gamma)-k}\in L^2(\mathbb{R}_-)
    \end{equation*}
    since $k\geq 2$. Note that we consider functions that depend only on the variable $\xi=x-st$, and so the $L^2$-norms are independent of time.
\end{proof}
\color{black}

\begin{lemma}
\label{Hestimate}
    (Bounds on H) Let $f$ such that there exists $\delta<1$ with $\|\frac{f}{v_\epsilon-1}\|_\infty\leq\delta$. Then one has the following bounds on $H(f)$:
    \begin{equation} \label{H}
        |H(f)|\leq \frac{C\epsilon}{(v_\epsilon-1)^{\gamma+2}}f^2,
    \end{equation}
    \begin{equation} \label{dxH}
        |\partial_xH(f)|\leq \frac{C}{(v_\epsilon-1)^2}f^2+\frac{C\epsilon}{(v_\epsilon-1)^{\gamma+2}}|f||\partial_x f|,
    \end{equation}
    \begin{equation}
    \begin{aligned} \label{dx2H}
        |\partial_{x}^2H(f)|\leq &\frac{C(v_\epsilon-1)^{\gamma-2}}{\epsilon}f^2+\frac{C\epsilon}{(v_\epsilon-1)^{\gamma+2}}|f|| \partial_{x}^2f| + \frac{C\epsilon}{(v_\epsilon-1)^{\gamma+2}}|\partial_xf|^2,
    \end{aligned}
        \end{equation}
        for some $C= C(s,\gamma,\delta, v_+)$.
\end{lemma}

\begin{proof}
    We first prove \eqref{H}. Recall that 
    \begin{equation*}
    H(f)=\psi_\epsilon(v_\epsilon+f)-\psi_\epsilon(v_\epsilon)-\psi_\epsilon'(v_\epsilon)f.
    \end{equation*}
    By Taylor's theorem,
    \begin{equation*}
        |H(f)|\leq \frac{f^2}{2}\sup_{|v-v_\epsilon|\leq |f|}|\psi_\epsilon^{(2)}(v)|.
    \end{equation*}
    By the hypothesis on $f$, for any $v$ such that $|v-v_\epsilon|\leq |f|$, it holds $0<(1-\delta)(v_\epsilon-1)<v-1<2v^+$.
    Lemma \ref{psiestimate} then implies that $|\psi_\epsilon^{(2)}(v)|\leq C\epsilon(v-1)^{-(\gamma+2)}$. Hence 
    \begin{equation*}
        |H(f)|\leq \frac{C\epsilon}{(v_\epsilon-1)^{\gamma+2}}f^2,
    \end{equation*}
    for some constant $C$ depending on $\gamma, v_+,\delta$. This is the first inequality. The other inequalities are proved in the same way. We differentiate $H$ with respect to $x$ and obtain
    \begin{equation*}
        \partial_xH(f)=\partial_xv_\epsilon\left(\psi_\epsilon'(v_\epsilon+f)-\psi_\epsilon'(v_\epsilon)-\psi_\epsilon^{(2)}(v_\epsilon)f\right)+\partial_xf\left(\psi_\epsilon'(v_\epsilon+f)-\psi_\epsilon'(v_\epsilon)\right).
    \end{equation*}
    which yields \eqref{dxH} by the same  arguments and Lemmas \ref{vepsestimate}, \ref{psiestimate}.
    We now differentiate a second time with respect to $x$:
    \begin{align*}
        \partial_{xx}H(f)= & \partial_{xx}v_\epsilon\left(\psi_\epsilon'(v_\epsilon+f)-\psi_\epsilon'(v_\epsilon)-\psi_\epsilon^{(2)}(v_\epsilon)f\right) + (\partial_xf)^2\psi_\epsilon^{(2)}(v_\epsilon+f)\\
        & + (\partial_x v_\epsilon)^2\left(\psi_\epsilon^{(2)}(v_\epsilon+f)-\psi_\epsilon^{(2)}(v_\epsilon)-\psi_\epsilon^{(3)}(v_\epsilon)f\right) \\
        & +2\partial_xf\partial_x v_\epsilon\left(\psi_\epsilon^{(2)}(v_\epsilon+f)-\psi_\epsilon^{(2)}(v_\epsilon)\right)+\partial_{xx}f\left(\psi_\epsilon'(v_\epsilon+f)-\psi_\epsilon'(v_\epsilon)\right).
    \end{align*}
     which yields \eqref{dx2H} by similar computations.
\end{proof}

\section{Well-posedness for the integrated system} \label{sectionGWP}

\subsection{Basic energy estimate}
The goal of this section is to prove the existence of strong solutions to the equation \begin{equation} \label{linearised}
        \partial_t V -\partial_x W_0 -\partial_x(\phi_\epsilon(v_\epsilon)\partial_xV)=\partial_x H(\partial_xV).
\end{equation}
Since $W_0$ is constant in time, we can control it for positive times through suitable assumptions on the initial data. Thus without loss of generality we treat $W_0$ as a perturbation of the system and prove the existence of a strong solution $V$ to \eqref{linearised}. Our strategy is to employ a fixed-point argument, which requires us to first derive appropriate energy estimates. The first (zero-th order) estimate is obtained by multiplying \eqref{linearised} by $V$:
\begin{equation} \label{VpreESTIMATES}
\begin{aligned}
    \frac{1}{2}\int_{\mathbb{R}} V^{2}(t)~dx + \int_{0}^{t} \int_{\mathbb{R}} \phi_{\epsilon}(v_{\epsilon}) |\partial_{x}V|^{2}~dxd{\tau} = \frac{1}{2}\int_{\mathbb{R}}V^{2}(0)~dx + \int_{0}^{t}\int_{\mathbb{R}} V (\partial_{x}H(\partial_x V) + \partial_{x}W_0)~dxd\tau.
\end{aligned}
\end{equation} 
From the left-hand side of the equation, we see that the natural energy space is $V\in L^\infty([0,T],L^2(\mathbb{R}))$ and  $\sqrt{\phi_\epsilon(v_\epsilon)}\partial_x V \in L^2((0,T)\times \mathbb{R})$. 
Let us now move to the right-hand side of \eqref{VpreESTIMATES} and try to close the estimate. For the term containing $W_0$, an integration by parts and Young's inequality yield
\begin{align*}
    \left|\int_0^t\int_{\mathbb{R}}V\partial_x W_0\right| = \left|-\int_0^t\int_{\mathbb{R}}W_0\partial_x V\right|\leq \frac{1}{2}\int_0^t\int_{\mathbb{R}}\phi_\epsilon(v_\epsilon)|\partial_x V|^2+\frac{1}{2}\int_0^t\int_{\mathbb{R}}\frac{W_0^2}{\phi_\epsilon(v_\epsilon)}.
\end{align*}
Hence this term can be bounded through suitable assumptions on $W_0$. For the term containing $H(\partial_x V)$, we also perform an integration by parts and use the estimate \eqref{H} on $H$ that we computed before:
\begin{align*}
    \left|\int_0^t\int_{\mathbb{R}}V\partial_x H(\partial_x V)\right| = \left|-\int_0^t\int_{\mathbb{R}}H(\partial_x V)\partial_x V\right| &\leq C\epsilon \int_0^t \int_{\mathbb{R}} \frac{|\partial_x V|^3}{(v_\epsilon-1)^{\gamma+2}}\\
    &\leq C\epsilon \left\|\frac{1}{\phi_\epsilon(v_\epsilon)(v_\epsilon-1)^{\gamma+1}}\right\|_{L^\infty_{t,x}}\int_0^t\int_{\mathbb{R}}\phi_\epsilon(v_\epsilon)|\partial_x V|^2\frac{|\partial_x V|}{v_\epsilon-1} .
\end{align*}
By recalling that $\phi_\epsilon(v_\epsilon)=\dfrac{\epsilon\gamma}{v_\epsilon(v_\epsilon-1)^{\gamma+1}}$ (see \eqref{diffusioncoefficient}), we obtain that
\begin{equation*}    \left|\int_0^t\int_{\mathbb{R}}V\partial_x H(\partial_x V)\right|  \leq C\int_0^t\int_{\mathbb{R}}\phi_\epsilon(v_\epsilon)|\partial_x V|^2\frac{|\partial_x V|}{v_\epsilon-1} \leq C\left\|\frac{\partial_x V}{v_\epsilon-1}\right\|_{L^\infty_{t,x}}\int_0^t\int_{\mathbb{R}}\phi_\epsilon(v_\epsilon)|\partial_x V|^2,
\end{equation*}
where $C=C(s,\gamma,\delta)$. In order to close the estimate, we need that the norm of the quantity $\partial_x V/(v_\epsilon-1)$ in $L^\infty([0,T]\times \mathbb{R})$ be small enough. The ideas of the proof of the well-posedness of \eqref{linearised} are thus the following:
\begin{itemize}
    \item Derive a-priori energy estimates for the quantities $\partial_x V/(v_\epsilon-1)$ and $\partial_x[\partial_x V/(v_\epsilon-1)]$.
    \item Use the fact that $\partial_x V/(v_\epsilon-1) \in L^\infty([0,T],H^1(\mathbb{R}))$ and the injection $H^1(\mathbb{R}) \hookrightarrow L^\infty(\mathbb{R}) $ to complete the estimates and bound these quantities in the same function spaces as chosen for $V$.
    \item Use a fixed-point argument to obtain the local-in-time existence and uniqueness of strong solutions to \eqref{linearised}.
\end{itemize}
\begin{remark}
    \label{stayuncongested}
    Note that the assumption that $\|\partial_x V/(v_\epsilon-1)\|_{L^\infty_{t,x}}$ is small, which enables to close the previous estimate, is also the one needed to bound the quantity $H(\partial_xV)$ and its derivatives (see Lemma \ref{Hestimate}). This assumption has a simple interpretation; by recalling that $\partial_x V=v-v_\epsilon$, we see that it prevents the perturbation $v$ from reaching 1. In other words, this assumption ensures that $\rho <1$, i.e. that the perturbation stays in the uncongested state.
\end{remark}
\begin{remark}
    As explained in the previous remark, it is crucial to bound the quantity $\partial_x V/(v_\epsilon-1)$. Note that adding the factor $1/(v_\epsilon-1)$ is not free, and the price to pay is that a lot of commutators appear and the equations become more complicated (see the next subsection). However, even if the weight $1/(v_\epsilon-1)$ is singular, the commutators are not difficult to bound. For instance, 
    \begin{equation*}
        \partial_x\left(\frac{\partial_x V}{v_\epsilon-1}\right)= \frac{\partial_x^2 V}{v_\epsilon-1}-\frac{v_\epsilon'}{(v_\epsilon-1)^2}\partial_x V.
    \end{equation*}
    From the estimate on $v_\epsilon'$ that we derived in Lemma \eqref{vepsestimate}, it follows that the coefficient $v_\epsilon'/(v_\epsilon-1)^2$ is bounded by $C(v_\epsilon-1)^{\gamma-1}/\epsilon$. Hence this coefficient is uniformly bounded in $t,x$ as soon as $\gamma \geq 1$, which is our hypothesis throughout the paper. The computations of the next subsections show that $\gamma\geq 1$ is the only hypothesis needed in order to obtain good bounds for every quantity which appears due to commutators. Furthermore, the factor $1/\epsilon$ in the coefficient $v_\epsilon'/(v_\epsilon-1)^2 $ of the commutator suggests that the operator $\partial_x $ scales as $1/\epsilon$ for solutions of the equations, i.e. that $\|\partial_x f\| \cong \|f\|/\epsilon$. This observation is reflected in the definition of the norm to be seen in Section \ref{constructglobalsolution}.
\end{remark}

\color{black}

\subsection{Formulation of the system for higher order estimates} In this subsection we describe our approach for deriving higher order energy estimates. Our estimates will heavily involve the quantity
\begin{equation}
    \label{defeta}
    \eta \equiv \eta(\partial_x V) := \frac{\partial_{x}V}{v_{\epsilon}-1}.
\end{equation} The evolution equation for $\eta$ reads as
\begin{equation}
    \begin{aligned}
        \partial_{t}\eta - \partial_{x}(\phi_{\epsilon}(v_{\epsilon}) \partial_{x}\eta) &= \frac{\partial_{x}^{2}W_0}{v_{\epsilon}-1} + \frac{\partial_{x}^{2}H}{v_{\epsilon}-1}  + \frac{s\eta v_{\epsilon}'}{v_{\epsilon}-1}\\[1ex] &+ \frac{\phi_{\epsilon}(v_{\epsilon})v_{\epsilon}' \partial_{x}^{2}V}{(v_{\epsilon}-1)^{2}} + \partial_{x} \left( \frac{\eta \phi_{\epsilon} (v_{\epsilon}) v_{\epsilon}'}{v_{\epsilon}-1} \right) + \frac{1}{v_{\epsilon}-1} \partial_{x} \left(\phi_{\epsilon}'(v_{\epsilon})v_{\epsilon}' \partial_{x}V \right),
    \end{aligned}
\end{equation} which may be expressed as \begin{equation} \label{etaLIN}
    \partial_{t} \eta - \partial_{x}(\phi_{\epsilon}(v_{\epsilon})\partial_{x}\eta)   =\mathcal{L}_{\epsilon}(\eta) + S_\epsilon(\partial_x V),
\end{equation} where  \begin{alignat*}{2}
    &\mathcal{L}_{\epsilon}(\eta) :=   \frac{s\eta v_{\epsilon}'}{v_{\epsilon}-1}  + \frac{\phi_{\epsilon}(v_{\epsilon})v_{\epsilon}'}{v_{\epsilon}-1}\left(\partial_{x}\eta + \frac{\eta v_{\epsilon}'}{v_{\epsilon}-1} \right)  + \partial_{x}\left( \frac{\eta \phi_{\epsilon}(v_{\epsilon})v_{\epsilon}'}{v_{\epsilon}-1}\right) + \frac{1}{v_{\epsilon}-1}\partial_{x} \left( \phi_{\epsilon}'(v_{\epsilon})v_{\epsilon}' (v_{\epsilon}-1) \eta \right), \\[1ex]
    &S_\epsilon (\partial_x V) := \frac{\partial_{x}^{2}[W_0+H(\partial_x V)]}{v_{\epsilon}-1}.
\end{alignat*}
Differentiating \eqref{etaLIN}, we obtain
\begin{equation} \label{dxetaLIN}
    \partial_{t}(\partial_{x}\eta) - \partial_{x}(\phi_{\epsilon}(v_{\epsilon})\partial_{x}^{2}\eta)  = \partial_{x}S_\epsilon(\partial_x V) +\mathcal{L}_{\epsilon}(\partial_{x}\eta) +  \mathcal{C}_\epsilon (\eta) ,
\end{equation}  where $  \mathcal{C}_\epsilon (\eta) := [\partial_x,\partial_x(\phi_\epsilon\partial_x\cdot)]\eta + [\partial_x,\mathcal{L}_\epsilon]\eta  = \partial_{x}(\partial_{x}\phi_{\epsilon}(v_{\epsilon})\partial_{x}\eta) + [\partial_x,\mathcal{L}_\epsilon] \eta$  appears due to commutators and \begin{equation*}
\begin{aligned}
[\partial_x,\mathcal{L}_\epsilon] \eta  &:=
  s \eta \partial_{x} \left( \frac{v_{\epsilon}'}{v_{\epsilon}-1} \right) + \partial_{x}\left(\frac{\phi_{\epsilon}(v_{\epsilon})v_{\epsilon}'}{v_{\epsilon}-1}\right)\partial_{x}\eta \\[1ex] &+ \partial_{x} \left( \phi_{\epsilon}(v_{\epsilon}) (v_{\epsilon}')^{2} \frac{1}{(v_{\epsilon}-1)^{2}}\right) \eta + \partial_{x} \left(\eta \partial_{x}\left( \frac{\phi_{\epsilon}(v_{\epsilon})v_{\epsilon}'}{v_{\epsilon}-1} \right) \right) \\[1ex] &- \frac{v_{\epsilon}'}{(v_{\epsilon}-1)^{2}} \partial_{x} (\phi_{\epsilon}'v_{\epsilon}'(v_{\epsilon}-1)\eta) + \frac{1}{v_{\epsilon}-1} \partial_{x} (\partial_{x}(\phi_{\epsilon}'(v_{\epsilon})v_{\epsilon}' (v_{\epsilon}-1))\eta).
\end{aligned}
\end{equation*} 

The equations \eqref{etaLIN} and \eqref{dxetaLIN} will be used to derive key estimates that allow us to eventually conclude the fixed-point argument.

Note that the operators $\mathcal{L}_\epsilon$ and $\mathcal{C}_\epsilon$, that appear in equations \eqref{etaLIN} and \eqref{dxetaLIN}, are linear. In order to improve readability, we give bounds for these two operators in the following lemma.
\begin{lemma}
    \label{linearoperatorestimates}
    There exists $C=C(\gamma,s,v_+)>0$ such that, for every $\alpha\in (0,1)$, for every $\epsilon>0$, for any $\eta$ sufficiently regular, there holds
    \begin{equation}
        \label{linearestimate1}
        \left|\int_0^t\int_{\mathbb{R}}\mathcal{L}_\epsilon(\eta)\eta \right| \leq \alpha \int_0^t\int_{\mathbb{R}}\phi_\epsilon(v_\epsilon)(\partial_x\eta)^2 + \frac{C}{\alpha\epsilon^2} \int_0^t\int_{\mathbb{R}}\phi_\epsilon(v_\epsilon)(\partial_x V)^2
    \end{equation}
    and 
    \begin{equation}
        \label{linearestimate2}
        \left|\int_0^t\int_{\mathbb{R}}\mathcal{C}_\epsilon(\eta)\partial_x\eta \right| \leq \alpha \int_0^t\int_{\mathbb{R}}\phi_\epsilon(v_\epsilon)(\partial_x^2\eta)^2 + \frac{C}{\alpha\epsilon^2} \int_0^t\int_{\mathbb{R}}\phi_\epsilon(v_\epsilon)\left((\partial_x\eta)^2+\frac{(\partial_x V^2)}{\epsilon^2}\right).
    \end{equation}
\end{lemma}
We defer the proof of this lemma to Appendix \ref{proofoflemma}.

\subsection{Construction of global strong solutions}
\label{constructglobalsolution}
In this subsection we lay out the framework for our fixed point argument before finally stating the result of this section, i.e. the well posedness of \eqref{linearised} in an appropriate space. Firstly, we fix $T>0$ arbitrary. For $t\in[0,T]$, we define the energies
\begin{equation} \label{E0}
    E_{0}(t; V) := \int_{\mathbb{R}} V^{2}(t)~dx, \quad D_{0}(t;V) := \int_{\mathbb{R}} \phi_{\epsilon}(v_{\epsilon})|\partial_{x}V(t)|^{2}~dx,
\end{equation} as well as
\begin{equation} \label{Ek}
    \begin{aligned}
        E_{k}(t;V) := \int_{\mathbb{R}} |\partial_{x}^{k-1}\eta|^{2}(t)~dx, \quad D_{k}(t;V) := \int_{\mathbb{R}} \phi_{\epsilon}(v_{\epsilon}) |\partial_{x}^{k}\eta|^{2}(t)~dx, \quad \text{ for } k=1,2.
    \end{aligned}
\end{equation}
 Then given $V_{1}$, we introduce the system
\begin{equation}
    \label{linearsystem}
        \left\{
    \begin{aligned}
        &\partial_t V_{2} -\partial_x W_0 -\partial_x(\phi_\epsilon(v_\epsilon)\partial_xV_{2})=\partial_x H(\partial_xV_{1}),\\
        &V_{2}|_{t=0} = V_{0},
    \end{aligned}
    \right.
\end{equation} and the application $\mathcal{A}^{\epsilon} : V_{1} \mapsto V_{2},$ where $V_{1}, V_{2} \in \mathcal{X}$ and \begin{equation*} 
    \mathcal{X} = \left\{ V : E_{k}(t, V) \in L^{\infty}(0,T) \text{ and } D_{k}(t; V) \in L^{1}(0,T)~ \text{ for } k=0,1,2.  \right\}, 
\end{equation*}  
\begin{equation*}
    \|V\|_{\mathcal{X}}^{2} := \sup_{t \in [0,T]} \left( \sum_{k=0}^{2}c_{k} \epsilon^{2k} \left[ E_{k}(t; V(t)) + \int_{0}^{t}D_{k}(\tau; V(\tau))~d\tau     \right]  \right),
\end{equation*} where $c_{k}=c_{k}(s,\gamma, v_{+})$, for $k=0,1,2$, are constants independent of $\epsilon$ which are to be determined. For a proof that the map $\mathcal{A}^\epsilon$ is well defined, see Appendix \ref{appendix}. For $\delta > 0$ we define the ball
\begin{equation}
    B_{\delta} := \left\{ V \in \mathcal{X} : \|V\|_{\mathcal{X}} < \delta \epsilon^{3/2}  \right\}.
\end{equation}
The remainder of this section aims to prove the following result.
\begin{proposition} \label{PROPstabilitycontraction}
    Assume that for some $\delta_{0}=\delta_{0}(s,v_{+}, \gamma) > 0$ it holds
        \begin{equation}  \label{IDassumptionSTABILITY} 
        \begin{aligned}
         \sum_{k=0}^2 \left(c_k \epsilon^{2k} E_{k}(0; V_{2})+\epsilon^{2k-1}\|\sqrt{x}\partial_x^k W_0\|_{L^2(\mathbb{R}_+)}^2\right) + \|W_0\|_{L^2_x}^2 + \left(\frac{T}{\epsilon}\right)^{1/\gamma}\left(\epsilon^2\|\partial_x W_0\|^2_{L^2_x}+\epsilon^{4} \|\partial_{x}^{2}W_{0}\|_{L^{2}_{x}}^{2}\right)  \leq \delta_0 \epsilon^3,
         \end{aligned}
    \end{equation}
    where $c_{k} = c_{k}(s,\gamma, v_{+})$ for $k=0,1,2$ are positive constants. Then there exists $\delta = \delta(s,\gamma,v_+)$ such that
    \begin{enumerate}
        \item the ball $B_{\delta}$ is stable by $\mathcal{A}^{\epsilon}$,
        \item the map $\mathcal{A}^{\epsilon}$ is a contraction on $B_{\delta}$.
    \end{enumerate}
    Consequently, $\mathcal{A}^{\epsilon}$ has a unique fixed point in $B_{\delta}$.
\end{proposition}

\subsection{Stability of the ball \texorpdfstring{$B_{\delta}$}{Lg}}
We let $\delta > 0$, $V_{1} \in B_{\delta}$ and $V_{2} := \mathcal{A}^{\epsilon}(V_{1})$. The aim of this subsection is to show that we can find $\delta>0$ such that $V_{2} \in B_{\delta}$. First, let us clarify the meaning of some notation we will use in this section.
\begin{itemize}
    \item We denote by $C= C(s,\gamma, v_{+})$ an arbitrary positive constant independent of $\epsilon$. This constant may change, even within the same line. We also denote $C' = C'(s,\gamma,v_+,c_0,c_1,c_2)$ another positive constant that may additionally depend on the constants $c_0,c_1,c_2$ appearing in the definition of the norm $\|\cdot\|_{\mathcal{X}}$.
    \item To ease notation, we will often shorten $\phi_{\epsilon}(v_{\epsilon}), \psi_{\epsilon}(v_{\epsilon})$ to $\phi_{\epsilon}, \psi_{\epsilon}$ respectively. For $i=1,2$ we also adopt the notation $\eta_{i} := \frac{\partial_{x}V_{i}}{v_{\epsilon}-1}$.
    \item We adopt the notation $X_{t}Y_{x} := X(0,T; Y(\mathbb{R}))$ for appropriate function spaces $X$ and $Y$.
\end{itemize}
Next, we make note of some useful estimates which will be repeatedly used in the remainder of the paper. We have up to a constant independent of $\epsilon$, \begin{align}  \nonumber &\phi_\epsilon(v_\epsilon) = \frac{\epsilon\gamma}{v_\epsilon(v_\epsilon-1)^{\gamma+1}},\\
\label{phiderivatives}
    &\phi_{\epsilon}^{(k)}(v_{\epsilon}) 
    \cong \frac{\phi_{\epsilon}(v_{\epsilon})}{(v_{\epsilon}-1)^{k}}, \\[1ex] & \partial_{x}v_{\epsilon} = v_{\epsilon}' ~\cong ~ \frac{1}{\epsilon}(v_{\epsilon}-1)^{\gamma+1}. \label{v'}
\end{align} As a consequence of \eqref{phiderivatives} and \eqref{v'}, we have $(v_{\epsilon}-1)^{2}\partial_{x}v_{\epsilon} \le C/\epsilon$ and $(v_{\epsilon}-1)^{2}/ \phi_{\epsilon} \le C/\epsilon$, provided $\gamma \ge 1$. Let us now note a Gagliardo-Nirenberg-Sobolev interpolation inequality which we will use. For $f \in H^1(\mathbb{R})$, 
\begin{align} \label{GNinf}
    &\|f\|_{L^{\infty}_{x}} \le C\|\partial_{x}f\|_{L^{2}_{x}}^{\frac{1}{2}} \|f\|_{L^{2}_{x}}^{\frac{1}{2}}, 
\end{align}
\color{black}
Finally, we mention some guidelines which our estimates will follow. \begin{itemize}
    \item We will often artificially place a factor of $\phi_{\epsilon}$ into the integral in order to obtain an expression which is a function of the energies \eqref{E0}-\eqref{Ek}. This results in the multiplication of a factor of $1/\epsilon$. For example, to estimate the term $ A:= \int_{0}^{t} \int_{\mathbb{R}} \partial_{x}V \partial_{x}\eta~dxd\tau$, we have
    \begin{equation}
        A \le \left\| \frac{1}{\phi_{\epsilon}}\right\|_{L^{\infty}_{t,x}} \int_{0}^{t} \int_{\mathbb{R}} | \sqrt{\phi_{\epsilon}} \partial_{x}V| |\sqrt{\phi_{\epsilon}} \partial_{x}\eta|~dxd\tau \le \frac{C}{\epsilon} \|\sqrt{\phi_{\epsilon}} \partial_{x}V\|_{L^{2}_{t,x}} \|\sqrt{\phi_{\epsilon}} \partial_{x}\eta\|_{L^{2}_{t,x}}.
    \end{equation}
    \item To estimate terms involving the expression $\partial_{x}^{2}V / (v_{\epsilon}-1)$, we will make use of the identity \begin{equation} \label{dx2V}
        \frac{\partial_{x}^{2}V}{v_{\epsilon}-1} = \partial_{x}\eta + \frac{v_{\epsilon}'\partial_{x}V}{(v_{\epsilon}-1)^{2}}.
    \end{equation}
\end{itemize}
\subsubsection{Estimates for \texorpdfstring{$k=0$}{Lg}}
Fix $t \in [0,T]$. We multiply \eqref{linearsystem} by $V_2$ and integrate on $\mathbb{R}\times(0,t)$ to obtain as before:
\begin{equation}
    \label{k=0mixedestimate}
    \frac{1}{2}\int_{\mathbb{R}}V_2^2(t)dx +\int_0^t\int_{\mathbb{R}}\phi_\epsilon|\partial_x V_2|^2dxd\tau = \frac{1}{2}\int_{\mathbb{R}}V_2^2(0)dx+\int_0^t\int_{\mathbb{R}}V_2(\partial_x H(\partial_x V_1)+\partial_x W_0)dxd\tau.
\end{equation}
\color{black}
Integrating by parts, we have
\begin{equation}
    \begin{aligned}
     \int_{0}^{t}\int_{\mathbb{R}} V_{2} \partial_{x}(H(\partial_{x}V_{1}) + W_0) =    - \int_{0}^{t} \int_{\mathbb{R}} H(\partial_{x}V_{1})\partial_{x}V_{2}  -\int_{0}^{t} \int_{\mathbb{R}} W_0 \partial_{x}V_2=: G_{1} + G_{2}.
     \end{aligned}
\end{equation}
     Using the estimate on $H$ and \eqref{GNinf}, and introducing $\eta_2 = \frac{\partial_x V_2}{v_\epsilon-1}$,
\begin{equation}
    \begin{aligned}     
     |G_{1}| &\le C\int_{0}^{t} \int_{\mathbb{R}} \phi_{\epsilon}(v_{\epsilon})  \frac{|\partial_{x}V_{1}|^{2}}{v_{\epsilon}-1} |\partial_{x}V_{2}| = C \int_{0}^{t}\int_{\mathbb{R}} \phi_{\epsilon}|\partial_{x}V_{1}|^{2} |\eta_{2}| \\[1ex] 
     &\le C \|\eta_{2}\|_{L^{\infty}_{t,x}} \|\sqrt{\phi_{\epsilon}}\partial_{x}V_{1}\|_{L^{2}_{t,x}}^{2} \le C\|\eta_2\|_{L^{\infty}_{t}L^{2}_{x}}^{1/2} \|\partial_{x}\eta_2\|_{L^{\infty}_{t}L^{2}_{x}}^{1/2}\|\sqrt{\phi_{\epsilon}}\partial_{x}V_{1}\|_{L^{2}_{t,x}}^{2} \\[1ex] 
     &\le \frac{C'}{\epsilon^{3/2}} \|V_{1}\|_{\mathcal{X}}^{2}\|V_{2}\|_{\mathcal{X}}.
    \end{aligned}
\end{equation}
\begin{remark}
    Note the following estimate, which can be obtained from \eqref{GNinf} and will also be used later on:
    \begin{equation}
        \|\eta\|_{L^\infty_{t,x}}\leq \frac{C'}{\epsilon^{3/2}}\|V\|_{\mathcal{X}}.
    \end{equation}
\end{remark}
We now estimate $G_2$. Young's inequality gives that
\begin{align*}
    \left|\int_{0}^{t} \int_{\mathbb{R}} W_0 \partial_xV_{2}\right| \leq \frac{1}{4}\int_{0}^{t} D_{0}(\tau; V_{2})~d\tau +\int_0^t\int_\mathbb{R}\frac{W_0^2}{\phi_\epsilon}.
\end{align*}
Using Fubini and the fact that $W_0$ is independant of time, one then has 
\begin{align*}
    \int_0^t\int_\mathbb{R}\frac{W_0^2}{\phi_\epsilon} = \int_\mathbb{R}W_0^2\int_0^t\frac{1}{\phi_\epsilon} = \int_\mathbb{R}W_0^2\int_0^t\frac{1}{\phi_\epsilon} \mathbf{1}_{\xi <0}\,d\tau dx + \int_\mathbb{R}W_0^2\int_0^t\frac{1}{\phi_\epsilon} \mathbf{1}_{\xi >0}\,d\tau dx,
\end{align*}
where $\xi:= x-s\tau$. By using Equation \eqref{asymptoticbound}, one then obtain for the first integral the estimate
\begin{equation*}
    \int_\mathbb{R}W_0^2\int_0^t\frac{1}{\phi_\epsilon} \mathbf{1}_{\xi <0}\,d\tau dx \leq \frac{C}{\epsilon}\int_\mathbb{R}W_0^2\int_0^{+\infty}\left(B-\frac{A_1}{\epsilon}\xi\right)^{-1-1/\gamma} \mathbf{1}_{\xi <0}\,d\tau dx \leq C\int_{\mathbb{R}}W_0^2\,dx,
\end{equation*}
and for the second integral 
\begin{align*}
    \int_\mathbb{R}W_0^2\int_0^t\frac{1}{\phi_\epsilon} \mathbf{1}_{\xi >0}\,d\tau dx \leq \frac{C}{\epsilon}\int_\mathbb{R}W_0^2\int_0^{x/s}\mathbf{1}_{\xi>0}d\tau dx=\frac{C}{\epsilon}\int_{\mathbb{R}}W_0^2 x \mathbf{1}_{x\geq 0}\,dx.
\end{align*} 
Returning to \eqref{k=0mixedestimate}, we get 
\begin{align*}
      &\int_{\mathbb{R}} V^{2}_2(t)~dx + \frac{3}{2}\int_{0}^{t} \int_{\mathbb{R}} \phi_{\epsilon}(v_{\epsilon}) |\partial_{x}V_2|^{2}~dxd\tau \\
      \le & \int_{\mathbb{R}}V^{2}_2(0)~dx + C \|W_{0}\|_{L^{2}_{x}}^{2} +\frac{C}{\epsilon}\|\sqrt{x}W_0\|_{L^2(\mathbb{R}_+)}+ \frac{C'}{\epsilon^{3/2}} \|V_{1}\|_{\mathcal{X}}^{2}\|V_{2}\|_{\mathcal{X}},
\end{align*}
i.e.
\begin{equation} 
\label{VpostESTIMATESimproved}
    E_{0}(t; V_{2}) + \frac{3}{2}\int_{0}^{t}D_{0}(\tau;V_{2})~d\tau \le E_{0}(0; V_{2}) + C \|W_0\|_{L^{2}_{x}}^{2}+\frac{C}{\epsilon}\|\sqrt{x}W_0\|^2_{L^2(\mathbb{R}_+)} + \frac{C'}{\epsilon^{3/2}} \|V_{1}\|_{\mathcal{X}}^{2}\|V_{2}\|_{\mathcal{X}} .
\end{equation}
\begin{remark}
    Note that \eqref{VpostESTIMATESimproved} does not depend on $T>0$. This is due to the assumption that $\sqrt{x}W_0\in L^2(\mathbb{R}_+)$. If this assumption is removed, one can still obtain a bound, but it is not global in time anymore. Indeed, we can simply apply Holder and Young's inequality to get
 \begin{equation*}
     \begin{aligned}
         \left|\int_{0}^{t} \int_{\mathbb{R}} W_0 \partial_{x}V_{2} \right| &\le \| \phi_{\epsilon}^{-1/2}\|_{L^{\infty}_{t,x}} \|\sqrt{\phi_{\epsilon}} \partial_{x}V_{2}\|_{L^{2}_{t,x}} \|W_0\|_{L^{2}_{t,x}} \le \frac{Ct}{\epsilon} \|W_{0}\|_{L^{2}_{x}}^{2} + \frac{1}{4} \int_{0}^{t} D_{0}(\tau; V_{2})~d\tau.
     \end{aligned}
 \end{equation*}
 Thus instead of \eqref{VpostESTIMATESimproved} we obtain 
\begin{equation}  \label{VpostESTIMATES}
    E_{0}(t; V_{2}) + \frac{3}{2}\int_{0}^{t}D_{0}(\tau;V_{2})~d\tau \le E_{0}(0; V_{2}) +  \frac{Ct}{\epsilon}\|W_{0}\|_{L^{2}_{x}}^{2} + \frac{C'}{\epsilon^{3/2}} \|V_{1}\|_{\mathcal{X}}^{2}\|V_{2}\|_{\mathcal{X}} .
\end{equation}
\end{remark}

\color{black}

\subsubsection{Estimates for \texorpdfstring{$k=1$}{Lg}}

From \eqref{etaLIN}, we deduce that $\eta_2$ solves 
\begin{equation*}
    \partial_{t} \eta_2 - \partial_{x}(\phi_{\epsilon}(v_{\epsilon})\partial_{x}\eta_2)   =\mathcal{L}_{\epsilon}(\eta_2) + S_\epsilon(\partial_x V_1).
\end{equation*}
\color{black}
Multiplying by $\eta_{2}$ and integrating in space and time, we get
\begin{equation} \label{k=1preESTIMATES}
    \begin{aligned}
        \frac{1}{2}\int_{\mathbb{R}} |\eta_{2}(t)|^{2}~dx &+ \int_{0}^{t} \int_{\mathbb{R}} \phi_{\epsilon}(v_{\epsilon})|\partial_{x}\eta_{2}|^{2}~dxd\tau - \frac{1}{2}\int_{\mathbb{R}} |\eta_{2}(0)|^{2}~dx \\[1ex]=& \int_{0}^{t}\int_{\mathbb{R}} \eta_{2} \frac{\partial_{x}^{2}H(\partial_{x}V_{1})}{v_{\epsilon}-1}~dxd\tau + \int_{0}^{t} \int_{\mathbb{R}} \eta_{2} \frac{\partial_{x}^{2}W_0}{v_{\epsilon}-1}~dxd\tau + \int_{0}^{t}\int_{\mathbb{R}} \mathcal{L}_\epsilon(\eta_2)\eta_2 \,\,dxd\tau \\
        &=: \sum_{n=1}^{3}I_{n}.
    \end{aligned}
\end{equation}
We now estimate each of $I_{1}, I_2 , I_{3}$. For $I_{1}$, we first integrate by parts to get 
\begin{equation*}
    I_{1} = - \int_{0}^{t} \int_{\mathbb{R}} \partial_{x}\eta_{2} \frac{\partial_{x}H(\partial_{x}V_{1})}{v_{\epsilon}-1} + \int_{0}^{t}\int_{\mathbb{R}} \eta_{2} \frac{v_{\epsilon}' \partial_{x}H(\partial_{x}V_{1})}{(v_{\epsilon}-1)^{2}} =: (1a) + (1b).
\end{equation*} Using \eqref{dxH}, 
\begin{equation*}
\begin{aligned}
    |(1a)| &\le C\int_{0}^{t} \int_{\mathbb{R}} |\partial_{x} \eta_{2}| \left| \frac{1}{v_{\epsilon}-1} \right| \left( \left| \frac{\partial_{x}V_{1}}{v_{\epsilon}-1} \right|^{2} + \frac{\epsilon}{(v_{\epsilon}-1)^{\gamma+2}}|\partial_{x}V_{1}| |\partial_{x}^{2}V_{1}| \right) .
\end{aligned}
\end{equation*}     
    The first term can be estimated as
\begin{equation*}
\begin{aligned}    
   \int_{0}^{t} \int_{\mathbb{R}} |\partial_{x} \eta_{2}| \left| \frac{1}{v_{\epsilon}-1} \right| \left| \frac{\partial_{x}V_{1}}{v_{\epsilon}-1} \right|^{2} &\le \left\| \frac{1}{\phi_{\epsilon}(v_{\epsilon}-1)^{2}}\right\|_{L^{\infty}_{t,x}} \|\eta_{1}\|_{L^{\infty}_{t,x}} \| \sqrt{\phi_{\epsilon}}\partial_{x}V_{1}\|_{L^{2}_{t,x}} \|\sqrt{\phi_{\epsilon}} \partial_{x} \eta_{2}\|_{L^{2}_{t,x}} \\[1ex] &\le \frac{C'}{\epsilon^{7/2}} \|V_{1}\|_{\mathcal{X}}^{2} \|V_{2}\|_{\mathcal{X}},
\end{aligned}
\end{equation*} where we have used \eqref{GNinf}. Note that the identity \eqref{dx2V} implies that \begin{equation} \label{dx2Vcomputation}
        \begin{aligned}
\left\|\sqrt{\phi_{\epsilon}}\frac{\partial_{x}^{2}V_{1}}{v_{\epsilon}-1} \right\|_{L^{2}_{t,x}} &\le \|\sqrt{\phi_{\epsilon}} \partial_{x}\eta_{1} \|_{L^{2}_{t,x}} + \|v_{\epsilon}'(v_{\epsilon}-1)^{-2} \sqrt{\phi_{\epsilon}} \partial_{x}V_{1}\|_{L^{2}_{t,x}}  \le \frac{C'}{\epsilon}\|V_{1}\|_{\mathcal{X}}.
        \end{aligned}
    \end{equation}Therefore, we have
\begin{equation*}
    \begin{aligned}        
        \int_{0}^{t}\int_{\mathbb{R}}        \frac{\epsilon}{(v_{\epsilon}-1)^{\gamma+1}}|\partial_{x}\eta_{2}| \left| \frac{\partial_{x}^{2}V_{1}}{v_{\epsilon}-1}\right|\left| \frac{\partial_{x}V_{1}}{v_{\epsilon}-1}\right| &\le 
        \left\|\frac{\epsilon}{(v_\epsilon-1)^{\gamma+1}\phi_\epsilon}\right\|_{L^\infty_{t,x}}\|\eta_1\|_{L^\infty_{t,x}} \left\|\sqrt{\phi_\epsilon}\partial_x\eta_2\right\|_{L^2_{t,x}}\left\|\sqrt{\phi_\epsilon}\frac{\partial_x^2 V_1}{v_\epsilon-1}\right\|_{L^2_{t,x}}\\
        &\le \frac{C'}{\epsilon^{7/2}} \|V_{1}\|_{\mathcal{X}}^{2}\|V_{2}\|_{\mathcal{X}}.
    \end{aligned} 
\end{equation*} Next note that using \eqref{dxH},
\begin{equation*}
    |(1b)| \leq \int_{0}^{t} \int_{\mathbb{R}} \frac{v_{\epsilon}'}{(v_{\epsilon}-1)^3} |\partial_{x}V_{2} | |\partial_{x}H(\partial_{x}V_{1})|  \le \int_{0}^{t} \int_{\mathbb{R}} | \partial_{x}V_{2} |\frac{v_{\epsilon}'}{(v_{\epsilon}-1)^3 }  \left( \left| \frac{\partial_{x}V_{1}}{v_{\epsilon}-1} \right|^{2} + \frac{\epsilon}{(v_{\epsilon}-1)^{\gamma+2}}|\partial_{x}V_{1}| |\partial_{x}^{2}V_{1}| \right).
\end{equation*} Then estimating in the same way as $(1a)$, we also find that $(1b) \le (C'/\epsilon^{7/2}) \|V_{1}\|_{\mathcal{X}}^{2}\|V_{2}\|_{\mathcal{X}}$ and so \begin{equation}
    |I_{1}| \le \frac{C'}{\epsilon^{7/2}} \|V_{1}\|_{\mathcal{X}}^{2}\|V_{2}\|_{\mathcal{X}}.
\end{equation}
Integrating by parts,
\begin{equation*}
    \begin{aligned}
        I_{2} = - \int_{0}^{t} \int_{\mathbb{R}} \frac{\partial_{x}W_0}{v_{\epsilon}-1} \partial_{x}\eta_{2} +  \int_{0}^{t}\int_{\mathbb{R}} \frac{v_{\epsilon}' \partial_{x}W_0}{(v_{\epsilon}-1)^{2}} \frac{\partial_{x}V_{2}}{v_{\epsilon}-1}.
    \end{aligned}
\end{equation*}
We first compute using Young's inequality that 
\begin{equation*}
    \left|\int_0^t\int_{\mathbb{R}}\frac{\partial_x W_0}{v_\epsilon-1}\partial_x\eta_2\right| \leq \frac{1}{8}\int_0^tD_1(\tau; V_2)\,d\tau + C\int_0^t\int_{\mathbb{R}}\frac{1}{\phi_\epsilon}\left(\frac{\partial_x W_0}{v_\epsilon-1}\right)^2.
\end{equation*}
The last integral can be estimated by
\begin{equation*}
    \int_0^t\int_{\mathbb{R}}\frac{1}{\phi_\epsilon}\left(\frac{\partial_x W_0}{v_\epsilon-1}\right)^2 \leq \int_{\mathbb{R}}(\partial_x W_0)^2\int_0^t\frac{1}{\phi_\epsilon(v_\epsilon-1)^2}\left(\mathbf{1}_{\xi<0}+\mathbf{1}_{\xi>0}\right)d\tau dx,
\end{equation*}
with $\xi:=x-s\tau$. The bound \eqref{asymptoticbound} then yields for the part $\{\xi<0\}$ the estimate
\begin{align*}
    \int_0^t\frac{1}{\phi_\epsilon(v_\epsilon-1)^2}\mathbf{1}_{\xi<0} \leq & \frac{C}{\epsilon}\int_0^t\left(B-\frac{A_1}{\epsilon}\xi\right)^{-1+1/\gamma}\mathbf{1}_{\xi<0}\\
    \leq & C\left[\left(B-\frac{A_1}{\epsilon}(x-st)\right)^{1/\gamma}-\left(B-\frac{A_1}{\epsilon}(x-\max(x,0))\right)^{1/\gamma}\right]\mathbf{1}_{x-st\leq 0}.
\end{align*}
Note that the bound obtained in the end is a positive nondecreasing function that goes to $+\infty$ as $t$ goes to $+\infty$, for every $x$. The inequality \eqref{doublebound} shows that this bound is optimal, i.e. that a similar lower bounds holds for this integral. Hence it seems complicated to bound the contribution of the integral $I_2$ uniformly in time. However, it is always possible to write a bound of the form 
\begin{equation*}
    \int_0^t\frac{1}{\phi_\epsilon(v_\epsilon-1)^2}\mathbf{1}_{\xi<0} \leq C\left(\frac{t}{\epsilon}\right)^{1/\gamma},
\end{equation*}
uniformly in $x$. For the part $\{\xi>0\}$, one has as before
\begin{align*}
    \int_\mathbb{R}(\partial_x W_0)^2\int_0^t\frac{1}{\phi_\epsilon(v_\epsilon-1)^2} \mathbf{1}_{\xi >0}\,d\tau dx \leq \frac{C}{\epsilon}\int_\mathbb{R}(\partial_x W_0)^2\int_0^{x/s}\mathbf{1}_{\xi>0}d\tau dx=\frac{C}{\epsilon}\int_{\mathbb{R}}(\partial_x W_0)^2 x \mathbf{1}_{x\geq 0}\,dx.
\end{align*}
For the second integral appearing in $I_2$, we use the same splitting between $\xi<0$ and $\xi >0$, together with \eqref{asymptoticbound} to obtain again
\begin{align*}
    \left| \int_{0}^{t}\int_{\mathbb{R}} \frac{v_{\epsilon}' \partial_{x}W_0}{(v_{\epsilon}-1)^{2}} \frac{\partial_{x}V_{2}}{v_{\epsilon}-1}\right| \leq \frac{1}{2\epsilon^2}\int_0^t D_0(\tau,V_2)d\tau + C\left(\frac{t}{\epsilon}\right)^{1/\gamma}\|\partial_x W_0\|^2_{L^2_x} + \frac{C}{\epsilon}\|\sqrt{x}\partial_x W_0\|_{L^2(\mathbb{R}_+)}^2.
\end{align*}

This leads to 
\begin{equation} \label{k=1ImprovedEstimate}
    |I_2| \leq \frac{1}{8}\int_0^tD_1(\tau;V_2)d\tau +\frac{1}{2\epsilon^2}\int_0^tD_0(\tau;V_2)d\tau +C\left(\frac{t}{\epsilon}\right)^{1/\gamma}\|\partial_x W_0\|^2_{L^2_x} + \frac{C}{\epsilon}\|\sqrt{x}\partial_x W_0\|_{L^2(\mathbb{R}_+)}^2.
\end{equation}
By comparison with \eqref{VpostESTIMATESimproved}, we see that multiplicating by the singular weight $1/(v_\epsilon-1)$ prevents the estimates from being uniform in time when $W_0\neq 0$.

\begin{remark}
    If we do not suppose that $\sqrt{x}\partial_x W_0 \in L^2(\mathbb{R}_+)$, but only that $W_0\in H^1(\mathbb{R})$, it is still possible to obtain a weaker estimate:
    \begin{equation*}
    \begin{aligned}
        |I_{2}| &\le \left\| \frac{1}{(v_{\epsilon}-1)\sqrt{\phi_{\epsilon}}}\right\|_{L^{\infty}_{t,x}} \|\sqrt{\phi_{\epsilon}} \partial_{x}\eta_{2}\|_{L^{2}_{t,x}} \|\partial_{x}W_0\|_{L^{2}_{t,x}} + \left\| \frac{v_{\epsilon}'}{(v_{\epsilon}-1)^{3}\sqrt{\phi_{\epsilon}}}\right\|_{L^{\infty}_{t,x}}\|\sqrt{\phi_{\epsilon} }\partial_{x}V_{2}\|_{L^{2}_{t,x}} \|\partial_{x}W_0\|_{L^{2}_{t,x}} \\[1ex] &\le \frac{Ct}{\epsilon} \|\partial_{x}W_{0}\|_{L^{2}_{x}}^{2} + \frac{C}{\epsilon^{2}} \int_{0}^{t}D_{0}(\tau;V_{2})~d\tau + \frac{1}{8} \int_{0}^{t} D_{1}(\tau;V_{2})~d\tau.
    \end{aligned}
\end{equation*}
\end{remark}
\color{black}
Finally, we use Equation \eqref{linearestimate1} of Lemma \ref{linearoperatorestimates} with $\alpha=1/8$ in order to bound $I_3$:
\begin{equation*}
    |I_3| = \left|\int_0^t\int_{\mathbb{R}}\mathcal{L}_\epsilon(\eta_2)\eta_2 \right| \leq \frac{1}{8} \int_0^t\int_{\mathbb{R}}\phi_\epsilon(\partial_x\eta_2)^2 + \frac{C}{\epsilon^2} \int_0^t\int_{\mathbb{R}}\phi_\epsilon(\partial_x V_2)^2
\end{equation*}
Collecting our estimates for $I_{1}-I_{3}$ and returning to \eqref{k=1preESTIMATES}, we find 
\begin{equation}
    \begin{aligned}
          E_{1}(t;V_{2}) + \frac{3}{2}\int_{0}^{t} D_{1}(\tau;V_{2})
~ d\tau  &\le E_{1}(0;V_{2}) + C\left(\frac{t}{\epsilon}\right)^{1/\gamma} \|\partial_{x}W_{0}\|_{L^{2}_{x}}^{2} +\frac{C}{\epsilon}\|\sqrt{x}\partial_x W_0\|^2_{L^2(\mathbb{R}_+)} \\[1ex] &~~~ + \frac{C'}{\epsilon^{7/2}} \|V_{1}\|_{\mathcal{X}}^{2}\|V_{2}\|_{\mathcal{X}} + \frac{C}{\epsilon^{2}}\int_{0}^{t}D_{0}(\tau;V_{2})~d\tau,
    \end{aligned}
\end{equation}
which yields, after multiplication by $\epsilon^2$,
\begin{equation}
\label{k=1postestimates}
    \begin{aligned}
         \epsilon^2 E_{1}(t;V_{2}) + \frac{3}{2}\epsilon^2\int_{0}^{t} D_{1}(\tau;V_{2})
~ d\tau  &\le \epsilon^2 E_{1}(0;V_{2}) + C\left(\frac{t}{\epsilon}\right)^{1/\gamma} \epsilon^2\|\partial_{x}W_{0}\|_{L^{2}_{x}}^{2} +C\epsilon\|\sqrt{x}\partial_x W_0\|^2_{L^2(\mathbb{R}_+)} \\[1ex] &~~~ + \frac{C'}{\epsilon^{3/2}} \|V_{1}\|_{\mathcal{X}}^{2}\|V_{2}\|_{\mathcal{X}} + C\int_{0}^{t}D_{0}(\tau;V_{2})~d\tau.
    \end{aligned}
\end{equation}
\begin{remark}
\label{k=1postestimatesIMPROVED}
If we do not suppose that $\sqrt{x}\partial_x W_0$ belong to $L^2(\mathbb{R}_+)$, we still have
\begin{equation*} 
    \begin{aligned}
           \epsilon^{2} E_{1}(t;V_{2})  +  \frac{3}{2}\epsilon^{2} \int_{0}^{t} D_{1}(\tau;V_{2})~d\tau \le & \epsilon^{2}E_{1}(0;V_{2}) + \frac{C'}{\epsilon^{3/2}} \|V_{1}\|_{\mathcal{X}}^{2}\|V_{2}\|_{\mathcal{X}}  + Ct\epsilon \|\partial_{x}W_{0}\|_{L^{2}_{t,x}}^{2} .
    \end{aligned}
\end{equation*}
\end{remark}
\color{black}
\subsubsection{Estimates for \texorpdfstring{$k=2$}{Lg}}

Equation \eqref{dxetaLIN} gives that $\partial_x \eta_2$ solves
\begin{equation*}
    \partial_{t}(\partial_{x}\eta_2) - \partial_{x}(\phi_{\epsilon}(v_{\epsilon})\partial_{x}^{2}\eta_2)  = \partial_{x}S_\epsilon(\partial_x V_1) +\mathcal{L}_{\epsilon}(\partial_{x}\eta_2) + \mathcal{C}_\epsilon (\eta_2). 
\end{equation*}
\color{black}
Multiplying by $\partial_{x}\eta_{2}$ and integrating in space and time leads to
\begin{equation} \label{dx2etaPREESTIMATES}
    \begin{aligned}
        \frac{1}{2}\int_{\mathbb{R}} |\partial_{x}\eta_{2}(t)|^{2}~dx &+ \int_{0}^{t} \int_{\mathbb{R}} \phi_{\epsilon}(v_{\epsilon}) |\partial_{x}^{2}\eta_{2}|^{2}~dxd\tau - \frac{1}{2}\int_{\mathbb{R}} |\partial_{x}\eta_{2}(0)|^{2}~dx \\[1ex] =&  \int_{0}^{t}\int_{\mathbb{R}}  \frac{\partial_{x}^{2}H(\partial_x V_1)}{v_{\epsilon}-1} \partial_{x}^{2}\eta_{2} + \int_{0}^{t}\int_{\mathbb{R}} \frac{\partial_{x}^{2}W_0}{v_{\epsilon}-1} \partial_{x}^{2}\eta_{2} + \int_{0}^{t}\mathcal{L}_\epsilon(\partial_x\eta_2)\partial_x \eta_2   + \int_{0}^{t}\int_{\mathbb{R}} \partial_{x}\eta_{2} ~ \mathcal{C}_\epsilon (\eta_{2})  \\[1ex] =&: \sum_{n=1}^{4}J_{n}. 
    \end{aligned}
\end{equation}
\color{black} We now estimate $J_{1}, J_2, J_{3}, J_4$. Using \eqref{dx2H}, \begin{equation} \label{J1}
    \begin{aligned}
        |J_{1}| &\le C \int_{0}^{t}\int_{\mathbb{R}} \frac{|\partial_{x}^{2}\eta_{2}|}{v_{\epsilon}-1} \left( \frac{(v_{\epsilon}-1)^{\gamma-2}}{\epsilon} |\partial_{x}V_{1}|^{2} + \frac{\epsilon}{(v_{\epsilon}-1)^{\gamma+2}}|\partial_{x}V_{1}| |\partial_{x}^{3}V_{1}| + \frac{\epsilon}{(v_{\epsilon}-1)^{\gamma+2}}|\partial_{x}^{2}V_{1}|^{2} \right) \\[1ex] &\le \frac{C}{\epsilon^{2}} \|\eta_{1}\|_{L^{\infty}_{t,x}}\|\sqrt{\phi_{\epsilon}} \partial_{x}^{2}\eta_{2}\|_{L^{2}_{t,x}} \| \sqrt{\phi_{\epsilon}} \partial_{x}V_{1}\|_{L^{2}_{t,x}} + C\epsilon \int_{0}^{t} \int_{\mathbb{R}} \frac{|\partial_{x}V_{1}|}{v_{\epsilon}-1} \frac{|\partial_{x}^{2}\eta_{2}|}{(v_{\epsilon}-1)^{\gamma+2}}|\partial_{x}^{3}V_{1}| \\[1ex] &~~~+ C\epsilon\int_{0}^{t} \int_{\mathbb{R}} \frac{|\partial_{x}^{2}\eta_{2}|}{v_{\epsilon}-1} \frac{ |\partial_{x}^{2}V_{1}|^{2}}{(v_{\epsilon}-1)^{\gamma+2}}.
    \end{aligned}
\end{equation}
We have that
\begin{equation*}
    \begin{aligned}
        \epsilon\int_{0}^{t} \int_{\mathbb{R}} \frac{|\partial_{x}V_{1}|}{v_{\epsilon}-1} \frac{|\partial_{x}^{2}\eta_{2}|}{(v_{\epsilon}-1)^{\gamma+2}}|\partial_{x}^{3}V_{1}|&\le C\|\eta_{1}\|_{L^{\infty}_{t,x}} \int_{0}^{t} \int_{\mathbb{R}} \frac{|\partial_{x}^{3}V_{1}|}{v_{\epsilon}-1} \phi_{\epsilon} |\partial_{x}^{2}\eta_{2}|  \\[1ex] &\le C \|\eta_{1}\|_{L^{\infty}_{t,x}} \left\|\sqrt{\phi_{\epsilon}}\partial_{x}^{2}\eta_{2}\right\|_{L^{2}_{t,x}} \left\|\sqrt{\phi_{\epsilon}}\frac{\partial_{x}^{3}V_{1}}{v_{\epsilon}-1}\right\|_{L^{2}_{t,x}}.
    \end{aligned}
\end{equation*} 
We compute that
\begin{equation*}
    \frac{\partial_{x}^{3}V_{1}}{v_{\epsilon}-1} = \partial_{x}^{2}\eta_{1} + \frac{2( v_{\epsilon}')}{v_{\epsilon}-1} \partial_{x}\eta_{1} + \frac{\partial_x V_1}{v_\epsilon-1}\left[\frac{(v_\epsilon')^2}{(v_\epsilon-1)^2}+\partial_x\left(\frac{v_\epsilon'}{v_\epsilon-1}\right)\right],
\end{equation*}
hence 
\begin{equation}
    \label{dx3v1}
    \left\|\sqrt{\phi_{\epsilon}}\frac{\partial_{x}^{3}V_{1}}{v_{\epsilon}-1}\right\|_{L^{2}_{t,x}} \leq \frac{C'}{\epsilon^2}\|V_1\|_{\mathcal{X}},
\end{equation}
which yields with the previous computations that
\begin{equation*}
     \epsilon\int_{0}^{t} \int_{\mathbb{R}} \frac{|\partial_{x}V_{1}|}{v_{\epsilon}-1} \frac{|\partial_{x}^{2}\eta_{2}|}{(v_{\epsilon}-1)^{\gamma+2}}|\partial_{x}^{3}V_{1}| \leq \frac{C'}{\epsilon^{11/2}}\|V_1\|_{\mathcal{X}}\|V_2\|_{\mathcal{X}}^2.
\end{equation*}
\color{black}
To estimate the final integral appearing in \eqref{J1}, let us first note that since \begin{equation*}
    \partial_{x} ( \sqrt{\phi_{\epsilon}}\partial_{x}\eta_{1}) =  \frac{v_{\epsilon}'\phi_{\epsilon}'}{2\sqrt{\phi_{\epsilon}}} \partial_{x}\eta_{1} + \sqrt{\phi_{\epsilon}} \partial_{x}^{2}\eta_{1},
\end{equation*} we have that
\begin{equation} \label{J1a}
    \begin{aligned}
        \|\partial_{x}(\sqrt{\phi_{\epsilon}}\partial_{x}\eta_{1})\|_{L^{2}_{t,x}} \le \frac{C}{\epsilon} \|\sqrt{\phi_{\epsilon}} \partial_{x}\eta_{1}\|_{L^{2}_{t,x}} + \|\sqrt{\phi_{\epsilon}}\partial_{x}^{2} \eta_{1}\|_{L^{2}_{t,x}} \le \frac{C'}{\epsilon^{2}} \|V_{1}\|_{\mathcal{X}}.
    \end{aligned}
\end{equation} 
Therefore,
\begin{equation} \label{dxeta-L2-Linf}
    \begin{aligned}
\|\sqrt{\phi_{\epsilon}}\partial_{x}\eta_{1}\|_{L^{2}_{t}L^{\infty}_{x}} \le \|\sqrt{\phi_{\epsilon}}\partial_{x}\eta_{1}\|_{L^{2}_{t,x}}^{1/2}\|\partial_{x}(\sqrt{\phi_{\epsilon}}\partial_{x}\eta_{1})\|_{L^{2}_{t,x}}^{1/2} \le \frac{C'}{\epsilon^{3/2}} \|V_{1}\|_{\mathcal{X}}
    \end{aligned}
\end{equation}
Using this estimate and \eqref{dx2V}, we have that
\begin{equation*}
    \begin{aligned}
\epsilon\int_{0}^{t} \int_{\mathbb{R}} \frac{|\partial_{x}^{2}\eta_{2}|}{v_{\epsilon}-1} \frac{ |\partial_{x}^{2}V_{1}|^{2}}{(v_{\epsilon}-1)^{\gamma+2}} &\le  C\int_{0}^{t} \int_{\mathbb{R}} \phi_{\epsilon}|\partial_{x}^{2}\eta_{2}| \frac{ |\partial_{x}^{2}V_{1}|^{2}}{(v_{\epsilon}-1)^{2}} \le C \int_{0}^{t} \int_{\mathbb{R}}\phi_{\epsilon}|\partial_{x}^{2}\eta_{2}| \left( |\partial_{x}\eta_{1}|^{2} + \frac{(v_{\epsilon}')^{2} |\partial_{x}V_{1}|^{2}}{(v_{\epsilon}-1)^{4}} \right) \\[1ex] &\le C\|\sqrt{\phi_{\epsilon}}\partial_{x}^{2}\eta_{2}\|_{L^{2}_{t,x}} \left( \|\partial_{x}\eta_{1}\|_{L^{\infty}_{t}L^{2}_{x}} \|\sqrt{\phi_{\epsilon}}\partial_{x}\eta_{1}\|_{L^{2}_{t}L^{\infty}_{x}} + \frac{C}{\epsilon^{2}} \|\eta_{1}\|_{L^{\infty}_{t,x}}\|\sqrt{\phi_{\epsilon}}\partial_{x}V_{1}\|_{L^{2}_{t,x}} \right) \\[1ex] &\le \frac{C'}{\epsilon^{11/2}} \|V_{1}\|_{\mathcal{X}}^{2}\|V_{2}\|_{\mathcal{X}}.
    \end{aligned}
\end{equation*} In summary, we find $J_{1} \le (C'/\epsilon^{11/2})\|V_{2}\|_{\mathcal{X}} \|V_{1}\|_{\mathcal{X}}^{2}$. \begin{remark}
    From the estimate for $J_{1}$ we can deduce that \begin{align} \label{L4estimateV}
&\|\sqrt{\phi_{\epsilon}}|\partial_{x}\eta|^{2}\|_{L^{2}_{t,x}}  \le \frac{C'}{\epsilon^{5/2}} \|V\|_{\mathcal{X}}, \\[1ex] &\|\sqrt{\phi_{\epsilon}}|\partial_{x}V|^{2}\|_{L^{2}_{t,x}}  \le \frac{C'}{\epsilon} \|V\|_{\mathcal{X}}, \label{L4estimateETA} \\
&\|\sqrt{\phi_{\epsilon}}\partial_{x}V\|_{L^{2}_{t}L^{\infty}_{x}}  \le \frac{C'}{\sqrt{\epsilon}} \|V\|_{\mathcal{X}}, \label{dxV-L2-Linf}
    \end{align} which will be useful for later estimates.
\end{remark}
For $J_2$, we proceed as in the $k=1$ case and obtain an estimate similar to \eqref{k=1ImprovedEstimate}:
\begin{equation*}
    |J_2| = \left|\int_{0}^{t}\int_{\mathbb{R}} \frac{\partial_{x}^{2}W_0}{v_{\epsilon}-1} \partial_{x}^{2}\eta_{2}\right|\leq \frac{1}{6}\int_0^tD_2(\tau;V_2)d\tau+C\left(\frac{t}{\epsilon}\right)^{1/\gamma}\|\partial_x^2W_0\|_{L^2_x}^2+\frac{C}{\epsilon}\|\sqrt{x}\partial_x^2W_0\|_{L^2(\mathbb{R}_+)}^2.
    \end{equation*}
\begin{remark}
    If $\sqrt{x}\partial_x^2W_0\notin L^2(\mathbb{R}_+)$, we still have using the Holder and Young inequalities that
\begin{equation*}
    |J_{2}|   \le \frac{C}{\sqrt{\epsilon}} \|\partial_{x}^{2}W_0\|_{L^{2}_{t,x}} \|\sqrt{\phi_{\epsilon}}\partial_{x}^{2}\eta_{2}\|_{L^{2}_{t,x}} \le \frac{Ct}{\epsilon} \|\partial_{x}^{2}W_{0}\|_{L^{2}_{x}}^{2} + \frac{1}{6} \int_{0}^{t} D_{2}(\tau;V_{2})~d\tau.
\end{equation*} 
\end{remark}
\color{black}

We now use Equation \eqref{linearestimate1} of Lemma \ref{linearoperatorestimates} with $\alpha=1/6$ to obtain that
\begin{equation*}
    |J_{3}| = \left|\int_{0}^{t}\mathcal{L}_\epsilon(\partial_x\eta_2)\partial_x \eta_2\right| \le \frac{1}{6}\int_0^t\int_{\mathbb{R}}\phi_\epsilon(\partial_x^2\eta_2)^2+ \frac{C}{\epsilon^2}\int_0^t\int_{\mathbb{R}}\phi_\epsilon (\partial_x \eta_2)^2.
\end{equation*} 

\color{black}

It now remains to estimate the term involving $\mathcal{C}_\epsilon$, i.e. $J_4$. We use this time Equation \eqref{linearestimate2} of Lemma \ref{linearoperatorestimates} with $\alpha= 1/6$ to get 
\begin{equation*}
    |J_4| = \left|\int_{0}^{t}\int_{\mathbb{R}} \partial_{x}\eta_{2} ~ \mathcal{C}_\epsilon (\eta_{2})\right| \leq \frac{1}{6} \int_0^t\int_{\mathbb{R}}\phi_\epsilon(\partial_x^2\eta)^2 + \frac{C}{\epsilon^2} \int_0^t\int_{\mathbb{R}}\phi_\epsilon\left((\partial_x\eta_2)^2+\frac{(\partial_x V_2^2)}{\epsilon^2}\right).
\end{equation*}

Returning to \eqref{dx2etaPREESTIMATES} with our estimates for $K_{1}-K_{7}$ and $J_{1}-J_{3}$, we have
\begin{equation*}
    \begin{aligned}
       E_{2}(t; V_{2}) + \int_{0}^{t} D_{2}(\tau;V_{2})~d\tau &\le 
 E_{2}(0; V_{2}) + C\left(\frac{t}{\epsilon}\right)^{1/\gamma} \|\partial_{x}^{2}W_{0}\|_{L^{2}_{x}}^{2} +\frac{C}{\epsilon}\|\sqrt{x}\partial_x^2 W_0\|_{L^2(\mathbb{R}_+)}^2+ \frac{C'}{\epsilon^{11/2}} \|V_{1}\|_{\mathcal{X}}^{2}\|V_{2}\|_{\mathcal{X}} \\[1ex] &~~~ +  \frac{C}{\epsilon^{2}} \int_{0}^{t} D_{1}(\tau;V_{2})~d\tau + \frac{C}{\epsilon^{4}} \int_{0}^{t}D_{0}(\tau;V_{2})~d\tau.
    \end{aligned}
\end{equation*}
We may assume that each $C>0$ appearing on the right hand-side is bounded by $B_0 = B_0(\gamma, v_{+},s) > 0$. Thus letting  $k_0 := 1/ (2B_0)$ and multiplying by $k_0\epsilon^{4}$ results in 
\begin{equation} \label{k=2presupremum}
    \begin{aligned}
       k_0 \epsilon^{4} E_{2}(t; V_{2}) + k_0\epsilon^{4} \int_{0}^{t} D_{2}(\tau;V_{2})~d\tau &\le 
 k_0 \epsilon^{4} E_{2}(0; V_{2}) + \left(\frac{t}{\epsilon}\right)^{1/\gamma}\epsilon^{4} \|\partial_{x}^{2}W_{0}\|_{L^{2}_{x}}^{2} + \frac{C'}{\epsilon^{3/2}}\|V_{1}\|_{\mathcal{X}}^{2} \|V_{2}\|_{\mathcal{X}} \\[1ex] &~~~+    \frac{\epsilon^{2}}{2} \int_{0}^{t} D_{1}(\tau;V_{2})~d\tau + \frac{1}{2} \int_{0}^{t}D_{0}(\tau;V_{2})~d\tau + \epsilon^3\|\sqrt{x}\partial_x^2 W_0\|_{L^2(\mathbb{R}_+)}^2.
    \end{aligned}
\end{equation} 
Adding to what we found for $k=1$, i.e. \eqref{k=1postestimates}, we find after simplifying the terms containing $D_1$ that
 \begin{equation*} 
    \begin{aligned}
         & \epsilon^{2} E_{1}(t; V_{2}) + \epsilon^{2} \int_{0}^{t} D_{1}(\tau;V_{2})~d\tau + k_0 \epsilon^{4} E_{2}(t; V_{2}) + k_0 \epsilon^{4} \int_{0}^{t} D_{2}(\tau;V_{2})~d\tau \\
         \le &  \epsilon^{2} E_{1}(0; V_{2}) + k_0\epsilon^{4} E_{2}(0; V_{2})+  \left(\frac{t}{\epsilon}\right)^{1/\gamma}\left(C\epsilon^2\|\partial_x W_0\|^2_{L^2_x}+\epsilon^{4} \|\partial_{x}^{2}W_{0}\|_{L^{2}_{x}}^{2}\right) + \frac{C'}{\epsilon^{3/2}}\|V_{1}\|_{\mathcal{X}}^{2} \|V_{2}\|_{\mathcal{X}} \\
         & + C\epsilon\|\sqrt{x}\partial_x W_0\|^2_{L^2(\mathbb{R}_+)}+ \epsilon^3\|\sqrt{x}\partial_x^2 W_0\|^2_{L^2(\mathbb{R}_+)} + \left(C+\frac{1}{2}\right)\int_0^tD_0(\tau;V_2)~d\tau.
    \end{aligned}
\end{equation*} 
Again, we may assume that each $C+1/2, C>0$ that appear on the right-hand side is bounded by $B_1=B_1(\gamma,v_+,s)>0$. Without loss of generality, we may assume that $B_1>1$. Thus letting $k_1:=1/(2B_1)$ and multiplying by $k_1$ yields 
 \begin{equation*} 
    \begin{aligned}
         & k_1 \epsilon^{2} E_{1}(t; V_{2}) + k_1\epsilon^{2} \int_{0}^{t} D_{1}(\tau;V_{2})~d\tau + k_1 k_0 \epsilon^{4} E_{2}(t; V_{2}) + k_1k_0\epsilon^{4} \int_{0}^{t} D_{2}(\tau;V_{2})~d\tau \\
         \le &  k_1\epsilon^{2} E_{1}(0; V_{2}) + k_1 k_0 \epsilon^{4} E_{2}(0; V_{2})+  \left(\frac{t}{\epsilon}\right)^{1/\gamma}\left(\epsilon^2\|\partial_x W_0\|^2_{L^2_x}+\epsilon^{4} \|\partial_{x}^{2}W_{0}\|_{L^{2}_{x}}^{2}\right) + \frac{C'}{\epsilon^{3/2}}\|V_{1}\|_{\mathcal{X}}^{2} \|V_{2}\|_{\mathcal{X}} \\
         & + \epsilon\|\sqrt{x}\partial_x W_0\|^2_{L^2(\mathbb{R}_+)}+ \epsilon^3\|\sqrt{x}\partial_x^2 W_0\|^2_{L^2(\mathbb{R}_+)} + \frac{1}{2} \int_0^tD_0(\tau;V_2)~d\tau.
    \end{aligned}
\end{equation*} 
Adding this inequality to the one that we obtained for $k=0$, i.e. \eqref{VpostESTIMATESimproved}, and simplifying the terms depending on $D_0$, we get 
 \begin{equation*} 
    \begin{aligned}
         & E_{0}(t; V_{2}) + \int_{0}^{t} D_{0}(\tau;V_{2})~d\tau  + k_1 \epsilon^{2} \left(E_{1}(t; V_{2}) +\int_{0}^{t} D_{1}(\tau;V_{2})~d\tau\right) + k_1 k_0 \epsilon^{4}\left( E_{2}(t; V_{2}) +  \int_{0}^{t} D_{2}(\tau;V_{2})~d\tau \right)\\
         & \le  E_{0}(0; V_{2})  + k_1\epsilon^{2} E_{1}(0; V_{2}) + k_1 k_0 \epsilon^{4} E_{2}(0; V_{2})+  \left(\frac{t}{\epsilon}\right)^{1/\gamma}\left(\epsilon^2\|\partial_x W_0\|^2_{L^2_x}+\epsilon^{4} \|\partial_{x}^{2}W_{0}\|_{L^{2}_{x}}^{2}\right) + \frac{C'}{\epsilon^{3/2}}\|V_{1}\|_{\mathcal{X}}^{2} \|V_{2}\|_{\mathcal{X}} \\
         & + \epsilon\|\sqrt{x}\partial_x W_0\|^2_{L^2(\mathbb{R}_+)}+ \epsilon^4\|\sqrt{x}\partial_x^2 W_0\|^2_{L^2(\mathbb{R}_+)}+ C\|W_0\|_{L^2_x}^2 + \frac{C}{\epsilon}\|\sqrt{x}W_0\|^2_{L^2(\mathbb{R}_+)}.
    \end{aligned}
\end{equation*} 
Again, we may assume that each $C>0$ in the right-hand side is bounded by $B_2=B_2(s,\gamma,v_+)>1$. We set $k_2:=1/B_2$ and multiply by $k_2$. In order to simplify computations, we also define 
\begin{equation}
    c_0:= k_2,\qquad c_1:= k_2 k_1, \qquad c_2:= k_2 k_1 k_0.
\end{equation}
We then obtain
\begin{equation*} 
    \begin{aligned}
          \sum_{k=0}^2 c_k \epsilon^{2k}\left(E_{k}(t; V_{2}) + \int_{0}^{t} D_{k}(\tau;V_{2})~d\tau \right) \leq &\sum_{k=0}^2 \left(c_k \epsilon^{2k} E_{k}(0; V_{2})+\epsilon^{2k-1}\|\sqrt{x}\partial_x^k W_0\|_{L^2(\mathbb{R}_+)}^2\right) + \|W_0\|_{L^2_x}^2\\
         & + \left(\frac{t}{\epsilon}\right)^{1/\gamma}\left(\epsilon^2\|\partial_x W_0\|^2_{L^2_x}+\epsilon^{4} \|\partial_{x}^{2}W_{0}\|_{L^{2}_{x}}^{2}\right)  + \frac{C'}{\epsilon^{3/2}}\|V_{1}\|_{\mathcal{X}}^{2} \|V_{2}\|_{\mathcal{X}}.
    \end{aligned}
\end{equation*} 
We use the inequality $t^{1/\gamma}\leq T^{1/\gamma}$ in the right-hand side, and then take the supremum in time in the left-hand side to obtain
\begin{equation} \label{k=2postestimates}
    \begin{aligned}
        \sup_{t \in [0,T]} &  \left( \sum_{k=0}^{2}  c_{k}\epsilon^{2k} \left[ E_{k}(t;V_{2}) + \int_{0}^{t}D_{k}(\tau;V_{2})~d\tau \right] \right) \le \sum_{k=0}^2\left[c_k \epsilon^{2k} E_k(0,V_2)+\epsilon^{2k-1}\|\sqrt{x}\partial_x^k W_0\|_{L^2(\mathbb{R}_+)}^2\right] \\[1ex] &+ \frac{C'}{\epsilon^{3/2}} \|V_{1}\|_{\mathcal{X}}^{2}\|V_{2}\|_{\mathcal{X}} + \|W_{0}\|_{L^{2}_{x}}^{2} + \epsilon^2\left(\frac{T}{\epsilon}\right)^{1/\gamma}\left( \left\| \partial_{x}W_{0}\right\|_{L^{2}_{t,x}}^{2}+\epsilon^2\|\partial_x^2W_0\|_{L^2_x}^2\right).
    \end{aligned}
\end{equation} 
Note that, when $W_0=0$, one can take $T=+\infty$ and obtain a global bound. Using the constants $c_0, c_1, c_2$ in the definition of the norm $\| \cdot \|_{\mathcal{X}}$,  \eqref{k=2postestimates} can be recast as
\begin{equation*}
    \begin{aligned}
    \|V_{2}\|_{\mathcal{X}}^{2} \le & \sum_{k=0}^2\left[c_k \epsilon^{2k} E_k(0,V_2)+\epsilon^{2k-1}\|\sqrt{x}\partial_x^k W_0\|_{L^2(\mathbb{R}_+)}^2\right] + \|W_{0}\|_{L^{2}_{x}}^{2}   \\
    &+ \epsilon^2\left(\frac{T}{\epsilon}\right)^{1/\gamma}\left( \left\| \partial_{x}W_{0}\right\|_{L^{2}_{t,x}}^{2}+\epsilon^2\|\partial_x^2W_0\|_{L^2_x}^2\right)+\frac{C'}{\epsilon^{3/2}}\|V_1\|_{\mathcal{X}}^2\|V_2\|_{\mathcal{X}} \\
    \leq & C' \delta^{2} \epsilon^{3/2} \|V_{2}\|_{\mathcal{X}}+  \delta_{0}^{2}\epsilon^{3} ,
    \end{aligned}
\end{equation*} where we have used our assumption on the initial data \eqref{IDassumptionSTABILITY}. Taking $\delta < 1$ to be such that $\delta < 1/(\sqrt{2}C')$, defining $\delta_{0} := \delta/2$ and using Young's inequality gives us \begin{equation}
    \|V_{2}\|_{\mathcal{X}}^{2} \le \delta^{2} \epsilon^{3},
\end{equation} and so $V_{2} \in B_{\delta}$ as required. This completes the first part of Proposition \ref{PROPstabilitycontraction}.

\begin{remark} \label{REMARK_improvedWP}
    When we do not suppose that $\sqrt{x}\partial_x^k W_0\in L^2(\mathbb{R}_+)$ for $k=0,1,2$, we still obtain that
\begin{equation*}
    \begin{aligned}
       \|V_2\|_{\mathcal{X}}^2 \leq \sum_{k=0}^2\left[c_k\epsilon^{2k}E_k(0;V_2)+T\epsilon^{2k-1}\|\partial_x^k W_0\|_{L^2_x}^2\right]+ C'\delta^2 \epsilon^{3/2} \|V_2\|_{\mathcal{X}}.
    \end{aligned}
\end{equation*} 
Hence the proof of the proposition still works after suitable modification of the assumption on the initial data. In this case, one has to impose the condition
\begin{equation*}
    \sum_{k=0}^2\left[c_k\epsilon^{2k}E_k(0;V_2)+T\epsilon^{2k-1}\|\partial_x^k W_0\|_{L^2_x}^2\right] \leq \delta_0^2\epsilon^3.
\end{equation*}
\end{remark}
\color{black}
\subsection{The map \texorpdfstring{$\mathcal{A}^{\epsilon}$}{Lg} is a contraction}
In this subsection we aim to show that the map $\mathcal{A}^{\epsilon}$ is a contraction. We consider two elements $V_{1}, V_{1}' \in B_{\delta}$ and define $V_{2} := \mathcal{A}^{\epsilon}(V_{1})$ and $V_{2}' := \mathcal{A}^{\epsilon}(V_{2}')$. Our goal is to show that there exists $\delta < 1$ with $\|V_{2} - V_{2}'\|_{\mathcal{X}} \le \delta \|V_{1} - V_{1}'\|_{\mathcal{X}}$. Defining $\tilde{V}_{1} := V_{1} - V_{1}'$ and $\tilde{V}_{2} := V_{2} - V_{1}'$, the following equations are satisfied:
\begin{equation}
        \left\{
    \begin{aligned} 
        &\partial_t \tilde{V}_{2} - \partial_{x}(\phi_{\epsilon}(v_{\epsilon}) \partial_{x}\tilde{V}_{2}) = \partial_{x} \left( H(\partial_{x}V_{1}) - H(\partial_{x}V_{1}') \right),\\
        &\partial_t \tilde{\eta}_{2} -\partial_x (\phi_{\epsilon}(v_{\epsilon}) \partial_{x}\tilde{\eta}_{2})  = \frac{\partial_{x}^{2}\left( H(\partial_{x}V_{1}) - H(\partial_{x}V_{1}') \right)}{v_{\epsilon}-1} + \mathcal{L}_{\epsilon}(\tilde{\eta}_{2}),  \\
        &\partial_{t} \partial_{x}\tilde{\eta}_{2} - \partial_{x}(\phi_{\epsilon}(v_{\epsilon})\partial_{x}^{2}\tilde{\eta}_{2}) = \partial_{x}\left(\frac{\partial_{x}^{2}\left( H(\partial_{x}V_{1}) - H(\partial_{x}V_{1}') \right)}{v_{\epsilon}-1} \right) + \mathcal{L}_{\epsilon}(\partial_{x}\tilde{\eta}_{2})  + \mathcal{C}_\epsilon (\tilde{\eta}_{2}). 
    \end{aligned}
    \right.
\end{equation}
We see that $\tilde{V}_2$, $\tilde{\eta_2}$ and $\partial_x\tilde{\eta}_2$ solve the same equations than $V_2,\eta_2,\partial_x\eta_2$ respectively, up to two differences:
\begin{itemize}
    \item The constant term $W_0$ is equal to zero here, as well as the initial data,
    \item The nonlinear terms (depending on $H$) are different.
\end{itemize}
With this in mind, multiplying the equations by $\tilde{V}_{2}, \tilde{\eta}_{2}, \partial_{x}\tilde{\eta}_{2}$ respectively and using again Lemma \ref{linearoperatorestimates} (taking advantage of the linearity of $\mathcal{L}_\epsilon$ and $\mathcal{C}_\epsilon$), we see that the computations of the previous section can be repeated. We get a similar estimate after suitable modification of the nonlinear terms. Since here $W_0$ and the initial data are equal to zero, we are left with
\begin{equation} \label{V2contractionPREestimates}
    \begin{aligned}
    \|\tilde{V}_{2}\|_{\mathcal{X}}^{2} \le& C\left|\int_{0}^{t} \int_{\mathbb{R}} \tilde{V}_{2} ~\partial_{x} \left( H(\partial_{x}V_{1}) - H(\partial_{x}V_{1}') \right)\right| + C\epsilon^{2}\left|\int_{0}^{t} \int_{\mathbb{R}} \tilde{\eta}_{2}~ \left( \frac{\partial_{x}^{2}\left( H(\partial_{x}V_{1}) - H(\partial_{x}V_{1}') \right)}{v_{\epsilon}-1}\right)\right| \\[1ex] &+ C\epsilon^{4}\left|\int_{0}^{t}\int_{\mathbb{R}} \partial_{x} \tilde{\eta}_{2} ~\partial_{x} \left( \frac{\partial_{x}^{2}(H(\partial_{x}V_{1}) - H(\partial_{x}V_{1}')}{v_{\epsilon}-1} \right)\right| =: C\sum_{k=1}^{3}\epsilon^{2k-2}|L_{k}|.
    \end{aligned}
\end{equation}
 To estimate each $L_{k}$ we will need the following result.
 \begin{lemma}
     If $f_{1}, f_{2} \in \mathcal{X}$   are such that $\|\frac{f_{1}}{v_\epsilon-1}\|_\infty + \|\frac{f_{2}}{v_\epsilon-1}\|_\infty \leq \delta$ for some $\delta < 1$, then there exists a constant $C>0$ independent of $\epsilon$ such that
     \begin{align} \label{Hdiff}
         &|H(f_{1}) - H(f_{2})| \le \frac{C \phi_{\epsilon}(v_{\epsilon})}{v_{\epsilon}-1} |f_{1}-f_{2}|( |f_{1}| + |f_{2}|), \\[1ex] \label{dxHdiff}
         &|\partial_{x}H(f_{1}) - \partial_{x}H(f_{2})| \le \frac{C\phi_{\epsilon}(v_{\epsilon})}{v_{\epsilon}-1} \left( \frac{v_{\epsilon}'}{v_{\epsilon}-1} |f_{1}-f_{2}| (|f_{1}|+|f_{2}|) + |\partial_{x}(f_{1}-f_{2})| |f_{1}| + |\partial_{x}f_{2}| |f_{1}-f_{2}| \right), \\[1ex]
         &|\partial_{x}^{2}H(f_{1}) - \partial_{x}^{2}H(f_{2})| \le   \frac{C(v_{\epsilon}-1)^{\gamma-2}}{\epsilon} |f_{1}-f_{2}|(|f_{1}|+f_{2}|)  \label{dx2Hdiff}  \\[1ex] \notag & \left. \hphantom{|\partial_{x}^{2}H(f_{1}) - \partial_{x}^{2}H(f_{2})| \le}+ \frac{C}{(v_{\epsilon}-1)^{2}}\left(|\partial_{x}(f_{1}-f_{2})||f_{2}| + |f_{1}-f_{2}||\partial_{x}f_{1}|\right) \right.  \\[1ex]&\left. \hphantom{|\partial_{x}^{2}H(f_{1}) - \partial_{x}^{2}H(f_{2})| \le} + \frac{C\phi_{\epsilon}(v_{\epsilon})}{v_{\epsilon}-1}|\partial_{x}(f_{1}-f_{2})|(|\partial_{x}f_{1}|+ |\partial_{x}f_{2}|)  \right. \notag \\[1ex] & \hphantom{|\partial_{x}^{2}H(f_{1}) - \partial_{x}^{2}H(f_{2})| \le~} +  \frac{C\phi_{\epsilon}(v_{\epsilon})}{v_{\epsilon}-1}\left(|f_{1}-f_{2}| |\partial_{x}^{2}f_{1}| + |\partial_{x}^{2}(f_{1}-f_{2})||f_{2}| + \frac{(\partial_{x}f_{2})^{2}}{v_{\epsilon}-1}|f_{1}-f_{2}| \right). \notag
         \end{align}
 \end{lemma}
 The proof follows the same approach as that of Lemma \ref{Hestimate}, and a very similar result can be seen in Lemma 3.4 of \cite{dalibard2019existence}. Thus we omit the proof.
Integrating by parts and using \eqref{Hdiff},
\begin{equation*}
    \begin{aligned}
    |L_{1}| &= \left|\int_{0}^{t} \int_{\mathbb{R}} \partial_{x}\tilde{V}_{2}~( H(\partial_{x}V_{1}) - H(\partial_{x} V_{1}'))\right| \le \int_{0}^{t} \int_{\mathbb{R}} |\partial_{x}\tilde{V}_{2}| \left( \frac{\phi_{\epsilon}}{v_{\epsilon}-1} |\partial_{x}\tilde{V}_{1}| (|\partial_{x}V_{1}| + |\partial_{x}V_{1}'|) \right) \\[1ex] &\le C \int_{0}^{t} \int_{\mathbb{R}} \sqrt{\phi_{\epsilon}} |\partial_{x}\tilde{V}_{2}| \frac{|\partial_{x}\tilde{V}_{1}|}{v_{\epsilon}-1} \left( \sqrt{\phi_{\epsilon}}|\partial_{x}V_{1}|+\sqrt{\phi_{\epsilon}}|\partial_{x}V_{1}'|\right)  \\[1ex] &\le  \left\|  \tilde{\eta}_{1} \right\|_{L^{\infty}_{t,x}} \|\tilde{V}_{2}\|_{\mathcal{X}} (\|V_{1}\|_{\mathcal{X}} + \|V_{2}\|_{\mathcal{X}}).
    \end{aligned}
\end{equation*}  Using the definition of the norm,
\begin{equation}
    |L_{1}| \le \frac{C}{\epsilon^{3/2}}\|\tilde{V}_{2}\|_{\mathcal{X}} \| \tilde{V}_{1}\|_{\mathcal{X}} \left( \|V_{1}\|_{\mathcal{X}} + \|V_{1}'\|_{\mathcal{X}}\right).
\end{equation}
 To deal with $L_{2}$, we first integrate by parts to get
\begin{equation}
    \begin{aligned}
        L_{2} &= - \int_{0}^{t} \int_{\mathbb{R}} \tilde{\eta}_{2} \frac{v_{\epsilon}' (\partial_{x}H(\partial_{x}V_{1}) - \partial_{x}H(\partial_{x}V_{1}'))}{(v_{\epsilon}-1)^{2}} - \int_{0}^{t} \int_{\mathbb{R}} \frac{\partial_{x}\tilde{\eta}_{2}}{v_{\epsilon}-1} (\partial_{x}H(\partial_{x}V_{1}) - \partial_{x}H(\partial_{x}V_{1}')) \\[1ex] &=: (2a) + (2b).
    \end{aligned}
\end{equation}
Using \eqref{dxHdiff}, \begin{equation*}
    \begin{aligned}
    |(2a)| &\le C\int_{0}^{t} \int_{\mathbb{R}} |\tilde{\eta}_{2}| \frac{(v_{\epsilon}')^{2}\phi_{\epsilon}}{(v_{\epsilon}-1)^{4}} |\partial_{x} \tilde{V}_{1}| ( |\partial_{x}V_{1}| + |\partial_{x}V_{1}'|) + C\int_{0}^{t} \int_{\mathbb{R}} |\tilde{\eta}_{2}| \frac{v_{\epsilon}'\phi_{\epsilon}}{(v_{\epsilon}-1)^{3}} |\partial_{x}^{2} \tilde{V}_{1}| |\partial_{x}V_{1}| \\[1ex] &~~~+C \int_{0}^{t} \int_{\mathbb{R}} |\tilde{\eta}_{2}| \frac{v_{\epsilon}'\phi_{\epsilon}}{(v_{\epsilon}-1)^{3}} |\partial_{x}^{2}V_{1}'| |\partial_{x} \tilde{V}_{1}| \\[1ex] &\le C\left\|\frac{(v_\epsilon')^2}{(v_\epsilon-1)^4}\right\|_{L^\infty_{t,x}} \|\tilde{\eta}_{2}\|_{L^{\infty}_{t,x}} \|\sqrt{\phi_{\epsilon}} \partial_{x}\tilde{V}_{1}\|_{L^{2}_{t,x}} ( \|\sqrt{\phi_{\epsilon}} \partial_{x}V_{1}\|_{L^{2}_{t,x}} + \|\sqrt{\phi_{\epsilon}} \partial_{x}V_{1}'\|_{L^{2}_{t,x}}) \\[1ex] &~~~+ C\left\|\frac{v_\epsilon'}{(v_\epsilon-1)^2}\right\|_{L^\infty_{t,x}}\|\tilde{\eta}_{2}\|_{L^{\infty}_{t,x}} \left(\|\sqrt{\phi_{\epsilon}} \frac{\partial_{x}^{2}\tilde{V}_{1}}{v_{\epsilon}-1}\|_{L^{2}_{t,x}} \| \sqrt{\phi_{\epsilon}} \partial_{x}V_{1}\|_{L^{2}_{t,x}} +  \|\sqrt{\phi_{\epsilon}} \frac{\partial_{x}^{2}V_{1}'}{v_{\epsilon}-1}\|_{L^{2}_{t,x}} \| \sqrt{\phi_{\epsilon}} \partial_{x}\tilde{V}_{1}\|_{L^{2}_{t,x}}\right).
    \end{aligned}
\end{equation*} Using \eqref{GNinf} and \eqref{dx2Vcomputation}, we therefore find
\begin{equation*}
    \begin{aligned}
        |(2a)| &\le \frac{C}{\epsilon^{7/2}} \|\tilde{V}_{2}\|_{\mathcal{X}} \| \tilde{V}_{1}\|_{\mathcal{X}} \left( \|V_{1}\|_{\mathcal{X}} + \|V_{1}'\|_{\mathcal{X}}\right).
    \end{aligned}
\end{equation*}
For $(2b)$ we use \eqref{dxHdiff} to get
\begin{equation*}
    \begin{aligned}
        |(2b)| \le \int_{0}^{t} \int_{\mathbb{R}} |\partial_{x} \tilde{\eta}_{2}| \frac{ \phi_{\epsilon}}{(v_{\epsilon}-1)^2} \left( \frac{v_{\epsilon}'}{v_{\epsilon}-1} |\partial_{x}\tilde{V}_{1}|( |\partial_{x}V_{1}| + |\partial_{x}V_{1}'|) + |\partial_{x}^{2}\tilde{V}_{1}| |\partial_{x}V_{1}| + |\partial_{x}^{2}V_{1}'||\partial_{x}\tilde{V}_{1}| \right) =: b_{1} + b_{2} + b_{3}.
    \end{aligned}
\end{equation*}
We then have that
\begin{equation*}
    \begin{aligned}
        b_{1} &\le C\left\|\frac{v_\epsilon'}{(v_\epsilon-1)^2} \right\|_{L^\infty_{t,x}}\int_{0}^{t} \int_{\mathbb{R}} |\sqrt{\phi_{\epsilon}}\partial_{x} \tilde{\eta}_{2}| |\sqrt{\phi_{\epsilon}}\partial_{x}\tilde{V}_{1}|( |\eta_1| + |\eta_1'|) \\[1ex] 
        &\le \frac{C}{\epsilon} \|\sqrt{\phi_{\epsilon}}\partial_{x}\tilde{\eta}_{2}\|_{L^{2}_{t,x}}\|\sqrt{\phi_{\epsilon}}\partial_{x}\tilde{V}_{1}\|_{L^{2}_{t,x}} \left( \|\eta_{1}\|_{L^{\infty}_{t,x}} + \|\eta_{1}'\|_{L^{\infty}_{t,x}}\right) \\[1ex] 
        &\le \frac{C}{\epsilon^{7/2}} \|\tilde{V}_{1}\|_{\mathcal{X}}\|\tilde{V}_{2}\|_{\mathcal{X}} \left(  \|V_{1}\|_{\mathcal{X}} + \|V_{1}'\|_{\mathcal{X}}\right).
    \end{aligned}
\end{equation*}Next,
\begin{equation*}
    \begin{aligned}
        b_{2} &\le \int_{0}^{t} \int_{\mathbb{R}} |\sqrt{\phi_{\epsilon}}\partial_{x}\tilde{\eta}_{2}| |\sqrt{\phi_{\epsilon}} \frac{\partial_{x}^{2}\tilde{V}_{1}}{v_{\epsilon}-1} | |\eta_1| \le C \|\eta_{1}\|_{L^{\infty}_{t,x}}\|\sqrt{\phi_{\epsilon}}\partial_{x}\tilde{\eta}_{2}\|_{L^{2}_{t,x}}  \|\sqrt{\phi_{\epsilon}}\frac{\partial_{x}^{2}\tilde{V}_{1}}{v_{\epsilon}-1}\|_{L^{2}_{t,x}}   \\[1ex] &\le \frac{C}{\epsilon^{7/2}} \|\tilde{V}_{1}\|_{\mathcal{X}}\|\tilde{V}_{2}\|_{\mathcal{X}}\|V_{1}\|_{\mathcal{X}}.
    \end{aligned}
\end{equation*}
We find similarly that 
\begin{equation}
     b_{3} \le \frac{C}{\epsilon^{7/2}}\|\tilde{V}_{1}\|_{\mathcal{X}}\|\tilde{V}_{2}\|_{\mathcal{X}}\|V_{1}'\|_{\mathcal{X}},
\end{equation} which yields with the previous estimates that 
\begin{equation*}
    \epsilon^2 |L_2| \leq \epsilon^2  \frac{C}{\epsilon^{7/2}} \|\tilde{V}_{2}\|_{\mathcal{X}} \| \tilde{V}_{1}\|_{\mathcal{X}} \left( \|V_{1}\|_{\mathcal{X}} + \|V_{1}'\|_{\mathcal{X}}\right) = \frac{C}{\epsilon^{3/2}} \|\tilde{V}_{2}\|_{\mathcal{X}} \| \tilde{V}_{1}\|_{\mathcal{X}} \left( \|V_{1}\|_{\mathcal{X}} + \|V_{1}'\|_{\mathcal{X}}\right).
\end{equation*}
We now estimate the term $L_{3}$. Using \eqref{dx2Hdiff} and \eqref{phiderivatives},
\begin{equation}
    \begin{aligned}
       |L_{3}|  \le  &\frac{C}{\epsilon}\int_{0}^{t} \int_{\mathbb{R}} |\partial_{x}^{2} \tilde{\eta}_{2}| (v_\epsilon-1)^{\gamma-3} |\partial_{x}\tilde{V}_{1}|(|\partial_{x}V_{1}| + |\partial_{x}V_{1}'|) + C\int_{0}^{t} \int_{\mathbb{R}}\frac{|\partial_{x}^{2} \tilde{\eta}_{2} |}{(v_{\epsilon}-1)^{3}}(|\partial_{x}^{2}\tilde{V}_{1}||\partial_{x}V_{1}| + |\partial_{x}\tilde{V}_{1}||\partial_{x}^{2}V_{1}|) \\[1ex] &+ C\int_{0}^{t}\int_{\mathbb{R}} |\partial_{x}^{2} \tilde{\eta}_{2}| \frac{\phi_{\epsilon}}{(v_{\epsilon}-1)^2} |\partial_{x}^{2}\tilde{V}_{1}|(|\partial_{x}^{2}V_{1}| + |\partial_{x}^{2}V_{1}'|) + C\int_{0}^{t}\int_{\mathbb{R}} |\partial_{x}^{2} \tilde{\eta}_{2}| \frac{\phi_{\epsilon}}{(v_{\epsilon}-1)^2}(|\partial_{x}\tilde{V}_{1}||\partial_{x}^{3}V_{1}| + |\partial_{x}^{3}\tilde{V}_{1}||\partial_{x}V_{1}'|) \\[1ex] &+ C\int_{0}^{t}\int_{\mathbb{R}} |\partial_{x}^{2} \tilde{\eta}_{2} |\frac{\phi_{\epsilon}}{(v_{\epsilon}-1)^{3}} |\partial_{x}^{2}V_{1}'|^{2}|\partial_{x}\tilde{V}_{1}| =: \sum_{n=1}^{5}M_{n}.
    \end{aligned}
\end{equation}
Using \eqref{dx2Vcomputation},
\begin{equation*}
    \begin{aligned}
        M_{1} &\le \frac{C}{\epsilon}\left\|\frac{(v_\epsilon-1)^{\gamma-2}}{\phi_\epsilon}\right\|_{L^\infty_{t,x}} \|\sqrt{\phi_{\epsilon}} \partial_{x}^{2} \tilde{\eta}_{2}\|_{L^{2}_{t,x}} \|\sqrt{\phi_{\epsilon}} \partial_{x}\tilde{V}_{1}\|_{L^{2}_{t,x}} \left( \|\eta_{1}\|_{L^{\infty}_{t,x}} + \|\eta_{1}'\|_{L^{\infty}_{t,x}} \right) \\[1ex] &\le \frac{C}{\epsilon^{11/2}} \|\tilde{V}_{2}\|_{\mathcal{X}} \|\tilde{V}_{1}\|_{\mathcal{X}} (\|V_{1}\|_{\mathcal{X}} + \|V_{1}'\|_{\mathcal{X}} ).
    \end{aligned}
\end{equation*}
Similarly,
\begin{equation*}
    \begin{aligned}
        M_{2} &\le C\left\|\frac{1}{\phi_\epsilon(v_\epsilon-1)}\right\|_{L^\infty_{t,x}}\|\sqrt{\phi_{\epsilon}}\partial_{x}^{2}\tilde{\eta}_{2}\|_{L^{2}_{t,x}} \left(\|\sqrt{\phi_{\epsilon}}\frac{\partial_{x}^{2}\tilde{V}_{1}}{v_{\epsilon}-1} \|_{L^{2}_{t,x}} \|\eta_{1}\|_{L^{\infty}_{t,x}} + \|\tilde{\eta}_{1}\|_{L^{\infty}_{t,x}}\|\sqrt{\phi_{\epsilon}}\frac{\partial_{x}^{2}V_{1}}{v_{\epsilon}-1} \|_{L^{2}_{t,x}} 
 \right) \\[1ex] &\le \frac{C}{\epsilon^{11/2}} \|\tilde{V}_{1}\|_{\mathcal{X}} \|\tilde{V}_{2}\|_{\mathcal{X}}(\|V_{1}\|_{\mathcal{X}}+\|V_{1}'\|_{\mathcal{X}}).
    \end{aligned}
\end{equation*}
Next,
\begin{equation*}
    \begin{aligned}
        M_{3} &\le C\int_{0}^{t} \int_{\mathbb{R}} |\sqrt{\phi_{\epsilon}}\partial_{x}^{2}\tilde{\eta}_{2}| \left|\sqrt{\phi_{\epsilon}} \frac{\partial_{x}^{2}\tilde{V}_{1}}{v_{\epsilon}-1} \right| \left( \left|\frac{\partial_{x}^{2}V_{1}}{v_\epsilon-1}\right| + \left|\frac{\partial_{x}^{2}V_{1}'}{v_\epsilon-1}\right|\right) \\[1ex] &\le C \left\|\sqrt{\phi_{\epsilon}}\partial_{x}^{2}\tilde{\eta}_{2}\right\|_{L^{2}_{t,x}} \left\| \sqrt{\phi_{\epsilon}} \frac{\partial_{x}^{2}\tilde{V_{1}}}{v_{\epsilon}-1}\right\|_{L^{2}_{t}L^{\infty}_{x}}\left(\left\|\frac{\partial_{x}^{2}V_{1}}{v_\epsilon-1}\right\|_{L^{\infty}_{t}L^{2}_{x}} + \left\|\frac{\partial_{x}^{2}V_{1}'}{v_\epsilon-1}\right\|_{L^{\infty}_{t}L^{2}_{x}}  \right).
    \end{aligned}
\end{equation*} Using \eqref{dx2V}, \eqref{dxeta-L2-Linf} and \eqref{dxV-L2-Linf},
\begin{equation} \label{dx2VcomputationL2LINF}
    \left\| \sqrt{\phi_{\epsilon}} \frac{\partial_{x}^{2}\tilde{V}_{1}}{v_{\epsilon}-1}\right\|_{L^{2}_{t}L^{\infty}_{x}} \le \|\sqrt{\phi_{\epsilon}}\partial_{x}\tilde{\eta}_{1}\|_{L^{2}_{t}L^{\infty}_{x}} + \left\|\frac{v_\epsilon'}{(v_\epsilon-1)^2}\right\|_{L^\infty_{t,x}}\|\sqrt{\phi_{\epsilon}}\partial_{x}\tilde{V}_{1}\|_{L^{2}_{t}L^{\infty}_{x}} \le \frac{C}{\epsilon^{3/2}}\|\tilde{V}_{1}\|_{\mathcal{X}}.
\end{equation} On the other hand \eqref{dx2V} also gives us
\begin{equation} \label{dx2V-Linf-L2}
\left\|\frac{\partial_{x}^{2}V_{1}}{v_\epsilon-1}\right\|_{L^{\infty}_{t}L^{2}_{x}} \le C\|\partial_{x}\eta_{1}\|_{L^{\infty}_{t}L^{2}_{x}} + \left\|\frac{v_\epsilon'}{(v_\epsilon-1)^2}\right\|_{L^\infty_{t,x}}\|\eta_1\|_{L^{\infty}_{t}L^{2}_{x}} \le \frac{C}{\epsilon^{2}} \|V_{1}\|_{\mathcal{X}},
\end{equation} and so
\begin{equation*}
    M_{3} \le \frac{C}{\epsilon^{11/2}} \|\tilde{V}_{1}\|_{\mathcal{X}} \|\tilde{V}_{2}\|_{\mathcal{X}}( \|V_{1}\|_{\mathcal{X}} + \|V_{1}'\|_{\mathcal{X}}) .
\end{equation*}
For $M_{4}$, we use again \eqref{dx3v1}, thus,
\begin{equation*}
    \begin{aligned}
        M_{4} &\le C\|\sqrt{\phi_{\epsilon}}\partial_{x}^{2}\tilde{\eta}_{2}\|_{L^{2}_{t,x}} \|\tilde{\eta}_{1}\|_{L^{\infty}_{t,x}}   \left\|\sqrt{\phi_{\epsilon}} \frac{\partial_{x}^{3}V_{1}}{v_{\epsilon}-1} \right\|_{L^{2}_{t,x}} +C\|\sqrt{\phi_{\epsilon}}\partial_{x}^{2}\tilde{\eta}_{2}\|_{L^{2}_{t,x}} \|\eta_{1}'\|_{L^{\infty}_{t,x}}   \left\|\sqrt{\phi_{\epsilon}} \frac{\partial_{x}^{3}\tilde{V}_{1}}{v_{\epsilon}-1} \right\|_{L^{2}_{t,x}} \\[1ex] &\le \frac{C}{\epsilon^{11/2}}\|\tilde{V}_{1}\|_{\mathcal{X}}\|\tilde{V}_{2}\|_{\mathcal{X}}  ( \|V_{1}\|_{\mathcal{X}} + \|V_{1}'\|_{\mathcal{X}}) .
    \end{aligned}
\end{equation*}
Lastly, using \eqref{dx2Vcomputation} and \eqref{dx2V-Linf-L2},
\begin{equation*}
    \begin{aligned}
        M_{5} &\le   C \|\tilde{\eta}_{1}\|_{L^{\infty}_{t,x}} \|\sqrt{\phi_{\epsilon}}\partial_{x}^{2}\tilde{\eta}_{2}\|_{L^{2}_{t,x}} \left\|\sqrt{\phi_{\epsilon}} \frac{\partial_{x}^{2}V_{1}'}{v_{\epsilon}-1}\right\|_{L^{2}_{t}L^{\infty}_{x}} \left\|\frac{\partial_{x}^{2}V_{1}'}{v_\epsilon-1}\right\|_{L^{\infty}_{t}L^{2}_{x}} \\[1ex] &\le \frac{C}{\epsilon^{7}} \|\tilde{V}_{1}\|_{\mathcal{X}}\|\tilde{V}_{2}\|_{\mathcal{X}}\|V_{1}'\|_{\mathcal{X}}^{2} = \frac{C}{\epsilon^{11/2}}\|\tilde{V}_{1}\|_{\mathcal{X}}\|\tilde{V}_{2}\|_{\mathcal{X}}\|V_{1}'\|_{\mathcal{X}} \left(\frac{\|V_{1}'\|_{\mathcal{X}}}{\epsilon^{3/2}} \right) .
    \end{aligned}
\end{equation*}
Finally,
\begin{align*}
    \epsilon^4|L_3|&\leq \epsilon^4\frac{C}{\epsilon^{11/2}}\|\tilde{V}_{1}\|_{\mathcal{X}}\|\tilde{V}_{2}\|_{\mathcal{X}}  ( \|V_{1}\|_{\mathcal{X}} + \|V_{1}'\|_{\mathcal{X}}) \left( 1 +  \frac{\|V_{1}'\|_{\mathcal{X}}}{\epsilon^{3/2}} \right) \\
    &\leq \frac{C}{\epsilon^{3/2}}\|\tilde{V}_{1}\|_{\mathcal{X}}\|\tilde{V}_{2}\|_{\mathcal{X}}  ( \|V_{1}\|_{\mathcal{X}} + \|V_{1}'\|_{\mathcal{X}}) \left( 1 +  \frac{\|V_{1}'\|_{\mathcal{X}}}{\epsilon^{3/2}} \right)
\end{align*}
Gathering our estimates for $L_1, L_2, L_3$ and returning to \eqref{V2contractionPREestimates}, using  the fact that $V_1,V_1'\in B_\delta$, we get
\begin{equation}
\begin{aligned}
    \|\tilde{V}_{2}\|_{\mathcal{X}}^{2} &\le \frac{C}{\epsilon^{3/2}}\|\tilde{V}_{1}\|_{\mathcal{X}}\|\tilde{V}_{2}\|_{\mathcal{X}}  ( \|V_{1}\|_{\mathcal{X}} + \|V_{1}'\|_{\mathcal{X}}) \left( 1 +  \frac{\|V_{1}'\|_{\mathcal{X}}}{\epsilon^{3/2}} \right) \\[1ex] &\le \delta C \|\tilde{V}_{1}\|_{\mathcal{X}}\|\tilde{V}_{2}\|_{\mathcal{X}}.
\end{aligned}
\end{equation}
Therefore taking $\delta < 1/C$ we obtain that $\mathcal{A}^{\epsilon}$ is a contraction, as required. This concludes the proof of Proposition \ref{PROPstabilitycontraction}.

\section{Asymptotic stability of \texorpdfstring{$(u_{\epsilon}, v_{\epsilon})$}{Lg}} \label{sec:ASYMstab}
So far, we have proved that if $(u,v)$ is a solution to the original system such that $v - v_{\epsilon}, u - u_{\epsilon} \in L^{1}_{0}(\mathbb{R})$ for all positive times then the system can be re-expressed in terms of the integrated quantities $(V,W_0)$, as \eqref{ARintegrated}. Furthermore, provided the initial energy is small enough (i.e. satisfying \eqref{IDassumptionSTABILITY}) then there exists a unique strong solution to \eqref{ARintegrated}. In order to prove the existence and uniqueness of strong solutions to the original system \eqref{ARoriginal}, we will need the following result:
\begin{proposition} \label{PROPstabilitylongtime}
   Provided that
    \begin{align} \label{PROPstabHYP1}
   & (u-u_{\epsilon})(0) \in L^{1}_{0}(\mathbb{R}) \cap H^{1}(\mathbb{R}),\quad \frac{\partial_x(u-u_\epsilon)(0)}{v_\epsilon-1}\in L^2(\mathbb{R}), \\[1ex] &(v - v_{\epsilon})(0)  \in L^{1}_{0}(\mathbb{R}),   \label{PROPstabHYP2}
    \end{align} we have 
\end{proposition}
\begin{enumerate}
    \item $\|u-u_{\epsilon}\|_{L^{\infty}(0,t; L^{1}(\mathbb{R}))} + \|v-v_{\epsilon}\|_{L^{\infty}(0,t; L^{1}(\mathbb{R}))} \le C(t, \epsilon), $
    \item $(u-u_{\epsilon})(t), (v-v_{\epsilon})(t) \in L^{1}_{0}(\mathbb{R})$ for each $t > 0$.
\end{enumerate} Our strategy of proof is as follows. \begin{itemize}
    \item \textbf{Step 1:} Establish the stability of the profile $u_{\epsilon}$. This involves bounding the quantity $u-u_{\epsilon}$ in $L^{\infty}_{t}H^{1}_{x}$, which will in particular be needed for the next steps.
    \item \textbf{Step 2:} Use the previous estimates to show that $u-u_{\epsilon}$ and $v-v_{\epsilon}$ remain in $L^{1}$ for positive times.
    \item \textbf{Step 3:} Complete the remaining part of Proposition \ref{PROPstabilitylongtime} using the estimates from the previous steps.
\end{itemize}

For the remainder of this section we define $v := v_{\epsilon} + \partial_{x}V$, $w := w_{\epsilon} + \partial_{x}W_0$ and $u:= w + \phi_{\epsilon}(v) \partial_{x}v$. Recall that $\eta = \partial_x V/(v_\epsilon-1)$ is bounded, with $\|\eta\|_{L^\infty_{t,x}} \leq C\epsilon^{-3/2}\|V\|_{\mathcal{X}}\leq C\delta $ since $V\in B_\delta$, hence $v\geq 1+(1-C\delta)(v_\epsilon-1)>1$ (up to reducing $\delta$) and $\phi_\epsilon(v)$ is well defined. As a consequence, the pair $(w,v)$ satisfies the system
\begin{equation*} 
    \left\{
    \begin{aligned}
        &\partial_t w = 0,\\
        &\partial_t v -\partial_x w - \partial_x (\phi_\epsilon(v)\partial_x v)=0.
    \end{aligned}
    \right.
\end{equation*}

\subsection{Stability of the profile \texorpdfstring{$u-u_{\epsilon}$}{Lg}}

From the system solved by $(w,v)$, we deduce that $u$  is a solution of 
\begin{equation} 
    \label{u-u_epsEQN}
        \partial_t (u - u_\epsilon) - \partial_x(\phi_\epsilon(v_\epsilon)\partial_{x}(u-u_{\epsilon})) = \partial_x[(\phi_\epsilon(v)-\phi_\epsilon(v_\epsilon))\partial_x u] =: G.
\end{equation}
We see that the quantity $G$ appearing on the right-hand side depends on the difference $\phi_\epsilon(v)-\phi_\epsilon(v_\epsilon)$. In order to control $G$, we will make use of the following lemma, which is analogous to Lemma \ref{Hestimate}.
\begin{lemma}
    Let $\Delta\phi_\epsilon(f) := \phi_\epsilon(v_{\epsilon}+f) - \phi_\epsilon(v_{\epsilon})$. Then we have
    \begin{align}  \label{Ef}
        &|\Delta\phi_\epsilon(f)| \le C \phi_{\epsilon}(v_{\epsilon}) \frac{|f|}{v_{\epsilon}-1}, \\[1ex]  \label{dxEf}
        &|\partial_{x}\Delta\phi_\epsilon(f)| \le C \phi_{\epsilon}(v_{\epsilon})\left(   \frac{\partial_{x}v_{\epsilon}|f|}{(v_{\epsilon}-1)^{2}}   +  \frac{|\partial_{x}f|}{v_{\epsilon}-1}  \right), \\[1ex]  \label{dx2Ef}
        &|\partial_{x}^{2}\Delta\phi_\epsilon(f)| \le C \phi_{\epsilon}(v_{\epsilon})\left( \frac{(\partial_{x}v_\epsilon)^{2}|f|}{(v_{\epsilon}-1)^{3}} + \frac{\partial_{x}v_{\epsilon} |\partial_x f|}{(v_{\epsilon}-1)^2} + \left| \frac{\partial_{x}f}{v_{\epsilon}-1} \right|^{2} + \frac{|\partial_{x}^{2}f|}{v_{\epsilon}-1}   \right).
    \end{align}
    As a consequence, we also have that
    \begin{equation} \label{phiv}
        |\phi_{\epsilon}(v)| =|\phi_\epsilon(v_\epsilon+\partial_x V)| \le |\phi_{\epsilon}(v) - \phi_{\epsilon}(v_{\epsilon})| + \phi_{\epsilon}(v_{\epsilon}) \le (|\eta| + 1)\phi_{\epsilon}(v_{\epsilon}).
    \end{equation}
\end{lemma} We omit the proof since it follows the same argument as that of Lemma \ref{Hestimate}. We now wish to prove the following existence result.
\begin{lemma} \label{LEMMAexistence-u-u_eps}
Suppose that $(U_{0}, V_{0}) \in H^{2}(\mathbb{R}) \times H^{3}(\mathbb{R})$ is such that \eqref{IDassumptionSTABILITY} is satisfied by $(W_{0}, V_{0})$, and that \eqref{PROPstabHYP1} is also satisfied. Consider the solution $(W_0,V) \in B_{\delta} \subset \mathcal{X}$ of \eqref{ARintegrated} obtained in Proposition \ref{PROPstabilitycontraction}. Then there exists a unique solution $u-u_{\epsilon}$ to \eqref{u-u_epsEQN} such that \begin{equation}
    u-u_{\epsilon} \in C([0,T]; H^{1}(\mathbb{R})) \cap L^{2}(0,T; H^{2}(\mathbb{R})). \end{equation}
   Moreover, $u-u_{\epsilon}$ satisfies
\begin{equation} \label{u-u_epsL2bounds}
    \begin{aligned}
          \sup_{t\in[0,T]}\left(\sum_{k=0}^{1} \left[ E_{k}(t; u-u_{\epsilon}) + \int_{0}^{t} D_{k}(\tau; u-u_{\epsilon})~d\tau\right] \right) \le  C,
    \end{aligned}
\end{equation}\color{black} where $C=C(\epsilon,\delta,v_+,\gamma,s)$.
\end{lemma} \begin{proof}
We start from the equation satisfied by $u-u_{\epsilon}$:
\begin{equation*} 
        \partial_t (u - u_\epsilon) - \partial_x(\phi_\epsilon(v_\epsilon)\partial_{x}(u-u_{\epsilon})) = \partial_x[(\phi_\epsilon(v)-\phi_\epsilon(v_\epsilon))\partial_x u] = G.
\end{equation*}
 Multiplying by $u-u_{\epsilon}$ and integrating by parts where appropriate, we have
\begin{equation*}
    \begin{aligned}
        \frac{1}{2} \int_{\mathbb{R}} |u-u_{\epsilon}|^{2}(t)~dx + \int_{0}^{t}\int_{\mathbb{R}} \phi_\epsilon(v_\epsilon) |\partial_{x}(u-u_{\epsilon})|^{2}~dxd\tau &-  \frac{1}{2} \int_{\mathbb{R}} |u-u_{\epsilon}|^{2}(0)~dx \\[1ex] &= - \int_{0}^{t}\int_{\mathbb{R}} \partial_{x}(u-u_{\epsilon}) (\phi_{\epsilon}(v) - \phi_{\epsilon}(v_{\epsilon}))\partial_{x}u~dxd\tau.
    \end{aligned}
\end{equation*}
With \eqref{Ef} and the triangle inequality, the right-hand side can be bounded by
\begin{align*}
    \left|\int_{0}^{t}\int_{\mathbb{R}} \partial_{x}(u-u_{\epsilon}) (\phi_{\epsilon}(v) - \phi_{\epsilon}(v_{\epsilon}))\partial_{x}u\right| \leq & C\|\eta\|_{L^\infty_{t,x}}\|\sqrt{\phi_\epsilon}\partial_x(u-u_\epsilon)\|_{L^2_{t,x}}\|\sqrt{\phi_\epsilon}\partial_x u\|_{L^2_{t,x}} \\
    \leq & C\|\eta\|_{L^\infty_{t,x}}\|\sqrt{\phi_\epsilon}\partial_x(u-u_\epsilon)\|_{L^2_{t,x}}^2 + C\|\eta\|_{L^\infty_{t,x}}\|\sqrt{\phi_\epsilon}\partial_xu_\epsilon\|_{L^2_{t,x}}^2 \\
    \leq & C\delta \|\sqrt{\phi_\epsilon}\partial_x(u-u_\epsilon)\|_{L^2_{t,x}}^2 + C(\delta,v_+,s,\epsilon),
\end{align*}


by using $\partial_x u_\epsilon = -s\partial_x v_\epsilon$ and $\|\eta\|_\infty \leq \delta$.
Hence we obtain the following estimate (up to reducing $\delta$):
\begin{equation*}
    \begin{aligned}
        \int_{\mathbb{R}} |u-u_{\epsilon}|^{2}(t)~dx + \int_{0}^{t}\int_{\mathbb{R}} \phi_\epsilon(v_\epsilon) |\partial_{x}(u-u_{\epsilon})|^{2}~dx \leq  \int_{\mathbb{R}} |u-u_{\epsilon}|^{2}(0)~dx + C(\delta,v_+,s,\epsilon).
    \end{aligned}
\end{equation*}
We now wish to find an estimate for $\partial_{x}(u-u_\epsilon)$. Notice that \eqref{u-u_epsEQN} has the same structure as the equation for $V$ (i.e. \eqref{linearised}), up to the source term.
As we did for $V$, we define 
\begin{equation}
\mu := \mu(\partial_x u) = \frac{\partial_x(u-u_\epsilon)}{v_\epsilon-1}.
\end{equation}
The computations that we did for $V$ imply that $\mu$ is solution of 
\begin{equation}
\partial_t \mu -\partial_x(\phi_\epsilon(v_\epsilon) \partial_x \mu) + \mathcal{L}_\epsilon(\mu) = \frac{\partial_x G}{v_\epsilon-1},
\end{equation}
where $\mathcal{L}_\epsilon$ is the same as in Equation \eqref{etaLIN}. Multiplying by $\mu$ and integrating yields
\begin{equation}
\label{basicmuestimate}
\frac{1}{2}\int_{\mathbb{R}}\mu^2(t) +\int_0^t\int_{\mathbb{R}}\phi_\epsilon(v_\epsilon)(\partial_x\mu)^2 = \frac{1}{2}\int_{\mathbb{R}}\mu^2|_{t=0} -\int_0^t\int_{\mathbb{R}}\mathcal{L}_\epsilon(\mu)\mu +\int_0^t\int_{\mathbb{R}}\mu\frac{\partial_x G}{v_\epsilon-1}.
\end{equation}
By \eqref{linearestimate1}, for  $\alpha=1/2$, there exists $C>0$ such that 
\begin{equation*}
\left|\int_0^t\int_{\mathbb{R}}\mathcal{L}_\epsilon(\mu)\mu \right| \leq \frac{1}{2} \int_0^t\int_{\mathbb{R}}\phi_\epsilon(v_\epsilon)(\partial_x\mu)^2 + C \int_0^t\int_{\mathbb{R}}\phi_\epsilon(v_\epsilon)[\partial_x(u-u_\epsilon)]^2.
\end{equation*}
The first integral can be absorbed by the left-hand side of \eqref{basicmuestimate}. The second integral can be bounded by using the energy estimate satisfied by $u-u_\epsilon$. Hence it is enough to bound the last integral in the right-hand side of \eqref{basicmuestimate}. We compute that 
\begin{align*}
\int_0^t\int_{\mathbb{R}}\mu\frac{\partial_x G}{v_\epsilon-1} = & -\int_0^t\int_{\mathbb{R}}\partial_x\mu\frac{G}{v_\epsilon-1} + \int_0^t\int_{\mathbb{R}}\mu\frac{\partial_x v_\epsilon G}{(v_\epsilon-1)^2} \\
=& -\int_0^t\int_{\mathbb{R}}\partial_x\mu \partial_x(\phi_\epsilon(v)-\phi_\epsilon(v_\epsilon))\mu - \int_0^t\int_{\mathbb{R}} \partial_x \mu \partial_x(\phi_\epsilon(v)-\phi_\epsilon(v_\epsilon))\frac{\partial_x u_\epsilon}{v_\epsilon-1}\\
& - \int_0^t\int_{\mathbb{R}}\partial_x\mu(\phi_\epsilon(v)-\phi_\epsilon(v_\epsilon))\frac{\partial_x^2(u-u_\epsilon)}{v_\epsilon-1} -\int_0^t\int_{\mathbb{R}}\partial_x\mu(\phi_\epsilon(v)-\phi_\epsilon(v_\epsilon))\frac{\partial_x^2 u_\epsilon}{v_\epsilon-1} \\
& +\int_0^t\int_{\mathbb{R}}\frac{\mu\partial_x v_\epsilon}{v_\epsilon-1} \partial_x(\phi_\epsilon(v)-\phi_\epsilon(v_\epsilon))\mu +\int_0^t\int_{\mathbb{R}} \frac{\mu\partial_x v_\epsilon}{v_\epsilon-1} \partial_x(\phi_\epsilon(v)-\phi_\epsilon(v_\epsilon))\frac{\partial_x u_\epsilon}{v_\epsilon-1}\\
& + \int_0^t\int_{\mathbb{R}}\frac{\mu\partial_x v_\epsilon}{(v_\epsilon-1)^2}(\phi_\epsilon(v)-\phi_\epsilon(v_\epsilon))\partial_x^2(u-u_\epsilon) + \int_0^t\int_{\mathbb{R}}\frac{\mu\partial_x v_\epsilon}{(v_\epsilon-1)^2}(\phi_\epsilon(v)-\phi_\epsilon(v_\epsilon))\partial_x^2 u_\epsilon \\
& =: \sum_{k=1}^8 I_k.
\end{align*}
By \eqref{dxEf} and Young's inequality, for any $\alpha>0$,
\begin{align*}
|I_1| \leq & C\int_0^t\int_{\mathbb{R}}\phi_\epsilon(v_\epsilon)|\mu\partial_x\mu|\left(\frac{\partial_x v_\epsilon |\eta|}{v_\epsilon-1}+ \frac{|\partial_{x}^{2} V|}{v_\epsilon-1}\right) \\
\leq & C\|\sqrt{\phi_\epsilon}\partial_x\mu\|_{L^2_{t,x}}\left(\frac{1}{\epsilon}\|\sqrt{\phi_\epsilon}\partial_x(u-u_\epsilon)\|_{L^2_{t,x}}\|\eta\|_{L^\infty_{t,x}}+\|\mu\|_{L^\infty_{t} L^2_{x}}\left\|\sqrt{\phi_\epsilon}\frac{\partial_{x}^{2}V}{v_\epsilon-1}\right\|_{L^2_{t} L^\infty_{x}}\right)\\
 \leq & \alpha \|\sqrt{\phi_\epsilon}\partial_x\mu\|_{L^2_{t,x}}^2 + \frac{C\delta^2}{\alpha}\left(\frac{1}{\epsilon^2}\|\sqrt{\phi_\epsilon}\partial_x(u-u_\epsilon)\|_{L^2_{t,x}}^2 + \|\mu\|_{L^\infty_{t} L^2_{x}}^2\right),
\end{align*}
where the last line uses the fact that 
\begin{equation*}
    \|\eta\|_{L^\infty} + \left\|\sqrt{\phi_\epsilon}\frac{\partial_x^2V}{v_\epsilon-1}\right\|_{L^2_tL^\infty_x}\leq \frac{C}{\epsilon^{3/2}}\delta \epsilon^{3/2}\leq C\delta.
\end{equation*}
(see Equation \eqref{dx2VcomputationL2LINF}). Similarly,
\begin{align*}
|I_2| \leq ~ & C\int_0^t\int_{\mathbb{R}}\phi_\epsilon(v_\epsilon)|\partial_x\mu|\left(\frac{\partial_x v_\epsilon |\eta|}{v_\epsilon-1}+ \frac{|\partial_{x}^{2} V|}{v_\epsilon-1}\right)\frac{\partial_x u_\epsilon}{v_\epsilon-1} \\
 \leq & \alpha  \|\sqrt{\phi_\epsilon}\partial_x\mu\|_{L^2_{t,x}}^2 + \frac{C}{\alpha}\left(\|\sqrt{\phi_\epsilon}\partial_x V\|_{L^2_{t,x}}^2\| \frac{\partial_x v_\epsilon}{(v_\epsilon-1)^{2}}\|^2_{L^\infty_{t,x}}+\|\sqrt{\phi_\epsilon}\frac{\partial_{x}^{2}V}{v_\epsilon-1}\|^2_{L^2_{t,x}}\right)\| \frac{\partial_x u_\epsilon}{v_\epsilon-1}\|^2_{L^\infty_{t,x}} \\
\leq & \alpha  \|\sqrt{\phi_\epsilon}\partial_x\mu\|_{L^2_{t,x}}^2 + C(\alpha,\epsilon,\delta,s,v_+,\gamma).
\end{align*}
By \eqref{Ef},
\begin{align*}
|I_3| \leq ~ & C\int_0^t\int_{\mathbb{R}}\phi_\epsilon |\partial_x\mu||\eta|\frac{|\partial_x^2(u-u_\epsilon)|}{v_\epsilon-1} \\
\leq & \alpha  \|\sqrt{\phi_\epsilon}\partial_x\mu\|_{L^2_{t,x}}^2 + \frac{C\delta^2}{\alpha}\left(\|\sqrt{\phi_\epsilon}\partial_x\mu\|_{L^2_{t,x}}^2 + \frac{1}{\epsilon^2} \|\sqrt{\phi_\epsilon}\partial_x(u-u_\epsilon)\|_{L^2_{t,x}}^2\right)
\end{align*}
by using 
\begin{equation*}
\frac{\partial^2_x(u-u_\epsilon)}{v_\epsilon-1} = \partial_x\mu + \frac{\partial_x v_\epsilon\partial_x(u-u_\epsilon)}{(v_\epsilon-1)^2}.
\end{equation*}
Then,
\begin{align*}
|I_4| &\leq C\int_0^t\int_{\mathbb{R}} \phi_{\epsilon}|\partial_x\mu|  |\eta|\frac{\partial_x^2u_\epsilon}{v_\epsilon-1} \leq \alpha  \|\sqrt{\phi_\epsilon}\partial_x\mu\|_{L^2_{t,x}}^2 + \frac{C\delta^2}{\alpha}\|\sqrt{\phi_\epsilon}\frac{\partial_x^2 u_\epsilon}{v_\epsilon-1}\|_{L^2_{t,x}}^2 \\
& \leq \alpha\|\sqrt{\phi_\epsilon}\partial_x\mu\|_{L^2_{t,x}}^2 + C(\alpha,\epsilon,\delta,s,v_+,\gamma),
\end{align*}
\begin{align*}
|I_5| \leq ~ & C \int_0^t\int_{\mathbb{R}}\mu^2\phi_\epsilon\frac{\partial_x v_\epsilon}{v_\epsilon-1}\left(\frac{\partial_x v_\epsilon |\eta|}{v_\epsilon-1}+ \frac{|\partial_{x}^{2} V|}{v_\epsilon-1}\right) \\
& \leq \frac{C\delta}{\epsilon^2}  \int_0^t\int_{\mathbb{R}}\phi_\epsilon [\partial_x(u-u_\epsilon)]^2 + \frac{1}{\epsilon}\|\mu\|_{L^\infty_{t} L^2_{x}}\|\sqrt{\phi_\epsilon}\partial_x(u-u_\epsilon)\|_{L^2_{t,x}}\left\|\sqrt{\phi_\epsilon}\frac{\partial_x^2V}{v_\epsilon-1}\right\|_{L^2_{t}L^\infty_{x}} \\
& \leq \alpha \|\mu\|_{L^\infty_{t} L^2_{x}}^2 + C(\alpha,\epsilon,\delta,s,v_+,\gamma),
\end{align*}
\begin{align*}
|I_6| \leq ~ & C \int_0^t\int_{\mathbb{R}}|\mu|\phi_\epsilon\frac{\partial_x v_\epsilon}{v_\epsilon-1}\left(\frac{\partial_x v_\epsilon |\eta|}{v_\epsilon-1}+ \frac{|\partial_{x}^{2} V|}{v_\epsilon-1}\right)\frac{\partial_x u_\epsilon}{v_\epsilon-1} \\
 \leq & \frac{C}{\epsilon^2}\left\|\sqrt{\phi_\epsilon}\partial_x(u-u_\epsilon)\right\|_{L^2_{t,x}}\left(\frac{1}{\epsilon}\left\|\sqrt{\phi_\epsilon}\partial_x V\right\|_{L^2_{t,x}} + \left\|\sqrt{\phi_\epsilon}\frac{\partial_x^2V}{v_\epsilon-1}\right\|_{L^2_{t,x}}\right) \\
 \leq & C(\epsilon,\delta,s,v_+,\gamma),
\end{align*}
\begin{align*}
|I_7| \leq & C\int_0^t\int_{\mathbb{R}}|\mu|\phi_\epsilon |\eta|\frac{\partial_x v_\epsilon\partial_x^2(u-u_\epsilon)}{(v_\epsilon-1)^2} \\
\leq & \frac{C}{\epsilon}\|\eta\|_{L^\infty_{t,x}}\left\|\sqrt{\phi_\epsilon}\partial_x(u-u_\epsilon)\right\|_{L^2_{t,x}}\left(\left\|\sqrt{\phi_\epsilon}\partial_x\mu\right\|_{L^2_{t,x}} + \frac{1}{\epsilon}\left\|\sqrt{\phi_\epsilon}\partial_x(u-u_\epsilon)\right\|_{L^2_{t,x}} \right)\\
\leq & \alpha\left\|\sqrt{\phi_\epsilon}\partial_x\mu\right\|_{L^2_{t,x}}^2 + C(\alpha,\epsilon,\delta,v_+,s,\gamma),
\end{align*}
and 
\begin{align*}
|I_8| \leq & C\int_0^t\int_{\mathbb{R}}|\mu|\phi_\epsilon |\eta|\frac{\partial_x v_\epsilon|\partial_x^2u_\epsilon|}{(v_\epsilon-1)^2} \leq  \frac{C}{\epsilon^3}\left\|\sqrt{\phi_\epsilon}\partial_x(u-u_\epsilon)\right\|_{L^2_{t,x}}\left\|\sqrt{\phi_\epsilon}\partial_x V\right\|_{L^2_{t,x}} \leq C(\epsilon,\delta,s,v_+,\gamma).
\end{align*}
Hence, coming back to Equation \eqref{basicmuestimate}, by taking $\alpha>0$ sufficiently small, up to reducing $\delta$, we obtain that there exists $\alpha_0>0$ such that 
\begin{equation*}
\alpha_0\left(\left\|\mu\right\|_{L^\infty_{t} L^2_{x}}^2 + \left\|\sqrt{\phi_\epsilon}\partial_x\mu\right\|_{L^2_{t,x}}^2\right) \leq \frac{1}{2}\|\mu_0\|_{L^2_{x}}^2+C(\alpha_0,\epsilon,\delta,s,v_+,\gamma),
\end{equation*}
with $C(\alpha_0,\epsilon,\delta,s,v_+,\gamma)\rightarrow +\infty$ as $\epsilon\rightarrow 0$.

\color{black}

With this estimate, the existence and uniqueness of $u-u_{\epsilon}$ is classical. It remains to verify that $u-u_{\epsilon}$ is continuous in time. Let $f \in H^{1}(\mathbb{R})$ with $\|f\|_{H^{1}(\mathbb{R})} = 1$. Then using $\partial_{t}(v-v_{\epsilon}) = \partial_{x}(u-u_{\epsilon})$, we have
   \begin{equation*}
       \begin{aligned}
            \langle \partial_{t}(v-v_{\epsilon}), f \rangle_{H^{-1} \times H^{1}} &= (\partial_{x}(u-u_{\epsilon}), f) = (u-u_{\epsilon}, \partial_{x}f)  \\[1ex] &\le  \| (u-u_{\epsilon})(t)\|_{L^{2}_{x}} 
        \end{aligned}
    \end{equation*}
    and so $\|\partial_t(v-v_\epsilon)\|_{L^\infty_t H^{-1}_x}\leq \|u-u_\epsilon\|_{L^\infty_t L^2_x}$.
    Since $v-v_\epsilon= \partial_x V$, we have that $\partial_{t}\partial_{x}V \in L^{\infty}(0,T; H^{-1}(\mathbb{R}))\subset L^{2}(0,T; H^{-1}(\mathbb{R}))$. Since we also have $\partial_{x}V \in L^{2}(0,T, H^{2}(\mathbb{R}))$, it follows (e.g. from Theorem II.5.13. of \cite{boyer_mathematical_2012}) that $\partial_{x}V \in C([0,T]; H^{1}(\mathbb{R}))$ and so $V \in C([0,T]; H^{2}(\mathbb{R}))$. Now since $w = u - \phi_{\epsilon}(v)\partial_x v$ and $w_{\epsilon} = u_{\epsilon} - \phi_{\epsilon}(v_{\epsilon})\partial_x v_\epsilon$, we have that $\partial_{x}W_0 = u-u_{\epsilon} + \phi_{\epsilon}'(v)\partial_{x}v - \phi_{\epsilon}'(v_{\epsilon})\partial_{x}v_{\epsilon}$. After rearranging, we find that \begin{equation*}
       u - u_{\epsilon} = \partial_{x}W_0 - \phi_{\epsilon}'(v)\partial_{x}V + (\phi_{\epsilon}'(v) - \phi_{\epsilon}'(v_{\epsilon}))\partial_{x}v_{\epsilon}.
    \end{equation*} The continuity in time of $W_0, v_{\epsilon}, \phi_\epsilon'(v)$ and $V$ imply that $u- u_{\epsilon} \in C([0,T]; H^{1}(\mathbb{R}))$ as claimed.
\end{proof}
\subsection{Bounding \texorpdfstring{$(v-v_{\epsilon})$ in $L^{1}(\mathbb{R})$}{Lg}} We first obtain a $L^{1}_{x}$ bound for $u-u_{\epsilon}$. Let $\{j_{n}\}_{n \in \mathbb{N}} \subset C^{2}(\mathbb{R})$ be a convex approximation of $| \cdot |$ with $|j_{n}'| \le C, ~j_{n}'' > 0$. For example, we may take $  j_{n}(z) := \sqrt{z^{2} + n^{-1}}-\sqrt{n^{-1}}$. Multiplying \eqref{u-u_epsEQN} by $j_{n}'(u-u_{\epsilon})$, we get
\begin{equation} \label{uL1preestimates}
    \begin{aligned}
        &\int_{\mathbb{R}} j_{n}(u-u_{\epsilon})(t)~dx-\int_{\mathbb{R}} j_{n}(u-u_{\epsilon})(0)~dx + \int_{0}^{t} \int_{\mathbb{R}} j_{n}''(u-u_{\epsilon}) \phi_{\epsilon}(v) |\partial_{x}(u-u_{\epsilon})|^{2} \\[1ex] =& \int_{0}^{t} \int_{\mathbb{R}} j_{n}'(u-u_{\epsilon}) \left( (\phi_{\epsilon}(v)-\phi_{\epsilon}(v_{\epsilon})) \partial_{x}^{2}u_{\epsilon} + (\partial_{x}\phi_{\epsilon}(v) - \partial_{x}\phi_{\epsilon}(v_{\epsilon}))\partial_{x}u_{\epsilon} \right).
    \end{aligned}
\end{equation}
Using $\partial_{x}u_{\epsilon} = -s\partial_{x}v_{\epsilon}$, \eqref{Ef}, \eqref{dxEf} and \eqref{dx2Vcomputation}, the right-hand side can be bounded by
\begin{equation*}
    \begin{aligned}
         &C\int_{0}^{t} \int_{\mathbb{R}} \phi_{\epsilon}(v_{\epsilon}) \partial_{x}V \frac{|\partial_{x}^{2}v_{\epsilon}|}{v_{\epsilon}-1} + C\int_{0}^{t} \int_{\mathbb{R}} \phi_{\epsilon}(v_{\epsilon}) \left( 
         \frac{\partial_{x}v_{\epsilon} |\partial_{x}V|}{(v_{\epsilon}-1)^{2}}  + \frac{|\partial_{x}^{2}V|}{v_{\epsilon}-1}
         \right) \partial_{x}v_{\epsilon}
   \\[1ex]  \le & C\|\sqrt{\phi_{\epsilon}}\partial_{x}V\|_{L^{2}_{t,x}} \left\|\sqrt{\phi_{\epsilon}}\frac{\partial_{x}^2 v_\epsilon}{v_\epsilon-1}\right\|_{L^{2}_{t,x}}  \\& +  C\|\sqrt{\phi_{\epsilon}}\partial_{x}v_{\epsilon}\|_{L^{2}_{t,x}} \left( \|\sqrt{\phi_{\epsilon}
   }\frac{\partial_{x}^{2}V}{v_{\epsilon}-1}\|_{L^{2}_{t,x}} + \left\|\frac{v_\epsilon'}{(v_\epsilon-1)^2}\right\|_{L^\infty_{t,x}}\|\sqrt{\phi_{\epsilon}}\partial_{x}V\|_{L^{2}_{t,x}}\right).
    \end{aligned} 
\end{equation*} By Lemma \ref{dxkveps},
\begin{equation*}
    \left\|\sqrt{\phi_{\epsilon}}\frac{\partial_{x}^2 v_\epsilon}{v_\epsilon-1}\right\|_{L^{2}_{t,x}} + \|\sqrt{\phi_{\epsilon}}\partial_{x}v_{\epsilon}\|_{L^{2}_{t,x}} \leq C(s,\gamma,v_+,\epsilon)t^{1/2},
\end{equation*} 
Hence we obtain a global bound of the form $C(\epsilon,\delta,s,v_+,\gamma)t^{1/2}$ for the right-hand side. Note that the second term on the left-hand side of \eqref{uL1preestimates} has a positive sign and therefore can be discarded. Returning to \eqref{uL1preestimates} and taking $n \to \infty$, we get
\begin{equation} \label{uL1postestimates}
    \begin{aligned}
        \|(u-u_{\epsilon})(t)\|_{L^{1}_{x}} \le \|(u-u_{\epsilon})(0)\|_{L^{1}_{x}} + C(\epsilon,\delta,s,v_+,\gamma)t^{1/2}.
    \end{aligned}
\end{equation} 

We now wish to find an estimate for $v-v_{\epsilon}$. From \eqref{lagrangian} we have that \begin{equation*}
    \partial_{t}(v-v_{\epsilon}) = \partial_{x}(u-u_{\epsilon}).
\end{equation*} Multiplying this equation by $j_{n}'(v-v_{\epsilon})$ where $j_{n}$ is defined as above and integrating in space, we have
\begin{equation} \label{v-veps_L1_preESTIMATE}
    \frac{d}{dt}\int_{\mathbb{R}} j_{n}(v-v_{\epsilon})~dx = \int_{\mathbb{R}} j_{n}'(v-v_{\epsilon}) \partial_{x}(u-u_{\epsilon})~dx \le  \int_{\mathbb{R}} |\partial_{x}(u-u_{\epsilon})|~dx.
\end{equation}Therefore, to obtain a $L^{1}(\mathbb{R})$ estimate for $v-v_{\epsilon}$ it is sufficient to bound $\partial_{x}(u-u_{\epsilon})$ in $L^{1}(\mathbb{R})$. Now, the evolution equation of $\partial_{x}(u-u_{\epsilon})$ can be expressed as 
\begin{align*}
    \partial_{t}\partial_{x}(u-u_{\epsilon}) - \partial_{x}(\phi_\epsilon(v)\partial_{x}^{2}(u-u_{\epsilon})) &= \partial_{x}\left( (\phi_\epsilon(v)-\phi_\epsilon(v_{\epsilon})) \partial_{x}^{2}u_{\epsilon} \right) + \partial_{x}[\partial_{x}(\phi_\epsilon(v)-\phi_\epsilon(v_\epsilon))\partial_{x}u] \\[1ex] &~~~+ \partial_{x}[\partial_{x}\phi_\epsilon(v_\epsilon)\partial_{x}(u-u_{\epsilon})].
\end{align*} Multiplying by $j_{n}'(\partial_{x}(u-u_{\epsilon}))$ and integrating by parts where appropriate,
\begin{equation} \label{dx(u-ueps)_preESTIMATES}
    \begin{aligned}
         &\int_{\mathbb{R}} j_{n}(\partial_{x}(u-u_{\epsilon}))(t)~dx  - \int_{\mathbb{R}} j_{n}(\partial_{x}(u-u_{\epsilon}))(0)~dx  + \int_{0}^{t} \int_{\mathbb{R}} j_{n}''(\partial_{x}(u-u_{\epsilon})) ~\phi_\epsilon(v) |\partial_{x}^{2}(u-u_{\epsilon})|^{2} \\[1ex] =& \int_{0}^{t} \int_{\mathbb{R}} j_{n}'(\partial_{x}(u-u_{\epsilon})) \partial_{x}\left((\phi_\epsilon(v)-\phi_\epsilon(v_{\epsilon}))\partial_{x}^{2}u_{\epsilon})\right)  +\int_{0}^{t} \int_{\mathbb{R}} j_{n}'(\partial_{x}(u-u_{\epsilon})) \partial_{x}[\partial_{x}(\phi_\epsilon(v)-\phi_\epsilon(v_\epsilon))\partial_{x}u]  \\[1ex] &+ \int_{0}^{t} \int_{\mathbb{R}} j_{n}'(\partial_{x}(u-u_{\epsilon})) \partial_{x} \left[ \partial_{x}\phi_\epsilon(v_\epsilon)\partial_{x}(u-u_{\epsilon} )\right] =: \sum_{n=1}^{3}I_{n}.
    \end{aligned}
\end{equation}
Firstly using \eqref{dxEf},
\begin{equation*}
    \begin{aligned}
        |I_{1}| &= \left|\int_{0}^{t} \int_{\mathbb{R}} j_{n}'(\partial_{x}(u-u_{\epsilon})) \left((\partial_{x}\phi_\epsilon(v)-\partial_{x}\phi_\epsilon(v_{\epsilon})) \partial_{x}^{2}u_{\epsilon} + (\phi_\epsilon(v)-\phi_\epsilon(v_{\epsilon}) \partial_{x}^{3}u_{\epsilon} \right)\right| \\[1ex] &\le C\int_{0}^{t} \int_{\mathbb{R}} |j_{n}'(\partial_{x}(u-u_{\epsilon}))| \phi_{\epsilon}(v_{\epsilon}) \left( |\partial_{x}^{2}v_{\epsilon}| \left( \frac{\partial_{x}v_{\epsilon}|\partial_{x}V|}{(v_{\epsilon}-1)^{2}}  + \frac{|\partial_{x}^{2}V|}{v_{\epsilon}-1} \right)   + \frac{|\partial_{x}^{3}v_{\epsilon}|}{v_{\epsilon}-1} |\partial_{x}V|     \right) \\[1ex] &\le \left\|\sqrt{\phi_{\epsilon}} \frac{\partial_{x}^{3}v_{\epsilon}}{v_{\epsilon}-1}\right\|_{L^{2}_{t,x}}\|\sqrt{\phi_{\epsilon}}\partial_{x}V\|_{L^{2}_{t,x}} + \left\|\sqrt{\phi_{\epsilon}} \frac{\partial_{x}^{2}v_{\epsilon}}{v_{\epsilon}-1}\right\|_{L^{2}_{t,x}}\left( \frac{C}{\epsilon} \|\sqrt{\phi_{\epsilon}}\partial_{x}V\|_{L^{2}_{t,x}} +  \|\sqrt{\phi_{\epsilon}}\partial_{x}^{2}V\|_{L^{2}_{t,x}} \right) \\
        &\leq C(\epsilon,\delta,s,v_+,\gamma)t^{1/2}.
    \end{aligned}
\end{equation*} 

Then 
\begin{equation*}
    I_2 =  \int_0^t\int_{\mathbb{R}}j_n'(\partial_{x}(u-u_{\epsilon}))\partial_x^2[\phi_\epsilon(v)-\phi_\epsilon(v_\epsilon)]\partial_x u + \int_0^t\int_{\mathbb{R}}j_n'(\partial_{x}(u-u_{\epsilon}))\partial_x[\phi_\epsilon(v)-\phi_\epsilon(v_\epsilon)]\partial_x^2 u =: K_1 +K_2.
\end{equation*}
Using \eqref{dx2Ef}, 
\begin{align*}
    |K_1|\leq & C\int_0^t\int_{\mathbb{R}}\phi_{\epsilon}(v_{\epsilon})\left( \frac{(\partial_{x}v_\epsilon)^{2}|\partial_x V|}{(v_{\epsilon}-1)^{3}} + \frac{\partial_{x}v_{\epsilon} |\partial_x^2 V|}{(v_{\epsilon}-1)^2} + \left| \frac{\partial_{x}^2V}{v_{\epsilon}-1} \right|^{2} + \frac{|\partial_{x}^{3}V|}{v_{\epsilon}-1}   \right)|\partial_x u| \\
    \leq & C\|\sqrt{\phi_\epsilon}\partial_x u\|_{L^2_{t,x}}\left(\|\sqrt{\phi_\epsilon}\partial_x V\|_{L^2_{t,x}}+ \left\|\sqrt{\phi_\epsilon}\frac{\partial_x^2V }{v_\epsilon-1} \right\|_{L^2_{t,x}}+ \left\|\sqrt{\phi_\epsilon}\frac{\partial_x^3V }{v_\epsilon-1} \right\|_{L^2_{t,x}} \right. \\
    & \left.+\left\|\sqrt{\phi_\epsilon}\frac{\partial_x^2V}{v_\epsilon-1}\right\|_{L^2L^\infty}\left\|\frac{\partial_x^2V}{v_\epsilon-1}\right\|_{L^\infty L^2} \right),
\end{align*}
and the triangle inequality 
\begin{equation*}
    \|\sqrt{\phi_\epsilon}\partial_x u\|_{L^2_{t,x}} \leq \|\sqrt{\phi_\epsilon}\partial_x (u-u_\epsilon)\|_{L^2_{t,x}} + \|\sqrt{\phi_\epsilon}\partial_x u_\epsilon\|_{L^2_{t,x}}
\end{equation*}
enables to bound $K_1$. For $K_2$, \eqref{dxEf} yields
\begin{align*}
    |K_2| \leq & C\int_0^t\int_{\mathbb{R}}\phi_{\epsilon}(v_{\epsilon})\left(   \frac{\partial_{x}v_{\epsilon}|\partial_x V|}{(v_{\epsilon}-1)^{2}}   +  \frac{|\partial_{x}^2 V|}{v_{\epsilon}-1}  \right)|\partial_x^2 u| \\
    \leq & \left(\left\|\sqrt{\phi_\epsilon}\partial_x V\right\|_{L^2_{t,x}}+\left\|\sqrt{\phi_\epsilon}\frac{\partial_x^2 V}{v_\epsilon-1}\right\|_{L^2_{t,x}}\right)\left(\left\|\sqrt{\phi_\epsilon}\partial_x^2(u-u_\epsilon)\right\|_{L^2_{t,x}}+\left\|\sqrt{\phi_\epsilon}\partial_x^2 u_\epsilon\right\|_{L^2_{t,x}}\right).
\end{align*}
For $I_3$, we decompose 
\begin{align*}
    |I_3| = &C\left|\int_0^t\int_{\mathbb{R}}j_n'\left[\left(\phi_\epsilon''(v_\epsilon)(\partial_x v_\epsilon)^2+\phi_\epsilon'(v_\epsilon)\partial_x^2v_\epsilon\right)\partial_x(u-u_\epsilon) + \phi_\epsilon'(v_\epsilon)\partial_x v_\epsilon \partial_x^2(u-u_\epsilon) \right]\right| \\
    \leq & C\left[\left\|\frac{\phi_\epsilon''(v_\epsilon)v_\epsilon'}{\phi_\epsilon(v_\epsilon)}\right\|_{L^\infty_{t,x}}\|\sqrt{\phi_\epsilon}\partial_x v_\epsilon\|_{L^2_{t,x}}+\left\|\frac{\phi_\epsilon'(v_\epsilon)(v_\epsilon-1)}{\phi_\epsilon(v_\epsilon)}\right\|_{L^\infty_{t,x}}\left\|\sqrt{\phi_\epsilon}\frac{\partial_x^2 v_\epsilon}{v_\epsilon-1}\right\|_{L^2_{t,x}}\right]\|\sqrt{\phi_\epsilon}\partial_x(u-u_\epsilon)\|_{L^2_{t,x}} \\
    &+ C\left\|\frac{\phi_\epsilon'(v_\epsilon)(v_\epsilon-1)}{\phi_\epsilon(v_\epsilon)}\right\|_{L^\infty_{t,x}}\|\sqrt{\phi_\epsilon}\partial_x v_\epsilon\|_{L^2_{t,x}}\left\|\sqrt{\phi_\epsilon}\frac{\partial_x^2(u-u_\epsilon)}{v_\epsilon-1}\right\|_{L^2_{t,x}}.
\end{align*}
\color{black}
Passing to the limit $n \to \infty$ with the help of Fatou's lemma gives us 
 \begin{equation}
     \|\partial_{x}(u-u_{\epsilon})(t)\|_{L^{1}_{x}} \le  \|\partial_{x}(u-u_{\epsilon})(0)\|_{L^{1}_{x}} +C(\epsilon,\delta,s,v_+,\gamma)t^{1/2}.
 \end{equation}
We have now shown that $(u-u_{\epsilon})(t), ~(v-v_{\epsilon})(t)$ belong to $L^{1}(\mathbb{R})$ for any $t>0$. Note that since the equations for both quantities are conservative, we actually have that $(u-u_{\epsilon})(t), ~(v-v_{\epsilon})(t) \in L^{1}_{0}(\mathbb{R})$. This marks the end of the proof of Proposition \ref{PROPstabilitylongtime}.
 \subsection{Concluding the proofs of Theorem \ref{THMstability} and Corollary \ref{CORstability}}
We are now in a position to justify the equivalence between the original and integrated systems, which will allow us to conclude the proof of Theorem \ref{THMstability}. Firstly, consider the original system with some initial data $(u_{0}, v_{0})$ which satisfies the assumptions of Theorem \ref{THMstability}, and let $(V,W)$ be the corresponding solution of the integrated system \eqref{ARintegrated}. Since we defined $(u,v) := (u_{\epsilon} + \partial_{x}U, v_{\epsilon} + \partial_{x}V)$, it follows that $(u,v)$ is a solution to the original system. Moreover, Lemma \ref{LEMMAexistence-u-u_eps} tells us that $\partial_{x}U \in C([0,T]; H^{1}(\mathbb{R}))$ and Theorem \ref{GWPintegrated} tells us that $\partial_{x}V \in C([0,T];H^{1}(\mathbb{R})) \cap L^{2}(0,T; H^{2}(\mathbb{R}))$. Thus, $(u,v) \in (u_{\epsilon}, v_{\epsilon}) + (C([0,T];H^{1}(\mathbb{R})))^{2}$ solves the original system and using Proposition \ref{PROPstabilitylongtime} we have that $(u,v)(t) \in L^{1}_{0}(\mathbb{R})$ for any $t \in (0,T]$.
It remains to show that the solution is unique. To this end, suppose that $(u,v)$ is another solution to the original system with initial data $(u_{0}, v_{0})$ satisfying the assumptions of Theorem \ref{THMstability}. Then we may write $(u,v) = (u_{\epsilon}, v_{\epsilon}) + (f,g)$ for some $f,g \in C([0,T]; H^{1}(\mathbb{R}) \cap L^{1}_{0}(\mathbb{R}))$. Now, define the integrated quantities
\begin{equation*}
    U(t,x) := \int_{-\infty}^{x} (u(t,z)-u_{\epsilon}(t,z))~dz, \quad V(t,x) := \int_{-\infty}^{x} (v(t,z)-v_{\epsilon}(t,z))~dz,
\end{equation*}
and \begin{equation*}
    W := U + \phi_{\epsilon}(v) - \phi_{\epsilon}(v_{\epsilon}).
\end{equation*} Then $(V,W)$ is a solution of the integrated system \eqref{ARintegrated}. Furthermore, thanks to the hypotheses on the initial data $(u,v)$ we have that $\partial_{x}V \in C([0,T]; H^{1}(\mathbb{R}) \cap L^{1}_{0}(\mathbb{R})) \cap L^{2}(0,T; H^{2}(\mathbb{R}))$. Therefore $V \in \mathcal{X}$. Moreover, due to the assumptions on the initial data, we in fact have $V \in B_{\delta}$ and so since $\mathcal{A}^{\epsilon}$ is a contraction on $B_{\delta}$ we easily deduce that $V$ is uniquely determined by the initial data. Since $W$ is constant in time, we conclude that the pair $(W,V)$ is unique and therefore so is $(u,v)$.

Finally, let us explain how Corollary \ref{CORstability} is obtained. With the additional assumption \eqref{CORassumption}, the inequality \eqref{k=2postestimates} becomes
\begin{equation} \label{CORk=2postestimates}
    \begin{aligned}
        \sup_{t \in [0,T]}  &\left( \sum_{k=0}^{2}  c_{k}\epsilon^{2k}\left[ E_{k}(t;V_{2}) + \int_{0}^{t}D_{k}(\tau;V_{2})~d\tau \right] \right)  \\[1ex] &\le C( E_{0}(0; V_{2}) + \epsilon^{2}E_{1}(0;V_{2}) + \epsilon^{4}E_{2}(0;V_{2})) + \frac{C}{\epsilon^{3/2}} \|V_{1}\|_{\mathcal{X}}^{2}\|V_{2}\|_{\mathcal{X}},
    \end{aligned}
\end{equation} and so we may take $T = + \infty$ which gives us a solution defined on $\mathbb{R}_{+} \times \mathbb{R}$. Then using \eqref{GNinf} we have that there exists $C > 0$ independent of time such that
\begin{equation}
       \|(v-v_{\epsilon})(t)\|_{L^{\infty}_{x}} \le  C \|(v-v_{\epsilon})(t)\|_{L^{2}_{x}}^{1/2} \|\partial_{x}(v-v_{\epsilon})(t)\|_{L^{2}_{x}}^{1/2},
    \end{equation} where the right hand-side tends to $0$ as $t \to \infty$ since $V \in L^{2}(\mathbb{R}_{+}; H^{2}(\mathbb{R}))$. This shows that $(v-v_{\epsilon})(t) \to 0$ as $t \to \infty$. Similarly, the bounds obtained in Lemma \ref{LEMMAexistence-u-u_eps} imply that \begin{equation*}
        \|(u-u_{\epsilon})(t)\|_{H^{1}_{x}} \to 0 \text{ as } t \to \infty
    \end{equation*} and so we also find that $(u-u_{\epsilon})(t) \to 0$ as $t \to \infty$, which completes the proof.

\appendix

\section{Proof that the map \texorpdfstring{$\mathcal{A}^\epsilon$}{Lg} is well defined}

\label{appendix}

We give here a proof that the map $\mathcal{A}^\epsilon$ of Section \ref{constructglobalsolution} is well defined, i.e. we show that Equation \eqref{linearsystem}:
\begin{equation*}
        \left\{
    \begin{aligned}
        &\partial_t V_{2} -\partial_x W_0 -\partial_x(\phi_\epsilon(v_\epsilon)\partial_xV_{2})=\partial_x H(\partial_xV_{1}),\\
        &V_{2}|_{t=0} = V_{0},
    \end{aligned}
    \right.
\end{equation*}
is well posed. Equation \eqref{linearsystem} rewrites as 

\begin{equation}
\label{system2}
        \left\{
    \begin{aligned}
        &\partial_t V_{2} =\partial_x(\phi_\epsilon \partial_x V_2) +f,\\
        &V_{2}|_{t=0} = V_{0},
    \end{aligned}
    \right.
\end{equation}
where $f=\partial_x H(\partial_x V_1)+\partial_x W_0$. Since the operator $\partial_x(\phi_\epsilon\partial_x\cdot)$ is uniformly elliptic, one can hope that it is the infinitesimal generator of an analytic evolution operator $(U(t,\tau))_{t,\tau\geq 0}$. The unique solution of \eqref{system2} is then given by 
\begin{equation*}
V_2(t) = U(t,0)V_0 + \int_0^t U(t,\tau)f(\tau)d\tau.
\end{equation*}
The proof that $\partial_x(\phi_\epsilon\partial_x\cdot)$ generates an analytic evolution operator is classical. Note that one has to be careful because the coefficient $\phi_\epsilon = \phi_\epsilon(t,x)$ depends on time and is unbounded. In order to simplify the situation, we take advantage of the fact that $\phi_\epsilon$ is a function of the variable $\xi:= x-st$. A change of variable $(t,x)\mapsto (t,\xi)\in (0,T)\times \mathbb{R}$ then yields that \eqref{system2} can be written as 
\begin{equation*}
        \left\{
    \begin{aligned}
        &\partial_t V_{2} =\partial_\xi(\phi_\epsilon(\xi) \partial_\xi V_2) + s\partial_\xi V_2 +f,\\
        &V_{2}|_{t=0} = V_{0},
    \end{aligned}
    \right.
\end{equation*}
Hence in these coordinates, the operator $\partial_\xi(\phi_\epsilon(\xi)\partial_\xi\cdot)+s\partial_\xi$ is independant of time. Consider the Hilbert space $H=L^2(\mathbb{R},\mathbb{C})$, and define the unbounded operator 
\begin{equation*}
T_c V:= \phi_\epsilon^{1/2}\partial_\xi V, \qquad \mathrm{with}\qquad D(T_c) = C^\infty_c(\mathbb{R}).
\end{equation*}
Let $T$ be the closure of $T_c$, i.e. 
\begin{equation*}
D(T)=  \{V\in H,\phi_\epsilon^{1/2}\partial_\xi V\in H, \exists  V_n \in C^\infty_c(\mathbb{R}), V_n \rightarrow V, \phi_\epsilon^{1/2}\partial_\xi V_n \rightarrow \phi_\epsilon^{1/2}\partial_\xi V\}.
\end{equation*} 
The adjoint $T^*$ of $T$ is given by 
\begin{equation*}
D(T^*) = \{V\in H, \partial_\xi(\phi_\epsilon^{1/2}V) \in H\},\qquad\mathrm{with} \qquad T^*V = -\partial_\xi(\phi_\epsilon^{1/2}V).
\end{equation*}
Indeed, if $V\in D(T^*)$, then $T^*V = -\partial_\xi(\phi_\epsilon^{1/2}V)$ in the sense of distributions. Conversely, if $V$ and  $\partial_\xi (\phi_\epsilon^{1/2}V)\in H$, then $<-\partial_\xi(\phi_\epsilon^{1/2}V),\psi>_H = <V, T\psi>_H$ for every $\psi\in C^\infty_c(\mathbb{R})$, and the same is true for every $\psi \in D(T)$ by density.
 
Since $T$ is closed and densely defined, $T^*T$ defined by 
\begin{equation*}
D(T^*T) = \{V\in D(T), TV \in D(T^*) \}, \qquad T^*T V = T^*(TV) =-\partial_\xi(\phi_\epsilon(\xi)\partial_x V)
\end{equation*}
 is self-adjoint and positive (see Theorem 13.13 in \cite{rudin1991functional}, Chapter 13). It follows that $-T^* T$ is the infinitesimal generator of an analytic semigroup. Now define the closed operator $B:= s\partial_\xi$, with $D(B) = H^1(\mathbb{R})$. Since there exists $\alpha > 0$ such that $\phi_\epsilon \geq \alpha$, we deduce that $D(T^* T)\subset D(T) \subset D(B)$. Furthermore, for every $\epsilon >0$, for every $V\in D(T^* T)$, one has that 
 \begin{align*}
 \|BV\| \leq \frac{s}{\alpha}\|TV\| = &\frac{s}{\alpha} |<V,T^* TV>_H|^{1/2} \leq \frac{s}{\alpha}\|V\|^{1/2}\|T^* TV\|^{1/2} \\
 \leq &\epsilon \|T^* TV\| + \frac{s^2}{4\alpha^2\epsilon}\|V\|,
 \end{align*}
by Young's inequality. By Theorem 12.37 in \cite{renardy2006introduction}, it follows that $-T^*T+B$ is the infinitesimal generator of an analytic semigroup.

Note that in fact it is possible to prove that  
\begin{equation*}
 D(T) = \{ V\in H, \phi_\epsilon^{1/2}\partial_x V \in H\}.
\end{equation*}
Indeed, let $\theta$ a smooth function such that $0\leq \theta\leq 1$, $\theta|_{[-\frac{1}{3},\frac{1}{3}]}\equiv 1$, $\mathrm{supp}(\theta)\subset[-1,1]$ and $\int_{\mathbb{R}}\theta=1$. For $V\in H$ with $\phi_\epsilon^{1/2}\partial_\xi V\in H$, define 
\begin{equation*}
V_n(\xi) := \theta\left(\frac{\xi}{n}\right) (\rho_n * V)(\xi),
\end{equation*}
where $\rho_n := n\theta(n\cdot)\in C^\infty_c(\mathbb{R})$ is an approximation of unity. Then $V_n\in C_c^\infty(\mathbb{R})$ and $V_n\rightarrow V$ in $L^2$. Furthermore,
\begin{equation*}
\phi_\epsilon^{1/2}\partial_\xi V_n = \frac{\phi_\epsilon^{1/2}}{n}\theta'\left(\frac{\xi}{n}\right) (\rho_n * V)(\xi) + \phi_\epsilon^{1/2}\theta\left(\frac{\xi}{n}\right) (\rho_n * \partial_\xi V)(\xi).
\end{equation*}
Since $\phi_\epsilon^{1/2}(\xi)\leq C(1+|\xi|^{1/2+1/(2\gamma)})$ and $\theta$ is compactly supported, 
\begin{equation*}
\left|\frac{\phi_\epsilon^{1/2}}{n}\theta'\left(\frac{\xi}{n}\right)\right|\leq Cn^{1/(2\gamma)- 1/2}\mathbf{1}_{n/3\leq |\xi|\leq n}.
\end{equation*}
Hence 
\begin{equation*}
\frac{\phi_\epsilon^{1/2}}{n}\theta'\left(\frac{\xi}{n}\right) (\rho_n * V)(\xi)\rightarrow 0 \qquad\mathrm{and}\qquad \phi_\epsilon^{1/2}\theta\left(\frac{\xi}{n}\right) (\rho_n * \partial_\xi V)(\xi)\rightarrow \phi_\epsilon^{1/2}\partial_\xi V.
\end{equation*} 
Hence $V\in D(T)$.

\section{Proof of Lemma \ref{linearoperatorestimates}}

\label{proofoflemma}
We give here the proof of Lemma \ref{linearoperatorestimates}.
\begin{proof}
We first do the computations of the estimate for $\mathcal{L}_\epsilon$. We compute that 
\begin{align*}
    \int_0^t\int_{\mathbb{R}}\mathcal{L}_\epsilon(\eta)\eta  =& \int_{0}^{t}\int_{\mathbb{R}} \frac{s\eta^{2}v_{\epsilon}'}{v_{\epsilon}-1}~dxd\tau + \int_{0}^{t}\int_{\mathbb{R}} \frac{\phi_{\epsilon}(v_{\epsilon})v_{\epsilon}'}{v_{\epsilon}-1} \eta \left( \partial_{x}\eta  + \frac{\eta v_{\epsilon}'}{v_{\epsilon}-1}\right)~dxd\tau \\
    &+ \int_{0}^{t} \int_{\mathbb{R}} \eta  \partial_{x} \left( \frac{\eta  \phi_{\epsilon}(v_{\epsilon})v_{\epsilon}'}{v_{\epsilon}-1}\right)~dxd\tau  + \int_{0}^{t} \int_{\mathbb{R}} \frac{\eta }{v_{\epsilon}-1} \partial_{x}( \phi_{\epsilon}'(v_{\epsilon})v_{\epsilon}'(v_{\epsilon}-1)\eta )~dxd\tau =:\sum_{k=1}^4 I_k.
\end{align*}
Let us fix $\alpha\in (0,1)$. For $I_1$, we exploit the smallness of $v_{\epsilon}'$ (i.e. \eqref{v'}) to get
\begin{equation*}
    |I_{1}| = \left|\int_{0}^{t}\int_{\mathbb{R}} \frac{s\eta^{2}v_{\epsilon}'}{v_{\epsilon}-1}~dxd\tau \right| \le  \left\| \frac{s v_{\epsilon}'}{\phi_{\epsilon}(v_{\epsilon}-1)^{3}}\right\|_{L^{\infty}_{t,x}}  \int_{0}^{t} \int_{\mathbb{R}} |\sqrt{\phi_{\epsilon}}\partial_{x}V |^{2} \le \frac{C}{\epsilon^{2}} \int_{0}^{t}D_{0}(\tau; V )~d\tau.
\end{equation*} Next,
\begin{equation*}
    \begin{aligned}
        |I_{2}| & = \left| \int_{0}^{t}\int_{\mathbb{R}} \frac{v_{\epsilon}'\phi_{\epsilon}}{v_{\epsilon}-1} \eta \left( \partial_{x}\eta + \frac{\eta v_{\epsilon}'}{v_{\epsilon}-1}\right)~dxd\tau  \right| \\
        &\le \int_{0}^{t} \int_{\mathbb{R}} \left| \frac{v_{\epsilon}' \phi_{\epsilon}}{v_{\epsilon}-1} \right| |\eta| |\partial_{x}\eta| + \int_{0}^{t} \int_{\mathbb{R}}  \left| \frac{v_{\epsilon}' \phi_{\epsilon}}{v_{\epsilon}-1}  \frac{v_{\epsilon}' \partial_{x}V }{(v_{\epsilon}-1)^{2}} \frac{\partial_{x}V }{v_{\epsilon}-1} \right| \\[1ex] &\le \left\| \frac{v_{\epsilon}'}{(v_{\epsilon}-1)^{2}}\right\|_{L^{\infty}_{t,x}} \| \sqrt{\phi_{\epsilon}}\partial_{x}V \|_{L^{2}_{t,x}} \|\sqrt{\phi_{\epsilon}}\partial_{x}\eta\|_{L^{2}_{t,x}} + \left\|\frac{(v_{\epsilon}')^{2}}{(v_{\epsilon}-1)^{4}}\right\|_{L^{\infty}_{t,x}} \|\sqrt{\phi_{\epsilon}} \partial_{x}V \|_{L^{2}_{t,x}}^{2} \\[1ex] &\le \frac{C}{\epsilon} \left( \int_{0}^{t}D_{0}(\tau; V )~d\tau \right)^{\frac{1}{2}}\left( \int_{0}^{t}D_{1}(\tau; V )~d\tau\right)^{\frac{1}{2}} + \frac{C}{\epsilon^{2}} \int_{0}^{t} D_{0}(\tau;V )~d\tau \\[1ex] &\le \frac{\alpha}{3} \int_{0}^{t} D_{1}(\tau; V )~d\tau + \frac{C}{\alpha\epsilon^{2}} \int_{0}^{t} D_{0}(\tau; V )~d\tau,
    \end{aligned} 
\end{equation*} 
where we used Young's inequality for the last line. Similarly,
\begin{equation*}
    \begin{aligned}
        |I_{3}| & = \left| \int_{0}^{t} \int_{\mathbb{R}} \eta \partial_{x} \left( \frac{\eta \phi_{\epsilon}(v_{\epsilon})v_{\epsilon}'}{v_{\epsilon}-1}\right)~dxd\tau \right| \\
        &\le \int_{0}^{t} \int_{\mathbb{R}} |\partial_{x}\eta| \left| \frac{\eta v_{\epsilon}'\phi_{\epsilon}}{v_{\epsilon}-1}\right| \le \left\| \frac{v_{\epsilon}'}{(v_{\epsilon}-1)^{2}} \right\|_{L^{\infty}_{t,x}} \|\sqrt{\phi_{\epsilon}}\partial_{x}\eta\|_{L^{2}_{t,x}} \|\sqrt{\phi_{\epsilon}}\partial_{x}V \|_{L^{2}_{t,x}} \\[1ex] &\le \frac{\alpha}{3} \int_{0}^{t} D_{1}(\tau;V )~d\tau + \frac{C}{\alpha\epsilon^{2}} \int_{0}^{t}D_{0}(\tau;V )~d\tau.
    \end{aligned}
\end{equation*}
To deal with the last term, we first integrate by parts to get \begin{equation*}
    I_{4} = \int_{0}^{t} \int_{\mathbb{R}} \frac{\eta}{v_{\epsilon}-1} \partial_{x}( \phi_{\epsilon}'(v_{\epsilon})v_{\epsilon}'(v_{\epsilon}-1)\eta )~dxd\tau =- \int_{0}^{t} \int_{\mathbb{R}} \frac{\partial_{x}\eta}{v_{\epsilon}-1} v_{\epsilon}' \phi_{\epsilon}' \partial_{x}V  + \int_{0}^{t} \int_{\mathbb{R}} \frac{(v_{\epsilon}')^{2} \eta}{(v_{\epsilon}-1)^{2}} \phi_{\epsilon}' \partial_{x}V .
\end{equation*} Then in the same spirit as the previous estimates and with the estimate \eqref{phiderivatives} on $\phi_\epsilon'$,
\begin{equation*}
    \begin{aligned}
        |I_{4}| &\le \left\| \frac{v_\epsilon'}{(v_\epsilon-1)^2}\right\|_{L^{\infty}_{t,x}} \|\sqrt{\phi_{\epsilon}}\partial_{x}\eta\|_{L^{2}_{t,x}} \|\sqrt{\phi_{\epsilon}}\partial_{x}V \|_{L^{2}_{t,x}} + \left\| \frac{(v_{\epsilon}')^{2}}{(v_{\epsilon}-1)^{4}}\right\|_{L^{\infty}_{t,x}}\|\sqrt{\phi_{\epsilon}}\partial_{x}V \|_{L^{2}_{t,x}}^{2} \\[1ex] &\le \frac{\alpha}{3} \int_{0}^{t} D_{1}(\tau; V )~d\tau + \frac{C}{\alpha\epsilon^{2}} \int_{0}^{t}D_{0}(\tau;V )~d\tau.
    \end{aligned} 
\end{equation*} 
We finally put all terms together to obtain that 
\begin{equation*}
    \left|\int_0^t\int_{\mathbb{R}}\mathcal{L}_\epsilon(\eta)\eta\right|\leq \sum_{k=1}^4 |I_k| \leq \alpha\int_0^t D_1(\tau;V)d\tau + \frac{C}{\alpha\epsilon^2}\int_0^t D_0(\tau; V)d\tau.
\end{equation*}

We now move on to the proof of the second estimate. Again, we fix $\alpha\in (0,1)$ and compute  that
\begin{equation}
    \begin{aligned}
        \int_{0}^{t} \int_{\mathbb{R}}\partial_{x}\eta  ~ \mathcal{C}_\epsilon (\eta ) =& -\int_{0}^{t} \int_{\mathbb{R}} \partial_{x}\phi_{\epsilon} \partial_{x}^{2}\eta  \partial_{x}\eta  +  \int_{0}^{t} \int_{\mathbb{R}} \eta  \partial_{x} \eta  \partial_{x} \left( \frac{sv_{\epsilon}'}{v_{\epsilon}-1} \right) \\[1ex] & +\int_0^t\int_{\mathbb{R}}(\partial_x \eta)^2\partial_x\left(\phi_\epsilon\frac{v_\epsilon'}{v_\epsilon-1}\right) +  \int_{0}^{t} \int_{\mathbb{R}} \eta  \partial_{x}\eta  \partial_{x}  \left(\frac{(v_{\epsilon}')^{2}\phi_{\epsilon}(v_{\epsilon})}{(v_{\epsilon}-1)^{2}}\right) \\[1ex] &+  \int_{0}^{t} \int_{\mathbb{R}} \partial_{x}\eta  \partial_{x} \left( \eta  \partial_{x}\left( \frac{v_{\epsilon}'\phi_{\epsilon}(v_{\epsilon})}{v_{\epsilon}-1}\right)\right) - \int_{0}^{t} \int_{\mathbb{R}} \partial_{x}\eta  \frac{v_{\epsilon}' \partial_{x}(\phi_{\epsilon}'(v_{\epsilon}) v_{\epsilon}'(v_{\epsilon}-1) \eta )}{(v_{\epsilon}-1)^{2}}  \\[1ex] &+ \int_{0}^{t} \int_{\mathbb{R}} \frac{\partial_{x}\eta }{v_{\epsilon}-1} \partial_{x}(\partial_{x}(v_{\epsilon}'\phi_{\epsilon}'(v_{\epsilon})(v_{\epsilon}-1))\eta ) =: \sum_{n=1}^{7}K_{n}.
    \end{aligned}
\end{equation}
We can estimate each of these terms using the same strategy as our estimates for $I_{1}-I_{4}$. This leads to
\begin{equation*}
\begin{aligned}
    &|K_{1}|\leq \left\|\frac{\partial_x\phi_\epsilon}{\phi_\epsilon}\right\|_{L^\infty_{t,x}}\|\sqrt{\phi_\epsilon}\partial_x^2\eta\|_{L^2_{t,x}}\|\sqrt{\phi_\epsilon}\partial_x\eta\|_{L^2_{t,x}} \le  \frac{\alpha}{6}  \int_{0}^{t} D_{2}(\tau;V )~d\tau + \frac{C}{\alpha\epsilon^{2}} \int_{0}^{t} D_{1}(\tau;V )~d\tau, 
    \\[1ex]&|K_{2}| =\left|\int_0^t\int_{\mathbb{R}}\frac{sv_\epsilon'}{v_\epsilon-1}\left((\partial_x\eta)^2+\eta\partial_x^2\eta\right)\right|\le \frac{\alpha}{6}  \int_{0}^{t} D_{2}(\tau;V )~d\tau + \frac{C}{\epsilon^{2}} \int_{0}^{t} D_{1}(\tau;V )~d\tau  + \frac{C}{\alpha\epsilon^{4}} \int_{0}^{t}D_{0}(\tau;V )~d\tau, 
    \\[1ex] &|K_3| =2\left|\int_0^t \phi_\epsilon \frac{v_\epsilon'}{v_\epsilon-1} \partial_x^2\eta \partial_x\eta \right|\leq \frac{\alpha}{6} \int_{0}^{t} D_{2}(\tau;V )~d\tau +\frac{C}{\alpha\epsilon^2}\int_0^tD_1(\tau;V)d\tau,
    \\[1ex] &|K_{4} | = \left|\int_0^t\int_{\mathbb{R}}\frac{(v_\epsilon')^2\phi_\epsilon}{(v_\epsilon-1)^2}\left(\eta\partial_x^2\eta+(\partial_x\eta)^2\right)\right| \le \frac{\alpha}{6} \int_{0}^{t} D_{2}(\tau;V )~d\tau + \frac{C}{\epsilon^{2}}\int_{0}^{t} D_{1}(\tau;V )~d\tau + \frac{C}{\epsilon^{4}} \int_{0}^{t}D_{0}(\tau;V )~d\tau, 
    \\[1ex]
    &|K_{5}|=\left|\int_0^t\int_{\mathbb{R}}\partial_x\left(\phi_\epsilon \frac{v_\epsilon'}{v_\epsilon-1}\right) \eta \partial_x^2\eta \right|,
\end{aligned}
\end{equation*} 
and we compute with the estimate \eqref{phiderivatives} and Lemma \ref{vepsestimate} that 
\begin{equation*}
    \left|\partial_x\left(\phi_\epsilon \frac{v_\epsilon'}{v_\epsilon-1}\right)\right| \leq \frac{C}{\epsilon^2}\phi_\epsilon (v_\epsilon-1)^{2\gamma} \leq \frac{C}{\epsilon^2}\phi_\epsilon,
\end{equation*}
hence 
\begin{equation*}
    |K_5|\leq  \frac{\alpha}{6}  \int_{0}^{t} D_{2}(\tau;V )~d \tau+ \frac{C}{\alpha\epsilon^{4}} \int_{0}^{t}D_{0}(\tau;V )~d\tau.
\end{equation*}
Next,
\begin{equation*}
    K_6 = - \int_{0}^{t} \int_{\mathbb{R}} (\partial_{x}\eta)^2  \frac{v_{\epsilon}' \phi_{\epsilon}'(v_{\epsilon}) v_{\epsilon}'(v_{\epsilon}-1) }{(v_{\epsilon}-1)^{2}} -\int_{0}^{t} \int_{\mathbb{R}} \eta \partial_{x}\eta  \frac{v_{\epsilon}' \partial_{x}(\phi_{\epsilon}'(v_{\epsilon}) v_{\epsilon}'(v_{\epsilon}-1) )}{(v_{\epsilon}-1)^{2}}.
\end{equation*}
With the estimate \eqref{phiderivatives} and Lemma \ref{vepsestimate}, we see that 
\begin{equation}
\label{auxiliary}
    \left|\partial_x[\phi_\epsilon' v_\epsilon'(v_\epsilon-1)]\right|\leq \frac{C}{\epsilon}(v_\epsilon-1)^\gamma.
\end{equation}
Hence 
\begin{equation*}
    |K_6| \leq  \frac{C}{\epsilon^{2}}\int_{0}^{t} D_{1}(\tau;V )~d\tau + \frac{C}{\epsilon^{4}} \int_{0}^{t}D_{0}(\tau;V )~d\tau .
\end{equation*}
To estimate $K_{7}$, we first integrate by parts to obtain
\begin{equation*}
    \begin{aligned}
        K_{7} =& - \int_{0}^{t} \int_{\mathbb{R}} \frac{\eta\partial_{x}^{2}\eta }{v_{\epsilon}-1} \partial_{x}[\phi_\epsilon' v_\epsilon'(v_\epsilon-1)] + \int_{0}^{t} \int_{\mathbb{R}} \frac{\eta_2\partial_{x}\eta }{(v_{\epsilon}-1)^{2}} v_{\epsilon}' \partial_{x}[\phi_\epsilon' v_\epsilon'(v_\epsilon-1)].
    \end{aligned}
\end{equation*} 

Then using again \eqref{auxiliary} and the Holder and Young inequalities,
\begin{equation}
\begin{aligned}
    |K_{7}| &\le \frac{C}{\epsilon}\|\sqrt{\phi_{\epsilon}}\partial_{x}V \|_{L^{2}_{t,x}} \left(\left\|\frac{(v_\epsilon-1)^{\gamma-2}}{\phi_\epsilon}\right\|_{L^\infty_{t,x}}\|\sqrt{\phi_{\epsilon}}\partial_{x}^{2}\eta \|_{L^{2}_{t,x}}  + \left\|\frac{v_\epsilon'(v_\epsilon-1)^{\gamma-3}}{\phi_\epsilon}\right\|_{L^\infty_{t,x}} \|\sqrt{\phi_{\epsilon}}\partial_{x}\eta \|_{L^{2}_{t,x}}\right) \\[1ex] 
    &\le \frac{C}{\epsilon^{2}}\int_{0}^{t} D_{1}(\tau;V )~d\tau + \frac{C}{\alpha\epsilon^{4}} \int_{0}^{t}D_{0}(\tau;V )~d\tau + \frac{\alpha}{6} \int_{0}^{t} D_{2}(\tau;V )~d\tau.
\end{aligned}
\end{equation} 
Summing $|K_1|,\dots, |K_7|$ finally yields the desired estimate.

\end{proof}

\bibliographystyle{abbrv}
\bibliography{ref}

\begin{thebibliography}{10}

\bibitem{aceves2023pedestrian}
P.~Aceves-S{\'a}nchez, R.~Bailo, P.~Degond, and Z.~Mercier.
\newblock Pedestrian models with congestion effects.
\newblock {\em HAL preprint: hal-04334055}, 2024.

\bibitem{aw2000resurrection}
A.~Aw and M.~Rascle.
\newblock Resurrection of second order models of traffic flow.
\newblock {\em SIAM J. Appl. Math}, 60(3):916--938, 2000.

\bibitem{berthelin2002existence}
F.~Berthelin.
\newblock Existence and weak stability for a pressureless model with unilateral constraint.
\newblock {\em Mathematical Models and Methods in Applied Sciences}, 12(02):249--272, 2002.

\bibitem{boudin2000solution}
L.~Boudin.
\newblock A solution with bounded expansion rate to the model of viscous pressureless gases.
\newblock {\em SIAM Journal on Mathematical Analysis}, 32(1):172--193, 2000.

\bibitem{boyer_mathematical_2012}
F.~Boyer and P.~Fabrie.
\newblock {\em Mathematical {Tools} for the {Study} of the {Incompressible} {Navier}-{Stokes} {Equations} {and Related} {Models}}.
\newblock Applied {Mathematical} {Sciences}. Springer New York, 2012.

\bibitem{bresch2014singular}
D.~Bresch, C.~Perrin, and E.~Zatorska.
\newblock Singular limit of a {N}avier--{S}tokes system leading to a free/congested zones two-phase model.
\newblock {\em Comptes Rendus Mathematique}, 352(9):685--690, 2014.

\bibitem{Burtea_2020}
C.~Burtea and B.~Haspot.
\newblock New effective pressure and existence of global strong solution for compressible navier–stokes equations with general viscosity coefficient in one dimension.
\newblock {\em Nonlinearity}, 33(5):2077, 2020.

\bibitem{chaudhuri_analysis_2022}
N.~Chaudhuri, P.~Gwiazda, and E.~Zatorska.
\newblock Analysis of the generalised {A}w-{R}ascle model.
\newblock {\em Commun Partial Differ Equ.}, 48(3):440--477, 2023.

\bibitem{duality2024HCL}
N.~Chaudhuri, M.~A. Mehmood, C.~Perrin, and E.~Zatorska.
\newblock Duality solutions to the hard-congestion model for the dissipative {A}w-{R}ascle system.
\newblock {\em arXiv preprint 2402.08295}, 2024.

\bibitem{HCL}
N.~Chaudhuri, L.~Navoret, C.~Perrin, and E.~Zatorska.
\newblock Hard congestion limit of the dissipative {A}w-{R}ascle system.
\newblock {\em Nonlinearity}, 37(4):045018, 2024.

\bibitem{Constantin_2020}
P.~Constantin, T.~D. Drivas, H.~Q. Nguyen, and F.~Pasqualotto.
\newblock Compressible fluids and active potentials.
\newblock {\em Ann. Inst. H. Poincar\'{e} C Anal. Non Lin\'{e}aire}, 37(1):145--180, 2020.

\bibitem{dalibard2019existence}
A.-L. Dalibard and C.~Perrin.
\newblock Existence and stability of partially congested propagation fronts in a one-dimensional {N}avier-{S}tokes model.
\newblock {\em Communications in Mathematical Sciences}, 2020.

\bibitem{dalibard2021local}
A.-L. Dalibard and C.~Perrin.
\newblock Local and global well-posedness of one-dimensional free-congested equations.
\newblock {\em Annales Henri Lebesgue}, 2024.

\bibitem{locallycosntrainedAR}
M.~Garavello and P.~Goatin.
\newblock The {A}w-{R}ascle traffic model with locally constrained flow.
\newblock {\em J. Math. Anal. Appl.}, 378(2):634--648, 2011.

\bibitem{guazzelli2018rheology}
{\'E}.~Guazzelli and O.~Pouliquen.
\newblock Rheology of dense granular suspensions.
\newblock {\em Journal of Fluid Mechanics}, 852:P1, 2018.

\bibitem{resonanceAr}
Y.~Hu and G.~Wang.
\newblock The resonance behavior for the coupling of two {A}w-{R}ascle traffic models.
\newblock {\em Adv. Differ. Equ.}, 2020(1), 2020.

\bibitem{lefebvre_maury}
A.~Lefebvre-Lepot and B.~Maury.
\newblock Micro-macro modelling of an array of spheres interacting through lubrication forces.
\newblock {\em Adv. Math. Sci. Appl.}, 21(2):535--557, 2011.

\bibitem{HCLMehmood}
M.~A. Mehmood.
\newblock Hard congestion limit of the dissipative {A}w-{R}ascle system with a polynomial offset function.
\newblock {\em J. Math. Anal. Appl.}, 533(1):128028, 2024.

\bibitem{renardy2006introduction}
M.~Renardy and R.~C. Rogers.
\newblock {\em An introduction to partial differential equations}, volume~13.
\newblock Springer Science \& Business Media, 2006.

\bibitem{rudin1991functional}
W.~Rudin.
\newblock {\em Functional Analysis}.
\newblock International series in pure and applied mathematics. McGraw-Hill, 1991.

\bibitem{sheng_concentration_2022}
S.~Sheng and Z.~Shao.
\newblock Concentration in vanishing adiabatic exponent limit of solutions to the {Aw}-{Rascle} traffic model.
\newblock {\em Asymptot. Anal.}, 129(2):179--213, 2022.

\bibitem{vasseur2016nonlinear}
A.~F. Vasseur and L.~Yao.
\newblock Nonlinear stability of viscous shock wave to one-dimensional compressible isentropic {N}avier--{S}tokes equations with density dependent viscous coefficient.
\newblock {\em Communications in Mathematical Sciences}, 14(8):2215--2228, 2016.

\end{thebibliography}

\end{document}